\documentclass[12pt]{amsart}

\usepackage{amssymb,latexsym}

\usepackage{enumerate}

 \usepackage[french,english]{babel}
 
 \usepackage{xcolor}

\usepackage[inner=1.0in,outer=1.0in,bottom=1.0in, top=1.0in]{geometry}

\allowdisplaybreaks

\makeatletter

\@namedef{subjclassname@2010}{  \textup{2010} Mathematics Subject Classification}

\makeatother
\newtheorem{thm}{Theorem}[section]

\newtheorem{lem}[thm]{Lemma}
\newtheorem{pro}[thm]{Proposition}
\theoremstyle{definition}

\newtheorem*{rem}{Remark}

\numberwithin{equation}{section}

\newcommand{\X}{\mathbb{X}}

\newcommand{\ex}{\mathbb{E}}
\newcommand{\E}{\mathcal{E}}
\newcommand{\re}{\textup{Re}}
\newcommand{\im}{\textup{Im}}
\newcommand{\Psr}{\Psi_{\textup{rand}}}
\newcommand{\pr}{\mathbb{P}}

\newcommand{\B}{\mathcal B}

\newcommand{\M}{\mathcal{M}}

\newcommand{\Lv}{\mathbf{L}}

\newcommand{\ub}{\mathbf{u}}
\newcommand{\vb}{\mathbf{v}}
\newcommand{\xb}{\mathbf{x}}
\newcommand{\yb}{\mathbf{y}}
\newcommand{\zb}{\mathbf{z}}
\newcommand{\kb}{\mathbf{k}}
\newcommand{\lb}{\mathbf{l}}

\newcommand{\me}{\textup{meas}}

\newcommand\be{\begin{equation}}
\newcommand\ee{\end{equation}}

\newcommand{\sumstar}{\sideset{}{^*}\sum}

\def\u{\mathbf u}
\def\v{\mathbf v}
 
 \def\D{\mathbf D}

\def\E{\mathbb E}
\def\L{\mathcal L}
\def\P{\mathbb P}

 \def\Phrand{ \Phi_T^{\mathrm{rand}}}

 \def\Phrq{ \Phi_{q,  T}^{\mathrm{rand}}}
 \def\Phrqtail{ \Phi_{>q,T}^{\mathrm{rand}}}

 \def\Phrhat{\widehat{\Phi}_T^{\mathrm{rand}}}

 \def\Phrqhat{\widehat{\Phi}_{q, T}^{\mathrm{rand}}}
 \def\Phrqtailhat{\widehat{\Phi}_{>q,  T}^{\mathrm{rand}}}

 \def\Phhat{\widehat{\Phi}_T}

\newcommand{\newabstract}[1]{%
  \par\bigskip
  \csname otherlanguage*\endcsname{#1}%
  \csname captions#1\endcsname
  \item[\hskip\labelsep\scshape\abstractname.]
}

\begin{document}

\baselineskip=17pt

\title[Zeros of linear combinations of $L$-functions near the critical line]{The number of zeros of linear combinations of $L$-functions near the critical line}

\author{Youness Lamzouri}
\address{Institut \'Elie Cartan de Lorraine, Universit\'e de Lorraine, BP 70239, 54506 Vandoeuvre-l\`es-Nancy Cedex, France}

\email{youness.lamzouri@univ-lorraine.fr}

\author[Yoonbok Lee]{Yoonbok Lee$^{\dagger}$}
\thanks{$^{\dagger}$corresponding author.}
\address{Department of Mathematics \\ Research Institute of Basic Sciences \\ Incheon National University \\ 119 Academy-ro, Yeonsu-gu, Incheon, 22012 \\ Korea}
\email{leeyb@inu.ac.kr, leeyb131@gmail.com}

\date{\today}

\begin{abstract} 
In this paper, we investigate the zeros near the critical line of linear combinations of $L$-functions belonging to a large class, which conjecturally contains all $L$-functions arising from automorphic representations on $\textup{GL}(n)$. More precisely, if $L_1, \dots, L_J$ are distinct primitive $L$-functions with $J\ge 2$, and $b_j$ are any nonzero real numbers, we prove that the number of zeros of $F(s)=\sum_{j=1}^J b_j L_j(s)$ in the region $\re(s)\geq 1/2+1/G(T)$ and $\im(s)\in [T, 2T]$ is asymptotic to $K_0 T G(T)/\sqrt{\log G(T)}$ uniformly in the range $ \log \log T \leq  G(T)\leq (\log T)^{\nu}$, where $K_0$ is a certain positive constant that depends on $J$ and the $L_j$. This establishes a generalization of a conjecture of Hejhal in this range. Moreover, the exponent $\nu$ verifies $\nu\asymp 1/J$ as $J$ grows.

\end{abstract}



\thanks{This work was supported by Incheon National University (International Cooperative) Research Grant in 2019. }

\maketitle

\section{Introduction}

The theory of $L$-functions has become a central part of modern number theory, due to its connection to various arithmetic, geometric and algebraic objects. $L$-functions are  represented by Dirichlet series which are absolutely convergent in half-planes. They satisfy certain conditions, including having a meromorphic continuation, an Euler product over primes and a functional equation. The prototypical example of an $L$-function is the Riemann zeta function. Other important examples include Dirichlet $L$-functions attached to primitive Dirichlet characters, and the Hasse-Weil $L$-functions attached to elliptic curves. The Langlands program predicts that all $L$-functions arise from automorphic representations over $\textup{GL}(n)$.

 $L$-functions are predicted to verify several hypotheses, the most important of which is the generalized Riemann hypothesis (GRH), which asserts that all non-trivial zeros of $L$-functions lie on the critical line $\re(s)=1/2$. On the other hand, there exist various Dirichlet series which have arithmetical significance and satisfy a functional equation, but are not $L$-functions, since they do not possess an Euler product. Most of these functions can be expressed as linear combinations of $L$-functions.  Important examples include Epstein zeta functions associated to quadratic forms, and the zeta function attached to ideal classes in number fields.  Unlike $L$-functions, these zeta functions are not expected to satisfy the   GRH, and some of them might even possess zeros inside the region of absolute convergence. The first to have investigated such a phenomenon are Davenport and Heillbronn \cite{DH}, who proved in 1936 that the Epstein zeta function of a positive definite quadratic form of class number $\ge 2$ has infinitely many zeros in the half-plane of absolute convergence $\re(s)>1$.

 For a complex valued function $f(z)$, we shall denote by $N_f(\sigma_1, \sigma_2, T)$ the number of zeros of $f$ in the rectangle $\sigma_1\leq \re(s) \leq \sigma_2$ and $T\leq \im(s)\leq 2T$. We also let $N_f(\sigma, T)$ be the number of zeros of $f$ in the region  $\re(s) \geq \sigma$ and $T\leq \im(s)\leq 2T$. Voronin \cite{V} proved that $N_E(\sigma_1, \sigma_2, T)\gg T$ for any $1/2<\sigma_1<\sigma_2<1$ fixed, where $E(s, Q)$ is the Epstein zeta function attached to a binary quadratic form $Q$ with integral coefficients and with class number at least $2$. Lee \cite{Le2} improved this result to an asymptotic formula $N_E(\sigma_1, \sigma_2, T)\sim c(\sigma_1, \sigma_2) T$ for some positive constant $c(\sigma_1, \sigma_2)$. Gonek and Lee \cite{GL} obtained a quantitative bound for the error term in this asymptotic formula, and this was subsequently improved by Lamzouri \cite{La} who showed that one can obtain a saving of a power of $\log T$ in the error term. 
 
 Throughout this paper we let $J\geq2$ be an integer, $b_1, \dots, b_J$ be non-zero real numbers such that $\sum_{j=1}^J b_j^2 = 1$, and we define
\begin{equation}\label{DefLinearComb}
 F(s):=F_{L_1, \dots, L_J}(s)= \sum_{j=1}^J b_j L_j(s),
 \end{equation}
for $L$-functions $L_1, \dots, L_J$. 
Lee, Nakamura and Pa\'nkowski \cite{LNP} generalized Voronin's result to zeros of linear combinations of $L$-functions in the strip $1/2<\sigma_1<\sigma_2<1$, where $\sigma_1, \sigma_2$ are fixed. More precisely, they established that
$$N_F(\sigma_1, \sigma_2, T)\gg T,$$ 
if the $L_j$ belong to the Selberg class of $L$-functions, and verify a stronger version of the Selberg orthogonality conjecture (see \eqref{SOC} below). In the special case where  the $L_j$ are Dirichlet $L$-functions or Hecke $L$-functions attached to the ideal class characters of a quadratic imaginary field, one obtains an asymptotic formula $N_F(\sigma_1, \sigma_2, T) \sim c_F(\sigma_1, \sigma_2)T$ as $T\to \infty$, for a certain positive constant  $c_F(\sigma_1,\sigma_2)$ (See \cite{BoJe}, \cite{KL} and \cite{Le2}).
Furthermore, by using the methods of the proof of Theorem \ref{Main} below, one can generalize this result by showing that 
$$N_F(\sigma_1, \sigma_2, T) = c_F(\sigma_1, \sigma_2) T + O\left(\frac{T}{(\log T)^{\delta}}\right), $$
if the $L$-functions $L_1, \dots, L_J$ satisfy the assumptions A1--A5 below and $\delta= \delta ( J, F, \sigma_1, \sigma_2 )$ is a positive constant.

Although linear combinations of $L$-functions have many zeros off the critical line, it was conjectured by Montgomery that $100\%$ of the zeros of $F(s)$ lie on the critical line, if the $L_j$ are primitive\footnote{An $L$-function is primitive if it cannot be written as a product of non-trivial $L$-functions, where the trivial $L$-function is the constant $1$.} $L$-functions satisfying assumptions A1 and A2 below. Bombieri and Hejhal \cite{BH} established this conjecture if  the $L_j$ satisfy the assumptions A1, A2, A3 and A5 below,  
conditionally on the GRH and a zero-spacing hypothesis for each of the $L_j $. Unconditionally, Selberg \cite{Se} established that a positive proportion of the zeros of $F(s)$ lie on the critical line, in the special case where all of the $L_j$ are Dirichlet $L$-functions having the same parity and conductor.

In this paper we study the zeros of the linear combination $F(s)$ where the $L$-functions $L_1 , \ldots, L_J $ satisfy the following assumptions:

\begin{enumerate}
\item[A1:] (Euler product)  For $ j = 1, \ldots , J $ and $\re(s)>1$ we have 
$$ L_j  ( s) = \prod_p \prod_{i=1}^d \bigg(   1 -   \frac{ \alpha_{j,i}(p)}{p^s} \bigg)^{-1} $$
where $ | \alpha_{j,i} (p)  | \leq p^{\theta}$ for some fixed $ 0 \leq \theta < 1/2$ and for every $ i = 1, \ldots , d.$

\item[A2:] (Functional equation) The functions $L_1, L_2, \dots, L_J$ satisfy the same functional equation
 $$ \Lambda_j(s) = \omega \overline{ \Lambda_j( 1- \bar{s})} ,$$ 
where
$$ \Lambda_j(s) := L_j(s) Q^s \prod_{\ell=1}^k \Gamma ( \lambda_{\ell} s+\mu_{\ell} ) , $$  
$ | \omega| =1 $, $Q>0$, $ \lambda_{\ell}>0 $ and $\mu_{\ell} \in \mathbb{C} $ with $ \re ( \mu_{\ell} ) \geq 0 $.
\item[A3:] (Ramanujan hypothesis on average)
$$ \sum_{ p \leq x } \sum_{i=1}^d | \alpha_{j,i }(p) |^2 = O( x^{1+\epsilon})$$
holds  for every $ \epsilon>0$ and for every $ j = 1, \ldots , J $ as $ x \to \infty $.

\item[A4:] (Zero density hypothesis) There exist positive constants $c_1, c_2$ such that for all $1\le j\le J$ and all $\sigma\geq 1/2$ we have 
\be\label{assumption zero}
 N_{L_j } ( \sigma, T ) \ll T^{1 - c_1(\sigma - 1/2) } (\log T)^{c_2}.
 \ee

\item[A5:] (Selberg orthogonality conjecture) By assumption A1 we can write 
 $$ \log L_j(s) = \sum_p \sum_{k=1}^\infty   \frac{ \beta_{L_j} (p^k)}{ p^{ks}}  , $$
where $\beta_{L_j} (p^k )$ are complex numbers.
Then, for all $1\leq j, k \leq J$ there exist constants $ \xi_j >0$ and $ c_{j,k}$ such that
 \be\label{SOC}
 \sum_{p \leq x} \frac{ \beta_{L_j}(p)  \overline{\beta_{L_k}(p) } }{p} = \delta_{j,k} \xi_j \log \log x + c_{j,k} + O \bigg( \frac{1}{ \log x} \bigg),
 \ee
  where $ \delta_{j,k} = 0 $ if $ j \neq k $ and $ \delta_{j,k} = 1 $ if $ j = k$. 
 
\end{enumerate}

\begin{rem}
The assumptions A1-A5 are standard, and are expected to hold for all $L$-functions arising from automorphic representations on GL$(n)$. In particular, they are verified by GL$(1)$ and GL$(2)$ $L$-functions, which are the Riemann zeta function and Dirichlet $L$-functions, and $L$-functions attached to Hecke holomorphic or Maass cusps forms. For GL$(1)$ $L$-functions, the Selberg orthogonality conjecture boils down to the fact that $L(s, \chi)$ is regular and non-zero at $s=1$, if $\chi$ is a non-principal Dirichlet character. For GL(2) $L$-functions, assumptions A3 and A5 are handled using the Ranking-Selberg convolution, while assumption A4 is proved in \cite{Lu} for $L$-functions attached to holomorphic cusp forms, and in \cite{SaSe} for $L$-functions attached to Maass forms. Assumptions A4 and A5 are used to investigate the joint distribution of $\log L_1, \log L_2, \dots, \log L_J$ near the critical line, which is a key component of the proof of Theorem \ref{Main} below. To this end, the zero density hypothesis A4 is used to approximate $\log L_j(\sigma+it)$ by short Dirichlet polynomials for ``almost all'' points $t$, while the Selberg orthogonality conjecture A5 insures the ``statistical independence'' of the functions $\log L_j(\sigma+it)$, when $\sigma$ is very close to $1/2$.
\end{rem}

For each $1\leq j\leq J$, write $L_j$ as a Dirichlet series $L_j (s) = \sum_{n=1}^\infty  \frac{ \alpha_{L_j }(n)}{n^s }$, which is absolutely convergent for $ \re (s) > 1$ by assumption A3. Then $F(s)$ has a Dirichlet series representation 
$$ \sum_{n=1}^\infty \frac{\alpha_F (n)}{n^s} = \sum_{n=1}^\infty   \frac{ \sum_{j=1}^J b_j \alpha_{L_j }(n)}{n^s }.$$
Since the first nonzero term dominates the others, $ F(s) $ has no zeros in $\re(s) > A$ for some constant $A>0$. Hence, we may define 
$$\sigma_F := \sup \{ \re(\rho)  : F( \rho ) = 0 \} \leq A .$$
Moreover, it follows from assumption A2 that $F(s)$ satisfy the functional equation 
\be\label{FE3}
 F(s) Q^s \prod_{\ell=1}^k \Gamma ( \lambda_\ell s + \mu_\ell )  = \omega \overline{F(1-\bar{s})}  Q^{1-s} \prod_{\ell=1}^k \Gamma (\lambda_\ell (1-s) + \bar{\mu}_\ell ) .
 \ee
Furthermore, since $F(s)$ has no zeros on $\re(s) > \sigma_F $, by \eqref{FE3} we see that 
$$ \left\{  - \frac{\mu_\ell +m}{\lambda_\ell}  :   - \frac{\re(\mu_\ell )+m}{\lambda_\ell} < 1- \sigma_F ,  \ell = 1, \ldots , k \mathrm{~and~} m=0, 1, 2, \ldots \right\} $$
is the set of (trivial) zeros of $F(s) $ on $\re(s) < 1- \sigma_F $. All the other zeros are in the strip $ 1- \sigma_F \leq \re(s) \leq \sigma_F $ and we may call them the nontrivial zeros. The number of nontrivial zeros $ \beta+i\gamma $ of $F(s)$ with $ 0 < \gamma < T$ is denoted by $N_F (T)$ and  it is well-known that  
$$ N_F(T) \sim \frac{d_F}{ 2 \pi} T \log T $$
as $ T \to \infty$, where  $d_F=2 \sum_{\ell=1}^k\lambda_{\ell}$.

 Bombieri and Hejhal \cite{BH} conjectured that the order of magnitude of the number of zeros of $F(s)$ off the critical line and up to height $T$ should be 
 $$ \frac{T\log T}{\sqrt{\log\log T}}.$$
 Motivated by this conjecture, Hejhal \cite{He, He2} studied the zeros of linear combinations of $L$-functions near the critical line. Suppose that $L_1$ and $L_2$ satisfy assumptions A1, A2, A4, and A5, as well as the Ramanujan-Petersson conjecture (which asserts that $|\alpha_{j, i}(p)|\leq 1$ for all $i, j$ and $p$) instead of the weaker assumption A3. Let  $F(s)= \cos(\alpha)  L_1(s)+ \sin(\alpha)  L_2(s) $, where  $\alpha$ is a real number. Then Hejhal \cite{He} proved that for ``almost all'' $\alpha$ with respect to a certain measure, we have 
\begin{equation}\label{BoundsHejhal}
\frac{T G(T)}{ \sqrt{ \log \log T}}    \ll    N_{F} \bigg( \frac12 + \frac{1}{G(T)} , T \bigg)  \ll    \frac{T G(T)}{ \sqrt{ \log \log T}} 
\end{equation}
in the range $(\log T)^\delta  \leq  G(T)  \leq \frac{ \log T}{ ( \log \log T)^\kappa} $ where $\delta > 0 $ and $ 1 < \kappa < 3 $ are fixed. He also conjectured (see Section 6 of \cite{He}) that the following asymptotic formula should hold for all  $\alpha \notin \frac{\pi}{2}\mathbb{Z}$ in the same range of $G(T)$
\begin{equation}\label{HejhalConjecture}
N_{F} \bigg( \frac12 + \frac{1}{G(T)} , T \bigg) \sim \frac{\sqrt{\xi_1+\xi_2}}{8\pi^{3/2}} \frac{T G(T)}{ \sqrt{ \log G(T)}}.
\end{equation}
In the short note \cite{He2}, Hejhal discussed a generalization of the bounds \eqref{BoundsHejhal} to linear combinations with three or more $L$-functions, but did not provide a complete proof of these bounds. 

In this paper, we prove a quantitative generalization of the conjectured asymptotic formula \eqref{HejhalConjecture}  for any linear combination of $L$-functions $F(s)= \sum_{j=1}^J b_j L_j (s)$ as in \eqref{DefLinearComb}, though in a smaller range of the parameter $G(T)$. More precisely, our main result is the following theorem.

\begin{thm}\label{Main}  Let $F(s)$ be defined by \eqref{DefLinearComb}, where the $L$-functions $L_1 , \ldots, L_J $ satisfy assumptions A1--A5. Let
 $T$ be large,  $\xi=\max_{j\le J} \xi_j$, and 
$$
0<\nu <1/(12 J+ 7 + 16 J \sqrt{3\xi})
$$ 
be a fixed real number. Then for $\sigma= 1/2+1/G(T)$ with $  \log\log T  \leq  G(T)\leq (\log T)^{\nu}$ we have 
$$
N_F ( \sigma, T) =  K_0\frac{ T G(T) }{  \sqrt{ \log G(T)}   } + O \bigg(  \frac{  T G(T)  }{ ( \log G(T))^{5/4} } \bigg),
$$
where 
$$K_0= K_0(J; \xi_1, \xi_2, \dots, \xi_J):= \frac{1  }{ 4 \pi^{J/2+1} \prod_{j=1}^J \sqrt{\xi_j}    } \sum_{n=1}^J  \int_{\mathcal{R}_n }     e^{ - \sum_{j=1}^J u_j^2 / \xi_j }   u_n   d\ub$$
and 
$$\mathcal{R}_n := \{ \ub \in \mathbb{R}^J : u_n = \max \{ u_1 , \ldots, u_J \} \}.$$
\end{thm} 

\begin{rem}
In the case $J=2$, an easy calculation shows that $K_0$ matches the constant predicted by Conjecture \eqref{HejhalConjecture} of Hejhal. Indeed, we have in this case 
$$K_0= \frac{1  }{ 4 \pi^{2}  \sqrt{\xi_1 \xi_2} } \left(\int_{u_1\geq u_2}     e^{ - u_1^2 / \xi_1-  u_2^2 / \xi_2}   u_1   d u_1 d u_2+\int_{u_2\geq u_1}     e^{ - u_1^2 / \xi_1-  u_2^2 / \xi_2}   u_2   d u_1 d u_2 \right).$$
 We can compute the first integral as
$$
 \int_{u_1\geq u_2}     e^{ - u_1^2 / \xi_1-  u_2^2 / \xi_2}   u_1   d u_1 d u_2= \frac{\xi_1}{2}\int_{-\infty}^{\infty} e^{-u_2^2(1/\xi_1+1/\xi_2)}du_2= \frac{\sqrt{\pi}\xi_1}{2} \sqrt{\frac{\xi_1\xi_2}{\xi_1+\xi_2}}.
 $$
Evaluating the second integral similarly, we thus deduce that 
$$ 
K_0=\frac{1  }{ 4 \pi^{2}  \sqrt{\xi_1 \xi_2} }\left( \frac{\sqrt{\pi}\xi_1}{2} \sqrt{\frac{\xi_1\xi_2}{\xi_1+\xi_2}}+  \frac{\sqrt{\pi}\xi_2}{2} \sqrt{\frac{\xi_1\xi_2}{\xi_1+\xi_2}}\right)= \frac{\sqrt{\xi_1+\xi_2}}{8\pi^{3/2}},
$$
as desired.
\end{rem} 
\begin{rem} In \cite{Le4}, Lee obtained an analogue of Theorem \ref{Main} (in a larger range of $G(T)$) in the case where $F(s)=E(s, Q)$ is the Epstein zeta function attached to a binary quadratic form $Q$ with integral coefficients and  class number  $2$ or $3$. In this case $E(s, Q)$ can be expressed as the linear combination of two Hecke $L$-functions. However, the method of \cite{Le4} does not seem to generalize to the case of linear combinations of three or more $L$-functions. 

\end{rem}


\section{Strategy of the proof of Theorem \ref{Main}, key ingredients and detailed results} 
Let $F(s)$ be defined by \eqref{DefLinearComb} where the $L_j $ satisfy assumptions A1--A5. In order to count the number of zeros of $F(s)$ in the region $\re(s)>\sigma$, $T\leq \im(s)\leq 2T$ we shall use Littlewood's lemma in a standard way. Let $\sigma_0>\sigma_F$. Then, $F(s)$ has no zeros in $\re(s)\geq \sigma_0$ and hence by Littlewood's lemma
(see (9.9.1) in \cite{Ti}), we have 
\begin{equation}\label{Littlewood}
\begin{aligned}
\int_{\sigma}^{\sigma_0}N_F(u, T) du
& = \frac{1}{2\pi} \int_{T}^{2T} \log |F(\sigma+it)| dt - \frac{1}{2\pi} \int_{T}^{2T} \log |F(\sigma_0 + it)| dt\\
& \ \ +  \frac{T}{2 \pi} (\sigma- \sigma_0 ) \log n_0 +O_F(\log T),
\end{aligned}
\end{equation}
where $n_0$ is the smallest positive integer such that 
$   \alpha_{F} (n_0)  \neq 0 $.
In order to estimate the integrals on the right hand side of this asymptotic formula, we shall construct a probabilistic random model for $F(\sigma+it)$. Recall that 
$$ F(\sigma+it)= \sum_{j=1}^J b_j L_j(\sigma+it)=\sum_{j=1}^J b_j\prod_p \prod_{i=1}^d \bigg(1-\frac{ \alpha_{j,i}(p)}{p^{\sigma+it}} \bigg)^{-1}.
$$ 
Let $\{\X(p)\}_{p}$ be a sequence of independent random variables, indexed by the prime numbers, and uniformly distributed on the unit circle. For $1\leq j\leq J$ we consider the random Euler products 
$$
L_j(\sigma, \X):=\prod_p \prod_{i=1}^d \bigg(1-\frac{\alpha_{j,i}(p) \X(p)}{p^{\sigma}} \bigg)^{-1}.
$$
These products converge almost surely for $\sigma>1/2$ by Kolmogorov's three series theorem. We shall prove that the integral $\frac{1}{T} \int_{T}^{2T}\log |F(\sigma+it)|dt$ is very close to the expectation 
 of $\log |F(\sigma, \X)|$, where the probabilistic random model $F(\sigma, \X)$ is defined by 
$$ F(\sigma, \X):= \sum_{j=1}^J b_j L_j(\sigma, \X).$$
\begin{thm}\label{AsympLogComb} Let $J\ge 2$ be an integer, $\xi=\max_{j\le J} \xi_j$, and 
$$
0<\nu <1/(12 J+ 7 + 16 J \sqrt{3\xi})
$$ 
be a fixed real number. Let
 $T$ be large. There exists a positive constant $\beta>0$ such that for $\sigma= 1/2+1/G(T)$ with $ \log \log T \leq G(T)\leq (\log T)^{\nu}$ we have 
$$ 
\frac{1}{T} \int_{T}^{2T}\log \big|F(\sigma+it)\big|dt = \ex\left(\log \big|F(\sigma, \X)\big|\right)+ O\left(\frac{1}{(\log T)^{\beta}}\right),
$$
where here and throughout we denote by $\ex(\cdot)$ the expectation.
\end{thm}
Thus, in order to estimate $N_F(\sigma, T)$ it remains to investigate the function 
$$ \M(\sigma):=\ex\left(\log |F(\sigma, \X)|\right),$$
and more precisely to estimate the difference $\M(\sigma)- \M(\sigma+h)$ for small $h$. We shall investigate this quantity in Section \ref{sec:rand} and prove the following result.
\begin{thm} \label{MainRandom}
Let $G_1(T) := G(T) \frac{ \log G(T)}{ \log G(T) - 1 }$ and $G_2(T) := G(T) \frac{ \log G(T)}{ \log G(T) + 1}$. 
Assume that $G(T) \geq 4 $. 
Then   for each $i=1, 2$   we have
$$
\M\left(\frac{1}{2}+\frac{1}{G(T)}\right)-  \M\left(\frac{1}{2}+\frac{1}{G_i(T)}\right)= (-1)^i \frac{2 \pi K_0}{(\log G(T))^{3/2}}  +   O \bigg(\frac{1}{(\log G(T))^{9/4}}\bigg),
$$
 where the constant $K_0$ is defined in Theorem \ref{Main}.
     \end{thm}

We now show how to deduce Theorem \ref{Main} from Theorems \ref{AsympLogComb} and \ref{MainRandom}.  
\begin{proof}[Proof of Theorem \ref{Main} assuming Theorems \ref{AsympLogComb} and \ref{MainRandom}]
First, note that 
$ \frac12 + \frac{1}{G_i(T)} = \sigma+ (-1)^i /(G(T)\log G(T)).$ 
Since $N_F (w,T)$ is a decreasing function of $w$ for each $T$, we see that
\begin{equation}\label{IntegralLittlewood}
\int_{\sigma}^{ \frac12 + \frac{1}{G_2 (T)}} N_F (w,T) dw \leq \frac{N_F (\sigma,T)}{ G(T) \log G(T) } \leq  \int_{\frac12 + \frac{1}{G_1(T)}}^{\sigma} N_F (w,T)dw. 
\end{equation}
By \eqref{Littlewood} and Theorems \ref{AsympLogComb} and \ref{MainRandom}, we obtain
\begin{align*}
   \frac{2\pi}{T} &  \int_{\sigma}^{ \frac12 + \frac{1}{G_i (T)}} N_F (w, T)dw   \\
  & =  \frac1T \int_T^{2T} \log|F(\sigma +it)|  -  \log \bigg| F \bigg(\frac12 + \frac{1}{G_i(T)} +it  \bigg)\bigg| dt + O\left(\frac{1}{G(T) \log G(T) }\right) \\
  & = \M(\sigma)- \M\left(\frac{1}{2}+ \frac{1}{G_{i}(T)}\right)+  O\left(\frac{1}{G(T) \log G(T) }\right) \\
  & = (-1)^i \frac{2 \pi K_0}{(\log G(T))^{3/2}}  +   O \bigg(\frac{1}{(\log G(T))^{9/4}}\bigg).
  \end{align*}
Inserting these estimates in \eqref{IntegralLittlewood} completes the proof. 
\end{proof}

We next describe the different ingredients that are used in the proof of Theorem \ref{AsympLogComb}. The first is a discrepancy bound for the joint distribution of the values of the $L$-functions $L_j(s)$, which generalizes the results of 
\cite{LLR} for the Riemann zeta function. For $\sigma>1/2$ we let $$ \Lv(\sigma+it)=\Big(\log |L_1(\sigma+it)|, \dots, \log |L_J(\sigma+it)|, \arg L_1(\sigma+it), \dots, \arg L_J(\sigma+it) \Big),$$
and similarly define the random vector
$$ \Lv(\sigma, \X)=\Big(\log |L_1(\sigma, \X)|, \dots, \log |L_J(\sigma, \X)|, \arg L_1(\sigma, \X), \dots, \arg L_J(\sigma, \X) \Big).$$
For  a Borel set $ \B$ in $\mathbb R^{2J}$ and for~$\sigma = 1/2+ 1/G(T) $, we define
\begin{equation}\label{PhiT def}
\Phi_T (\B ):= \frac1T \mathrm{meas}  \{ t\in[T, 2T] :  \Lv (\sigma +it )\in \B  \}  
\end{equation}
and 
 \be\label{Phi rand def}
\Phrand  (\B ):=
 \mathbb{P}(\Lv(\sigma , \X) \in \B),
 \ee
 where here and throughout, meas will denote the Lebesgue measure on $\mathbb{R}$. 
 We will prove that the measure $ \Phrand $ is absolutely continuous and investigate its density function $H_T  (\ub, \vb)$ in Section \ref{sec:rand}.
 
We define the \emph{discrepancy} between these two distributions as
$$
 \D_T(\B):= \Phi_T (\B  )- \Phrand (\B ) .
$$ 
Then we prove the following result which generalizes Theorem 1.1 of \cite{LLR}, and might be of independent interest.  
\begin{thm}\label{Discrepancy}
Let $T$ be large and $\sigma=1/2+ 1/G(T)$ where $  \log\log T\leq G(T)\leq \sqrt{\log T}/\log\log T $. Then we have 
 $$
 \sup_{\mathcal{R}} \big|\D_T(\mathcal R)\big| \ll \frac{\sqrt{G(T)}\log\log T}{\sqrt{\log T}},
$$
where $\mathcal R$ runs over all rectangular boxes of $\mathbb{R}^{2J}$ (possibly unbounded) with sides parallel to the coordinate axes. 
\end{thm}
We shall use this result to approximate the integral $\frac 1T\int_T^{2T} \log |F(\sigma+it)|dt$ by the expectation $\ex(\log| F(\sigma, \X)|)$. However, in doing so we need to control the large values and the logarithmic singularities of both $\log |F(\sigma+it)|$ and $\log|F(\sigma, \X)|$. To this end we prove the following propositions. 

 \begin{pro} \label{MomentsLogF}
  Let $T$ be large, and $\sigma=1/2+1/G(T)$ with $2<G(T)\leq c_0\sqrt{ \frac{\log T}{\log\log T} },$ for some small constant $c_0>0$. There exist positive constants $C_1, C_2>0$  such that for every positive integer $k \leq (\log T)/(C_1 G(T)\log\log T)$ we have

$$   
 \frac1T  \int_T^{2T} \big|\log  |F(\sigma+it)|\big|^{2k}  dt \ll (C_2 k \log \log T)^k G(T)^{3k+2} ( \max\{ k, G(T)^{3/2} \log G(T) \} )^{2k}.
$$
\end{pro}

 \begin{pro} \label{MomentsLogR}
Let $T$ be large, and  $\sigma=1/2+1/G(T)$ with $2<G(T)\leq  \log T.$ There exists a constant $C_3>0$ such that for every integer $k \geq 1$ we have
$$  \E \left[ \big|\log  |F ( \sigma , \X)|\big|^{2k} \right]   \ll (\log \log T)^{J} (C_3 k(k +\log\log T))^k. 
$$
 \end{pro}
We also need the following result on the large deviations of $\log L_j(\sigma+it)$ and $\log L_j(\sigma, \X)$. 
\begin{lem}\label{LargeDevLj}
Let $1\leq j\leq J$, and $T$ be large. Let $\sigma=1/2+1/G(T)$ with $\log\log T  \leq G(T)\leq c_0 \sqrt{\log T/\log\log T}$ for some small constant $c_0>0$. Then, there exists a positive constant $c_1 $ such that for all $\sqrt{ \log G(T)} \leq \tau \leq (\log\log T)^2$ we have 
$$ 
\frac{1}{T} \textup{meas} \{t\in [T, 2T] : |\log L_j(\sigma+it)| \geq \tau \} \ll \frac{\tau}{\sqrt{\log G(T)}} \exp\left(-\frac{\tau^2}{\xi_j \log G(T)+ c_1 G_0 (T)  }\right),
$$
where $G_0 (T) = \max\{\sqrt{\log G(T)}, \log_3 T\}$. 
Furthermore, the same bound holds for 
$$\pr( | \log L_j  (\sigma, \X ) | > \tau), $$
 in the same range of $\tau$. 
\end{lem}

Our last ingredient in the proof of Theorem \ref{AsympLogComb} is the following lemma, which provides bounds for the probability of ``the concentration'' of the random variable $|F(\sigma, \X)|$.  
\begin{lem}\label{ConcentrationRandom}
Let $\varepsilon>0$ be small and $J\geq 1$ be fixed. Then for any real numbers $\sigma>1/2$, $R>0$ and $M>2\pi$ we have
$$ \pr\left(\Lv(\sigma, \X) \in [-M, M]^{2J}, \ \text{ and } \ R<|F(\sigma, \X)|<R+\varepsilon\right) \ll M^{2J-1} e^{2M} (R\varepsilon+ \varepsilon^2),
$$
where the implicit constant is absolute.
\end{lem}

The plan of the remaining part of the paper is as follows. With all the ingredients now in place, we shall first prove Theorem \ref{AsympLogComb} in Section \ref{sec:AsympLogComb}. In Section \ref{sec:Preliminaries} we gather together several preliminary results that will be used in subsequent sections, and prove Lemma \ref{LargeDevLj}. These will be used to bound the discrepancy of the joint distribution of $\log L_j(s)$ and prove Theorem \ref{Discrepancy} in Section \ref{sec:Discrepancy}. Then in Section \ref{sec:MomentsLogF}, we shall establish Proposition \ref{MomentsLogF}. 
Finally in Section \ref{sec:rand} we shall investigate the distribution of the random vector $\Lv(\sigma, \X)$ and establish Theorem \ref{MainRandom}, Proposition \ref{MomentsLogR}, and Lemma \ref{ConcentrationRandom}.



\section{Proof of Theorem \ref{AsympLogComb}}\label{sec:AsympLogComb}

In this section we establish Theorem \ref{AsympLogComb} using the ingredients listed in the previous section, namely Theorem \ref{Discrepancy}, Propositions \ref{MomentsLogF} and \ref{MomentsLogR}, and Lemmas 
\ref{LargeDevLj} and \ref{ConcentrationRandom}. 

We let $\sigma=1/2+1/G(T)$ where $\log \log T \leq  G(T)\leq (\log T)^{\theta}$, and $0<\theta\leq 1/2$ is a real number that we shall optimize later. We start by showing how to use  Lemma  \ref{LargeDevLj} and Proposition \ref{MomentsLogF} to control the large values and the logarithmic singularities of $\log|F( \sigma+it)|$. Let $\alpha>0$ be a positive constant to be chosen and define 
$$
\L:=\alpha\log\log T, \text{ and } M:=(G(T)\log \log T)^3. 
$$ 
 We define the following sets
$$ S_1(T):= \left\{ t\in [T, 2T] : \Lv(\sigma+it) \in (-\L, \L)^{2J}\right\},$$ 
$$S_2(T):= \left\{ t\in [T, 2T] : \log|F(\sigma+it)|>-M\right\},
$$
and 
$$ S_0(T):= S_1(T) \cap S_2(T).$$
Let $\xi := \xi_{\max} :=\max_{j\leq J} \xi_j $.  Then it follows from Lemma \ref{LargeDevLj} that
\begin{equation*}\label{MeasureS1}
\begin{aligned}
\me([T, 2T]\setminus S_1(T)) & \leq \sum_{j=1}^{J} \me \left\{ t\in [T, 2T] : |\log L_j(\sigma+it)| \ge \L \right\}\\
& \ll \frac{T {e^{O(\sqrt{\log\log T})}} 
}{(\log T)^{\alpha^2/(\xi\theta)}},
\end{aligned}
\end{equation*} 
On the other hand, using Proposition \ref{MomentsLogF} with  $k=\lfloor(\log \log T)^{5/4}\rfloor$ gives
\begin{equation*}\label{MeasureST} 
\begin{aligned}
 \me([T, 2T]\setminus S_2(T)) &\leq 
\frac{1}{M^{2k}} \int_{T}^{2T} \big| \log|F( \sigma +it)| \, \big|^{2k} dt \\
 &\ll T \exp\left(-(\log \log T)^{5/4}\right).
\end{aligned}
\end{equation*}
Therefore we deduce that
\begin{equation}\label{MeasureS0T}
\me([T, 2T]\setminus S_0(T))  \ll  \frac{T {e^{O(\sqrt{\log\log T})}} 
}{(\log T)^{\alpha^2/(\xi\theta)}}.
\end{equation}
Combining this bound with Proposition \ref{MomentsLogF}, and using  H\"older's inequality with $r= \lfloor \log\log T\rfloor$ we have
\begin{equation}\label{TrunIntLog}
\begin{aligned}
& \left|\int_{t\in [T, 2T]\setminus S_0(T)}  \log | F( \sigma +it)| dt \right| \\
& \leq \big(\me\{t\in [T, 2T]\setminus S_0(T)\}\big)^{1-1/2r}\left(\int_{T}^{2T}\big| \log |F( \sigma +it)| \, \big|^{2r}dt\right)^{1/2r}\\
& \ll \left(\frac{T  e^{O(\sqrt{\log\log T})} }{(\log T)^{\alpha^2/(\xi\theta)}}\right)^{1-1/2r} \Big(T  (\log T)^{(6r+2)\theta} (\log\log T)^{4r} \Big)^{1/2r} \\
& \ll \frac{T {e^{O(\sqrt{\log\log T})}} 
}{(\log T)^{\alpha^2/(\xi\theta)-3\theta}}.
\end{aligned}
\end{equation}

We now define for $\tau\in \mathbb{R}$
$$ \Psi_T(\tau):= \frac{1}{T} \me\left\{ t\in S_0(T): \log |F ( \sigma +it)|>\tau\right\}$$
and similarly
$$ 
\Psr(\tau):= \pr\Big( \X\in \mathcal{S}, \text { and }\log |F ( \sigma, \X)|>\tau\Big),
$$
where $\mathcal{S}$ is the event $\Lv(\sigma, \X)\in (-\L, \L)^{2J}$ and $\log|F ( \sigma, \X)|>-M$.
First we observe that for $t\in S_0(T)$ we have $$\log |F( \sigma +it)| \leq \log \bigg(\sum_{j=1}^J  |b_jL_j ( \sigma +it)|\bigg)\leq \log \bigg(
\sum_{j=1}^J  |b_j| e^{\L}\bigg)=\L +\log\bigg(\sum_{j=1}^J |b_j|\bigg),$$
and hence we have 
\begin{equation}\label{EndRange}
\Psi_T(\tau)= \Psr(\tau)=0,
\end{equation}
for $$\tau > \widetilde{\L}:= \L +\log\bigg(\sum_{j=1}^J |b_j|\bigg).$$
Using a geometric covering argument, we prove the following result which shows that $\Psi_T(\tau)$ is very close to $\Psr(\tau)$ uniformly in $\tau$. This will imply Theorem \ref{AsympLogComb}. 
\begin{pro}\label{CoveringResult}
Let $T$ be large. Then we have
$$ \sup_{\tau\leq \widetilde{\L}} \left| \Psi_T(\tau)- \Psr(\tau)\right| \ll \frac{(\log\log T)^{2J+1}}{(\log T)^{(1-\theta-16 J\alpha)/(4J+2)}}.$$
\end{pro}
\begin{proof} 
We let $0<\varepsilon= \varepsilon(T)\leq e^{-2\L}$ be a small parameter to be chosen. 
We shall consider three cases depending on the size of $\tau$. 

\smallskip

\noindent \textbf{Case 1}: $\tau\leq  -M.$ In this case, it follows from the definition of the set $S_0(T)$ together with \eqref{MeasureS0T} that
$$
  \Psi_T(\tau)= \frac{\me(S_0(T))}{T}=\frac{1}{T} \me\left\{ t\in [T, 2T] : \Lv(\sigma+it) \in (-\L, \L)^{2J}\right\} +O\left(\frac{1}{(\log T)^2}\right).
 $$ 
Similarly, using Proposition \ref{MomentsLogR} with $k=\lfloor \log \log T\rfloor$ we have
\begin{align*}
\pr\left(\log\big|F( \sigma, \X)\big|<-M\right) &\leq 
\frac{1}{M^{2k}} \ex\left( \big| \log|F( \sigma, \X)| \, \big|^{2k} \right) \ll \frac{1}{(\log T)^{2}}.
\end{align*}
Therefore, it follows from the definition of the event $\mathcal{S}$ that
 $$ \Psr(\tau)=  \pr\big(\Lv(\sigma, \X) \in (-\L, \L)^{2J}\big) +O\left(\frac{1}{(\log T)^2}\right). $$
 Hence Theorem \ref{Discrepancy} yields 
 \begin{equation}\label{DiscrepancyCase1}
 \Psi_T(\tau)= \Psr(\tau) + O\left(\frac{\log\log T}{(\log T)^{(1-\theta)/2}}\right).
 \end{equation}
 
 \smallskip
 
\noindent \textbf{Case 2}: $-M< \tau\leq \log(C_3\varepsilon)+ \L$, where $C_3$ is a suitably large constant. 
In this case we have 
\begin{equation}\label{SecondCase}
  \Psi_T(\tau)= \frac{1}{T} \me\big\{ t\in [T, 2T]: \Lv(\sigma+it) \in (-\L, \L)^{2J}\setminus \mathcal{U}_J(e^{\tau}, \L)\big\},
\end{equation}
 where $\mathcal{U}_J(y, \L)$ is the bounded subset of $\mathbb{R}^{2J}$ defined by
\begin{align*}
\mathcal{U}_J(y, \L)
:= \Big\{ (u_1, \dots, u_J, v_1, \dots, v_J) \in \mathbb{R}^{2J} : \  & |u_j|, |v_j| <\L \text{ for all } 1\le j\le J,\\
& \text{ and } \left|\sum_{j=1}^J b_j e^{u_j+iv_j}\right|\leq y \Big\}.
\end{align*}
 We cover $\mathcal{U}_J(e^{\tau}, \L)$ with $K(\tau)$ distinct
  hypercubes $\mathcal{B}_k(\tau)$  of the form $\prod_{j=1}^{2J} [z_j, z_j + \varepsilon )$  
 with non-empty intersection with $\mathcal{U}_J(e^{\tau}, \L)$. Note that 
$$ K(\tau)  \ll \left(\frac{\L}{\varepsilon}\right)^{2J}.
$$
Now, let $1\leq k\leq K(\tau)$ and $(u_1, \dots, u_J, v_1, \dots, v_J) \in \mathcal{B}_k(\tau) \cap \mathcal{U}_J(e^{\tau}, \L)$. Recall that this intersection is non-empty by construction. Then, for any $(x_1, \dots, x_J, y_1, \dots, y_J)\in \mathcal{B}_k(\tau)$ we have $|x_j-u_j|\leq \varepsilon$ and 
$|y_j-v_j|\leq \varepsilon$ for all $1\leq j\leq J$. Hence, we deduce that $|x_j|, |y_j|<\L+\varepsilon$ for all $1\leq j\leq J$ and 
$$ \left|\sum_{j=1}^J b_j e^{x_j+iy_j}\right|=\left|\sum_{j=1}^J b_j e^{u_j+iv_j}\right| +O(\varepsilon e^{\L}) \leq C_4\varepsilon e^{\L}$$
for some positive constant $C_4$ since $e^{\tau} \leq C_3\varepsilon e^{\L}$ by our assumption. Therefore, we have shown that
$$ \mathcal{U}_J(e^{\tau}, \L) \subset  \bigcup_{k \le K(\tau)} \mathcal{B}_k(\tau)   \subset \mathcal{U}_J(C_4\varepsilon e^{\L}, \L+\varepsilon).$$
Hence, appealing to Theorem \ref{Discrepancy} we obtain
\begin{equation}\label{BoundMesExceptional}
\begin{aligned}
& \frac{1}{T} \me\big\{ t\in [T, 2T]: \Lv(\sigma+it) \in \mathcal{U}_J(e^{\tau}, \L)\big\}\\
& \leq 
\sum_{k=1}^{K(\tau)} \frac{1}{T} \me\big\{ t\in [T, 2T]: \Lv(\sigma+it) \in \mathcal{B}_k(\tau)\big\}\\
&= \sum_{k=1}^{K(\tau)}\pr\left(\Lv(\sigma, \X)\in \mathcal{B}_k(\tau)\right) + O\left(\frac{K(\tau)\log\log T}{(\log T)^{(1-\theta)/2}}\right)\\
&\leq \pr\left(\Lv(\sigma, \X)\in \mathcal{U}_J(C_4\varepsilon e^{\L}, \L+\varepsilon)\right) + O\left(\frac{\L^{2J+1} }{\varepsilon^{2J}(\log T)^{(1-\theta)/2}}\right).
\end{aligned}
\end{equation}
Moreover, it follows from Lemma \ref{ConcentrationRandom} that
\begin{equation}\label{BoundPrExceptional}
\pr\left(\Lv(\sigma, \X)\in \mathcal{U}_J(C_4\varepsilon e^{\L}, \L+\varepsilon)\right) \ll \L^{2J} e^{4\L}\varepsilon^2.
\end{equation}
Combining this bound with \eqref{SecondCase} and \eqref{BoundMesExceptional} gives
$$ \Psi_T(\tau)= \frac{1}{T} \me\big\{ t\in [T, 2T]: \Lv(\sigma+it) \in (-\L, \L)^{2J}\big\} + O\left(\L^{2J} e^{4\L}\varepsilon^2+\frac{\L^{2J+1} }{\varepsilon^{2J}(\log T)^{(1-\theta)/2}}\right).$$
Similarly, it follows from \eqref{BoundPrExceptional} that
$$ \Psr(\tau)= \pr\left(\Lv(\sigma, \X) \in (-\L, \L)^{2J}\right) +O(\L^{2J} e^{4\L}\varepsilon^2).$$
Thus, using Theorem \ref{Discrepancy} we deduce that in this case
\begin{equation}\label{DiscrepancyCase2}
\Psi_T(\tau)=  \Psr(\tau) + O\left(\L^{2J} e^{4\L}\varepsilon^2+\frac{\L^{2J+1} }{\varepsilon^{2J}(\log T)^{(1-\theta)/2}} \right).
\end{equation}

\smallskip

\noindent \textbf{Case 3}: $\log(C_3\varepsilon)+ \L<\tau\le \widetilde{\L}.$ 
In this case we have 
 $$  \Psi_T(\tau)= \frac{1}{T} \me\big\{ t\in [T, 2T]: \Lv(\sigma+it) \in  \mathcal{V}_J(e^{\tau}, \L)\big\},$$
 where $\mathcal{V}_J(y, \L)$ is the bounded subset of $\mathbb{R}^{2J}$ defined by
\begin{align*}
\mathcal{V}_J(y, \L)
:= \Big\{ (u_1, \dots, u_J, v_1, \dots, v_J) \in \mathbb{R}^{2J} : \  & |u_j|, |v_j| <\L \text{ for all } 1\le j\le J,\\
& \text{ and } \left|\sum_{j=1}^J b_j e^{u_j+iv_j}\right| > y \Big\}.
\end{align*}
Similarly as before, we cover $\mathcal{V}_J(e^{\tau}, \L)$ with $\widetilde{K}(\tau)$ distinct hypercubes $\widetilde{\mathcal{B}}_k(\tau)$, each of which has nonempty intersection with $\mathcal{V}_J(e^{\tau}, \L)$ and sides of length $\varepsilon$.
The number of such hypercubes is 
$$ \widetilde{K}(\tau)  \ll \left(\frac{\L}{\varepsilon}\right)^{2J}.
$$
Now, let $1\leq k\leq \widetilde{K}(\tau)$ and $(u_1, \dots, u_J, v_1, \dots, v_J) \in \widetilde{\mathcal{B}}_k(\tau) \cap \mathcal{V}_J(e^{\tau}, \L)$. Then, for any $(x_1, \dots, x_J, y_1, \dots, y_J)\in \widetilde{\mathcal{B}}_k(\tau)$ we have $|x_j-u_j|\leq \varepsilon$ and 
$|y_j-v_j|\leq \varepsilon$ for all $1\leq j\leq J$. Hence, we deduce that $|x_j|, |y_j|<\L+\varepsilon$ for all $1\leq j\leq J$ and 
$$ \left| \sum_{j=1}^J b_j e^{x_j+iy_j} \right| = \left| \sum_{j=1}^J b_j e^{u_j+iv_j} \right| +O(\varepsilon e^{\L})> e^{\tau}- \frac{C_3}{2}\varepsilon e^{\L}$$
if $C_3$ is suitably large. Therefore, we have shown that
\begin{equation}\label{UpperLowerBoxes2}
\mathcal{V}_J(e^{\tau}, \L) \subset  \bigcup_{k \le \widetilde{K}(\tau)} \widetilde{\mathcal{B}}_k(\tau)  \subset \mathcal{V}_J\left(e^{\tau}-\frac{C_3}{2}\varepsilon e^{\L}, \L+\varepsilon\right).
\end{equation}
Now, using Theorem \ref{Discrepancy} we deduce
\begin{equation}\label{UpperBThirdCase}
\begin{aligned}
\Psi_T(\tau)& \leq 
\sum_{k=1}^{\widetilde{K}(\tau)} \frac{1}{T} \me\big\{ t\in [T, 2T]: \Lv(\sigma+it) \in \widetilde{\mathcal{B}}_k(\tau)\big\}\\
&= \sum_{k=1}^{\widetilde{K}(\tau)}\pr\left(\Lv(\sigma, \X)\in \widetilde{\mathcal{B}}_k(\tau)\right) + O\left(\frac{\widetilde{K}(\tau)\log\log T}{(\log T)^{(1-\theta)/2}}\right)\\
&\leq \pr\left(\Lv(\sigma, \X)\in \mathcal{V}_J\left(e^{\tau}-\frac{C_4}{2}\varepsilon e^{\L}, \L+\varepsilon\right)\right) + O\left(\frac{\L^{2J+1} }{\varepsilon^{2J}(\log T)^{(1-\theta)/2}}\right).
\end{aligned}
\end{equation}
Moreover, it follows from Lemma \ref{ConcentrationRandom} that 
$$  \pr\left( \Lv(\sigma, \X) \in (-\L-\varepsilon, \L+\varepsilon)^{2J} \text{ and } e^{\tau}-\frac{C_3}{2}\varepsilon e^{\L}< \left| \sum_{j=1}^J   b_j L_j(\sigma, \X)  \right| \leq e^{\tau}\right) \ll  \L^{2J} e^{4\L} \varepsilon,
$$
since $\tau\leq \widetilde{\L}=\L+O(1)$ by our assumption. Furthermore, since  the density $H_{T}(\ub, \vb)$ of the random vector $\Lv(\sigma, \X)$ is uniformly bounded in $\ub, \vb$ by Lemma \ref{lemma H_T bound 1}, we have 
\begin{equation}\label{RandomProbSCube}
\pr\left( \Lv(\sigma, \X) \in (-\L-\varepsilon, \L+\varepsilon)^{2J}\setminus (-\L, \L)^{2J}\right)\ll \L^{2J-1} \varepsilon.
\end{equation}
Combining these bounds we obtain
$$ 
\pr\left(\Lv(\sigma, \X)\in \mathcal{V}_J\left(e^{\tau}-\frac{C_4}{2}\varepsilon e^{\L}, \L+\varepsilon\right)\right)=\Psr(\tau)+ O\left(\L^{2J} e^{4\L} \varepsilon\right).
$$
Hence, inserting this estimate in \eqref{UpperBThirdCase} we deduce 
\begin{equation}\label{UB3}\Psi_T(\tau) \leq \Psr(\tau)+ O\left(\L^{2J} e^{4\L} \varepsilon + \frac{\L^{2J+1} }{\varepsilon^{2J}(\log T)^{(1-\theta)/2}}\right).
\end{equation}

We now proceed to prove the corresponding lower bound. Let $\tau_1$ be such that $e^{\tau}=e^{\tau_1}-\frac{C_3}{2}\varepsilon e^{\L}$. Then, it follows from \eqref{UpperLowerBoxes2} and Theorem \ref{Discrepancy} that
\begin{equation}\label{LowerBThirdCase}
\begin{aligned}
& \frac{1}{T} \me\big\{ t\in [T, 2T]: \Lv(\sigma+it) \in \mathcal{V}_J(e^{\tau}, \L+\varepsilon)\big\} \\
& \geq 
\sum_{k=1}^{\widetilde{K}(\tau_1)} \frac{1}{T} \me\big\{ t\in [T, 2T]: \Lv(\sigma+it) \in \widetilde{\mathcal{B}}_k(\tau_1)\big\}\\
&= \sum_{k=1}^{\widetilde{K}(\tau_1)}\pr\left(\Lv(\sigma, \X)\in \widetilde{\mathcal{B}}_k(\tau_1)\right) + O\left(\frac{\widetilde{K}(\tau)\log\log T}{(\log T)^{(1-\theta)/2}}\right)\\
&\geq \Psr(\tau_1) + O\left(\frac{\L^{2J+1} }{\varepsilon^{2J}(\log T)^{(1-\theta)/2}}\right).\\
\end{aligned}
\end{equation}
Moreover, by Lemma \ref{ConcentrationRandom} we have 
\begin{align*}
\Psr(\tau_1)&= \Psr(\tau) + O\Big(\pr\big(\Lv(\sigma, \X)\in (-\L, \L)^{2J} : e^{\tau}<  \left| \sum_{j=1}^J b_j L_j(\sigma, \X) \right| \leq e^{\tau}+ \frac{C_3}{2}\varepsilon e^{\L}  \big)\Big)\\
&= \Psr(\tau)+O\left(\L^{2J} e^{4\L} \varepsilon\right).
\end{align*}
Finally, we use Theorem \ref{Discrepancy} together with \eqref{RandomProbSCube} to deduce
\begin{align*} 
 \frac{1}{T} & \me\big\{ t\in [T, 2T]: \Lv(\sigma+it) \in \mathcal{V}_J(e^{\tau}, \L+\varepsilon)\big\}- \Psi_T(\tau) \\
& \leq  \frac{1}{T} \me\big\{ t\in [T, 2T]: \Lv(\sigma+it) \in(-\L-\varepsilon, \L+\varepsilon)^{2J}\setminus (-\L, \L)^{2J}\big\} \\
& =  \pr\left( \Lv(\sigma, \X) \in (-\L-\varepsilon, \L+\varepsilon)^{2J}\setminus (-\L, \L)^{2J}\right) + O\left(\frac{\log\log T}{(\log T)^{(1-\theta)/2}}\right)\\
&\ll \L^{2J-1} \varepsilon + \frac{\log\log T}{(\log T)^{(1-\theta)/2}}.
\end{align*}
Inserting these estimates in \eqref{LowerBThirdCase} yields
\begin{equation}\label{LB3}
\Psi_T(\tau) \geq \Psr(\tau)+  O\left(\L^{2J} e^{4\L} \varepsilon +\frac{\L^{2J+1} }{\varepsilon^{2J}(\log T)^{(1-\theta)/2}}\right).
\end{equation}

Thus we deduce from the estimates \eqref{DiscrepancyCase1}, \eqref{DiscrepancyCase2}  and  \eqref{UB3} and \eqref{LB3} that in all cases we have
$$ \Psi_T(\tau) = \Psr(\tau)+  O\left(\L^{2J} e^{4\L} \varepsilon +\frac{\L^{2J+1} }{\varepsilon^{2J}(\log T)^{(1-\theta)/2}}\right).$$
The desired result follows by choosing 
$$
\varepsilon= \frac{(\log\log T)^{1/(2J+1)}}{(\log T)^{(1-\theta+8\alpha)/(4J+2)}}.
$$

\end{proof}

\begin{proof}[Proof of Theorem \ref{AsympLogComb}]
By \eqref{EndRange} we have
$$
 \int_{-M}^{\widetilde{\L}}\Psi_T(\tau)d\tau =  \int_{-M}^{\widetilde{\L}} \frac{1}{T} \int_{\substack{ t\in S_0(T)\\ \log |F(\sigma+it)|>\tau}} dt d\tau= \frac{1}{T} \int_{t\in S_0(T)} \left(\log\big|F(\sigma+it)\big|+M\right)dt.$$
 Combining this identity with \eqref{TrunIntLog} and using that $\me(S_0(T))=T\Psi_T(-M)$ we obtain
\begin{equation}\label{MainApprox}
\frac{1}{T}\int_T^{2T} \log\big|F(\sigma+it)\big|dt =  \int_{-M}^{\widetilde{\L}}\Psi_T(\tau)d\tau -M \Psi_T(-M) + O\left(\frac{ {e^{O(\sqrt{\log\log T})}}
}{(\log T)^{\alpha^2/(\xi\theta)-3\theta}}\right).
\end{equation}

We now repeat the exact same approach for the random model $F (\sigma, \X)$. 
Using the same argument leading to \eqref{TrunIntLog} but with Lemma \ref{MomentsLogR} instead of Lemma \ref{MomentsLogF}, we deduce similarly that 
$$\ex\left(\log \big|F(\sigma, \X)\big|\right)= \ex\left(\mathbf{1}_{\mathcal{S}} \cdot \log \big|F(\sigma, \X)\big|\right)+ O\left(\frac{ {e^{O(\sqrt{\log\log T})}}
 }{(\log T)^{\alpha^2/(\xi\theta)}}\right),$$
where $\mathbf{1}_{\mathcal{S}}$ is the indicator function of $\mathcal{S}$. 
Therefore, reproducing the argument leading to \eqref{MainApprox} we obtain
\begin{equation}\label{MainApprox2}
\ex\left(\log \big|F(\sigma, \X)\big|\right)= \int_{-M}^{\widetilde{\L}}\Psr(\tau)d\tau -M \Psr(-M) + O\left(\frac{ {e^{O(\sqrt{\log\log T})}} 
}{(\log T)^{\alpha^2/(\xi\theta)}}\right).
\end{equation}

Combining \eqref{MainApprox} and \eqref{MainApprox2} together with Proposition \ref{CoveringResult} we deduce that 
\begin{multline}\label{LogIntegralLast}
\frac{1}{T}\int_T^{2T} \log\big|F(\sigma+it)\big|dt 
 = \ex\left(\log \big|F(\sigma, \X)\big|\right) \\
 + O\left(
\frac{(\log\log T)^{2J+4}}{(\log T)^{(1-\theta-16 J\alpha)/(4J+2)-3\theta}} + \frac{ {e^{O(\sqrt{\log\log T})}}
 }{(\log T)^{\alpha^2/(\xi\theta)-3\theta}} \right).
\end{multline}
We first require that $\theta$ satisfies $\theta<\alpha/\sqrt{3\xi} $, so that the exponent of $\log T$ in the second error term of \eqref{LogIntegralLast} is negative. Therefore, we might choose $\alpha$ be slightly bigger than $\sqrt{3\xi} \theta$. Hence, in order to insure that the exponent of $\log T$ in the error term of \eqref{LogIntegralLast} is negative, we thus require that $\theta$ satisfies the inequality
$$
12J \theta + 7\theta +16 J \sqrt{3\xi} \theta <1 \Longleftrightarrow \theta < \frac{1}{12 J+ 7 + 16 J \sqrt{3\xi}}.
$$
This completes the proof.

\end{proof}


\section{Preliminary results}\label{sec:Preliminaries}
 In this section we provide several of the technical lemmas that we shall need later. We first record several useful facts. Since
$$ \log L_j (s) = \sum_p \sum_{i=1}^d     \sum_{k=1}^\infty   \frac{ \alpha_{j,i}(p)^k }{ k p^{ks}},$$
we see that
$$ \beta_{L_j }(p^k) = \frac{1}{k} \sum_{i=1}^d \alpha_{j,i}(p)^k   $$
and 
\begin{equation}\label{beta bound 0}
  |\beta_{L_j}(p^k)| \leq \frac{d}{k}   p^{k\theta} 
  \end{equation}
for $ k = 1, 2 , \ldots $ and $ j = 1 , \ldots , J$.  
 For later use, we remark that
\be\label{prime sum bounds} 
  \sum_{ k=3}^\infty \sum_p   \frac{  | \beta_{L_j }( p^k) |}{p^{k/2}}  < \infty, \quad 
  \sum_{ k=2}^\infty \sum_p   \frac{  | \beta_{L_j }( p^k) |^2 }{p^{k}}  < \infty, \quad 
   \sum_p   \frac{  | \beta_{L_j }( p ) |^4 }{p^2}  < \infty .
\ee
One can easily show \eqref{prime sum bounds} by applying assumption A3 and inequalities
\be\label{beta bound 1} 
   | \beta_{L_j } (p^k) | \leq    \frac1k  \sum_{i=1}^d  | \alpha_{j,i}(p)|^k \leq \frac{p^{(k-2)\theta}}{k} \sum_{i=1}^d  |\alpha_{j,i} (p)|^2  \quad \mathrm{for~} k \geq 2 
 \ee 
 and  
\be\label{beta bound 2}
| \beta_{L_j } (p ) |^2  \leq    \bigg(  \sum_{i=1}^d  | \alpha_{j,i}(p)| \bigg)^2    \leq d  \sum_{i=1}^d  |\alpha_{j,i} (p)|^2  .     
 \ee
 \begin{lem}\label{lem:SecondMomentTail}
 Let $1\leq j , k  \leq J$. 
Then uniformly for $1/2<\sigma\le 1$ we have
\begin{equation}\label{SOC2}
 \sum_p  \frac{ \beta_{L_{j}}(p)  \overline{\beta_{L_{k}}(p) }  }{p^{2 \sigma}}= \delta_{j,k } \xi_j \log\left(\frac{1}{\sigma-1/2}\right) + c'_{j,k} + O\left((\sigma-1/2)\log\left(\frac{1}{\sigma-1/2}\right)\right)
\end{equation}
for some constants $c'_{j,k}$. 
Moreover, uniformly for $1/2<\sigma\le1$ we have
\begin{equation}\label{SOCTail}
 \sum_{p>Y}  \frac{ \beta_{L_{j}}(p)  \overline{\beta_{L_{k}}(p) } }{p^{2 \sigma}}= \delta_{j,k } \xi_j  \log\left(\frac{1}{(\sigma-1/2)\log Y}\right) +O(1)  \mathrm{~~if~} 2\leq Y\leq \exp \Big(\frac{1}{2\sigma-1}\Big), 
\end{equation}
 and 
\begin{equation}\label{SecondMomentTail}
\sum_{p^n>Y} \frac{|\beta_{L_{j}}(p^n)|^2}{p^{2n\sigma}}\ll \frac{Y^{1-2\sigma }}{ (2 \sigma-1)\log Y }   \mathrm{~~if~} Y\geq  \exp \Big(\frac{1}{2 \sigma-1 }\Big) .
\end{equation}
 \end{lem}
 
 \begin{proof} We start by proving \eqref{SOC2}. First, by partial summation and \eqref{SOC} we derive

 \be\label{SOC3} \begin{split}
 \sum_p       \frac{ \beta_{L_{j}}(p)  \overline{   \beta_{L_{k}}(p) }  }{p^{2 \sigma}}     =  &     \delta_{j,k} \xi_{j}   \int_2^\infty  \frac{ u^{- 2 \sigma}}{ \log u} du + b_{j,k} \\
 & + O\bigg((\sigma- \frac12)\left(1 + \int_2^\infty  \frac{ u^{-2 \sigma}}{ \log u } du \right)\bigg),   
\end{split}\ee
for some constants $b_{j,k}.$
 To evaluate the integral on the right hand of this estimate we use the substitution $ w= (2\sigma-1) \log u $. This gives
\be \label{IntSOC} 
\begin{split}
 \int_2^\infty \frac{ u^{-2\sigma}}{\log u } du  & = \int_{(2 \sigma-1)\log 2}^\infty  e^{-w} \frac{dw}{w}  \\
 & =  \int_{(2 \sigma-1)\log 2}^1 \frac{dw}{w} +    \int_{ (2 \sigma-1)\log 2   }^1  ( e^{-w}-1) \frac{dw}{w}    + \int_1^\infty e^{-w} \frac{dw}{w}   \\
 & =  \log\left(\frac{1}{\sigma -1/2}\right)-\log(2\log 2)  -\gamma + O(\sigma-1/2), 
\end{split}\ee
where $$\gamma= \int_{0}^1  ( 1-e^{-w}) \frac{dw}{w}    - \int_1^\infty e^{-w} \frac{dw}{w}, $$ 
is the Euler-Mascheroni constant. Inserting this estimate in \eqref{SOC3} implies \eqref{SOC2}.

We next establish \eqref{SOCTail}. Similarly to \eqref{SOC3} one has 
$$ \sum_{p>Y} \frac{ \beta_{L_{j}}(p)  \overline{   \beta_{L_{k}}(p) } }{p^{2 \sigma}}    =   \delta_{j,k } \xi_{ j}   \int_Y^\infty  \frac{ u^{- 2 \sigma}}{ \log u} du + O(1).
$$
To estimate this integral we again use the substitution $ w= (2\sigma-1) \log u $.
Then similarly to \eqref{IntSOC} one obtains
$$ 
\int_Y^\infty  \frac{ u^{- 2 \sigma}}{ \log u} du = \int_{(2 \sigma-1)\log Y}^\infty  e^{-w} \frac{dw}{w}= \log \left(\frac{1}{(\sigma-1/2) \log Y}\right)+O(1),
$$ 
which implies \eqref{SOCTail}.
 
We finally turn to the proof of \eqref{SecondMomentTail}. By partial summation and \eqref{SOC} it follows that in the range $Y>  \exp(1/(2\sigma-1 )) $ we have
 $$ 
 \sum_{p>Y} \frac{|\beta_{L_j}(p)|^2}{p^{2\sigma}}
 \ll \int_Y^\infty u^{-2\sigma} \frac{du}{\log u } + \frac{ Y^{1-2\sigma}}{ \log Y} \ll \int_{ ( 2 \sigma-1) \log Y}^\infty e^{-w} \frac{dw}{w} +  \frac{ Y^{1-2\sigma}}{ \log Y} \ll \frac{ Y^{1-2\sigma}}{ (2 \sigma -1 )  \log Y}.
 $$ 
 Now, we will bound the contribution of the prime powers. By \eqref{beta bound 0} and \eqref{beta bound 1}, we have 
 $$ |\beta_{L_j} (p^n)|^2 \leq \frac{d p^{2(n-1)\theta}}{n^2}\sum_{i=1}^d |\alpha_{  j,i }(p)|^2, $$
 and hence by assumption A3 and partial summation, we have
\begin{align*}
  \sum_{\substack{p^n >  Y\\ n\geq 2}} \frac{|\beta_{L_j} (p^n)|^2}{p^{2n\sigma}} \ll & \sum_{p > \sqrt{Y}} \sum_{n=2}^\infty    \frac{\sum_{i=1}^d |\alpha_{  j,i}(p)|^2}{n^2 p^{2( \sigma -\theta)n + 2\theta}} + \sum_{p \leq \sqrt{Y}} \sum_{n \geq \frac{ \log Y}{ \log p}}     \frac{\sum_{i=1}^d |\alpha_{  j,i}(p)|^2}{n^2 p^{2( \sigma -\theta)n + 2\theta}}    \\
  \ll & \sum_{p > \sqrt{Y}}     \frac{\sum_{i=1}^d |\alpha_{  j,i}(p)|^2}{p^{4\sigma - 2\theta }} + \sum_{p \leq \sqrt{Y}}   \frac{\sum_{i=1}^d |\alpha_{  j,i}(p)|^2}{ Y^{2( \sigma -\theta) }p^{ 2\theta}}  \\
  \ll & Y^{ ( 1+ \varepsilon - 4 \sigma + 2 \theta)/2} + Y^{ - 2 ( \sigma- \theta) +  (1+\varepsilon - 2 \theta )/2 } 
  \end{align*}
  for any $\varepsilon>0$.
By choosing $\varepsilon $ sufficiently small, we obtain
 $$ \sum_{ p^n > Y } \frac{|\beta_{L_j} (p^n)|^2}{p^{2n\sigma}} \ll  \frac{ Y^{1-2\sigma}}{ (2 \sigma -1 )  \log Y},
 $$ 
 which completes the proof.
 \end{proof}
Let $L(s)$ be an $L$-function satisfying assumptions A1--A4.  Here and throughout, we define for $Y\geq 2$ and  $\sigma, t\in \mathbb{R}$ 
$$
R_{L, Y}(\sigma+it):= \sum_{p^n\leq Y } \frac{ \beta_{L}(p^n)}{p^{n (\sigma+it)}}  \text{ and } R_{L, Y}(\sigma, \X):= \sum_{p^n\leq Y } \frac{ \beta_{L}(p^n)\X(p)^n}{p^{n\sigma}}, 
$$
where $\{\X(p)\}_p$ is a sequence of independent random variables, uniformly distributed on the unit circle.

Our next result shows that $\log L(\sigma+it)$ can be approximated by $R_{L,Y}(\sigma+it)$ for $1/2+ 1/G(T) \leq  \sigma \leq 1$, and for all $t\in [T, 2T]$ except for an exceptional set with a very small measure. This is accomplished using the zero-density estimates \eqref{assumption zero}.
\begin{lem}\label{lem : Dir Poly approx}
Let $L(s)$ be an $L$-function satisfying assumptions A1--A4.  Let $T$ be large and $G(T)$ be such that $2<G(T)\leq c_0\sqrt{\log T/\log\log T},$ for some suitably small constant $c_0>0$. Put  $Y=e^{A G(T) \log\log T}$ for a constant $A\geq 5$. Then there is a positive constant $c_1$ such that for all $t \in [T, 2T]$ except for a set of measure $\ll T \exp(-c_1 \log T/ G(T))$, we have 
$$ 
\log L(\sigma + it) = R_{L,Y} (\sigma+it) + O\left(\frac{1}{(\log T)^{A/2-2}}\right)
$$
uniformly for $\sigma \geq 1/2+ 1/G(T)$.
\end{lem}

\begin{proof} 
By Perron's formula we have
 \be\label{Perron 1} \begin{split}
R_{L,Y}(\sigma+it)= & \frac{1}{2 \pi i } \int_{c-iY}^{c+iY} \log L( \sigma+it + w) \frac{Y^w}{w} dw\\
& +O\left(Y^{-\sigma+1/4}\sum_{p}\sum_{n=1}^{\infty} \frac{| \beta_{L}(p^n)|}{p^{5n/4}|\log(Y/p^n)|}\right)
\end{split} \ee
where $c = 5/4-\sigma$. To bound the error term of this last estimate, we split the sum {over primes} into three parts: $p^n\leq Y/2$, $Y/2<p^n<2Y$ and $p^n\geq 2Y$. The terms in the first and third parts satisfy $|\log(Y/p^n)|\geq \log 2$, and hence their contribution is 
$$
\ll Y^{-\sigma+1/4}\sum_{p}\sum_{n=1}^{\infty} \frac{| \beta_{L}(p^n)|}{p^{5n/4}}=  Y^{-\sigma+1/4}
\left(\sum_{p} \frac{|\beta_{L}(p)|}{p^{5/4}}+O(1) \right)\ll Y^{-1/4}
$$
by \eqref{beta bound 0} and \eqref{beta bound 2}. 
To handle the contribution of the terms $Y/2<p^n<2Y$, we put $r=Y-p^n$, and use {the lower bound} $|\log(Y/p^n)|\gg |r|/Y$. Then the contribution of these terms is 
$$\ll Y^{-\sigma+\theta-5/4}\sum_{r\leq Y}\frac{1}{r}\ll  Y^{-1/2+\theta} \log Y.$$

Let $w_0= -1/(2G(T))$ and assume 
 that $L(\sigma+it +w )$ has no zeros in the half-strip  given by
 $ \re(w) \geq  -3/(4G(T)) $, $ | \im(w) | \leq Y+1$. 
 Then in the slightly smaller half-strip $\{ w : \re(w) \geq w_0, |\im(w)|\leq Y\}$ we have
\be\label{L'/L 1}
 \frac{L'}{L}(\sigma+it +w) \ll  G(T) \log T 
\ee
 (see    Proposition 5.7 in \cite{IK}).
 Observe that this holds  for all $ t \in [T,2T]$ except for $t$ in a set of measure 
  \begin{align*}
  & \ll Y \cdot N(\tfrac12+\tfrac{1}{4 G(T)}, 2T) \ll T^{1-c_2/(4G(T))} (\log T)^{c_3} \exp(A G(T)\log\log T) \\
  &  \ll T \exp\left(-c_1 \frac{\log T}{G(T)}\right)
  \end{align*}
for some constants $c_1, c_2, c_3>0$  by \eqref{assumption zero} and our assumption on $G(T)$. Now, integrating both sides of \eqref{L'/L 1} along the horizontal segment from $w$ to $w+B$, where $B$ is sufficiently large, we see that for such $t$ we have
$$ \log L(\sigma+it+w) \ll G(T) \log T,$$ 
for all $w$ such that $\re(w) \geq w_0, |\im(w)|\leq Y$.
Using this and shifting the contour to the left in \eqref{Perron 1}, we obtain
\begin{align*}
 R_{L,Y}(\sigma+it)-&\log L(\sigma+it) \\
 =& \frac{1}{2 \pi i } \left(\int_{c-iY}^{w_0-iY} + \int_{w_0-iY}^{w_0+iY} + \int_{w_0+iY}^{c+iY}\right) \log L(\sigma+it + w) \frac{Y^w}{w} dw  \\
 & +O\left(Y^{-1/2+\theta} \log Y+ Y^{-1/4}\right) \\
 \ll &  G(T)( \log T) Y^{w_0} \log Y  \ll  \frac{1}{(\log T)^{ A/2-2} }.
 \end{align*} 
  This proves the lemma.
\end{proof}

We now establish the analogous result for the random model $\log L(\sigma, \X)$.
 
 \begin{lem}\label{ApproxDirichletRandom}
 Let $\sigma>1/2$ and $Y\geq  \exp(1/(2\sigma-1 )) $.  Then, for all $\varepsilon>0$ we have 
 $$  \pr\big(\left| \log L_j (\sigma, \X)- R_{L_j, Y}(\sigma, \X)\right| \geq \varepsilon\big) \ll   \frac{Y^{1-2\sigma }}{ \varepsilon^2  (2 \sigma-1)\log Y } .$$
 \end{lem}
 \begin{proof}
By Chebyshev's inequality we have
\begin{equation}\label{ChebyIneq}
 \pr\big(\left| \log L_j (\sigma, \X)- R_{L_j, Y}(\sigma, \X)\right|\geq \varepsilon\big)
\leq \frac{1}{\varepsilon^2} \ex\left(\left| \log L_j (\sigma, \X)- R_{L_j, Y}(\sigma, \X)\right|^2\right).
 \end{equation}
Furthermore, observe that 
\begin{align*} \ex\left(\left| \log L_j (\sigma, \X)- R_{L_j, Y}(\sigma, \X)\right|^2\right) 
&= \sum_{p_1^n, p_2^m> Y} \frac{\beta_{L_j}(p_1^n)\overline{\beta_{L_j}(p_2^m)}}{p_1^{n\sigma}p_2^{m\sigma}}\ex\left(\X(p_1)^n\overline{\X(p_2)}^m\right) \\
& =\sum_{p^n> Y} \frac{|\beta_{L_j}(p^n)|^2}{p^{2n\sigma}}, 
\end{align*}
 since $\ex\left(\X(p_1)^n\overline{\X(p_2) ^m}\right)= 1$ only when $n=m$ and $p_1=p_2$, and is $0$ otherwise. Combining  this identity with Lemma \ref{lem:SecondMomentTail} and  \eqref{ChebyIneq} completes the proof.
  \end{proof}


We also need a standard mean value estimate, which follows from Lemma 3.3 of \cite{Ts}.
\begin{lem}\label{lem:dir poly mmt 1}
Let $    z \geq 2 $ be a real number and $k$ be a positive integer such that $ k \leq \log T/ \log z$. Let $\{a(p)\}_{p}$ be a sequence of complex numbers. Let $\{\X(p)\}_p$ be a sequence of independent random variables uniformly distributed on the unit circle. Then we have
$$ \frac{1}{T} \int_T^{2T} \bigg|  \sum_{   p \leq z } \frac{ a(p)}{ p^{it}} \bigg|^{2k} dt \ll k! \bigg(  \sum_{  p \leq z } |a(p)|^2 \bigg)^k,$$
and
$$ \E \bigg( \bigg|  \sum_{   p \leq z } a(p)\X(p) \bigg|^{2k}   \bigg)   \leq  k! \bigg(  \sum_{  p \leq z } |a(p)|^2 \bigg)^k. $$
\end{lem}
Using this result we bound the ($2k$)th moments of $R_{L_j,Y} (\sigma +it)$ and $R_{L_j,Y} (\sigma, \X)$. 

\begin{lem}\label{lem:dir poly mmt 2}
Let $1\leq j\leq J$.  
Let $\sigma\geq 1/2$ and $Y\geq 2 $ be real numbers. Then for all positive integers $k \leq \log T/ \log Y $, we have
$$ 
\frac{1}{T} \int_T^{2T} | R_{L_j,Y} (\sigma +it) |^{2k} dt \ll k!(\xi_j \log \log Y + O(\sqrt{\log \log Y}))^k, 
$$
and 
$$ \E \big(| R_{L_j,Y} (\sigma, \X) |^{2k} \big) \ll k!(\xi_j\log \log Y+O(\sqrt{\log \log Y}) )^k. $$
\end{lem}
\begin{proof}
We only prove the first estimate, as the second is similar and simpler. First, by \eqref{prime sum bounds} we have 
$$ R_{L_j ,Y} (\sigma +it)= \sum_{p\leq Y} \frac{\beta_{L_j} (p)}{p^{\sigma+it}} + \sum_{p\leq \sqrt{Y}} \frac{\beta_{L_j} (p^2)}{p^{2\sigma+2it}}+ O(1) 
$$
uniformly in $Y$ and $t$. Hence it follows from Minkowski's inequality that 
\begin{equation}\label{MinkowskiMom}
\begin{aligned}
\left(\int_T^{2T} | R_{L_j ,Y} (\sigma +it) |^{2k} dt\right)^{1/2k}
  \leq &  \left(\int_T^{2T} \left|\sum_{p\leq Y} \frac{\beta_{L_j} (p)}{p^{\sigma+it}}\right| ^{2k} dt \right)^{1/2k}  \\
  & + \left( \int_T^{2T} \left| \sum_{p\leq \sqrt{Y}} \frac{\beta_{L_j} (p^2)}{p^{2\sigma+2it}}\right|^{2k} dt \right)^{1/2k}   + C_1 T^{1/2k},
\end{aligned}
\end{equation}
for some constant $C_1>0$. Now, it follows from Lemma \ref{lem:dir poly mmt 1} together with Stirling's formula and assumption A5 that 
\begin{equation}\label{MomPrimesA}
\frac{1}{T} \int_T^{2T} \left|\sum_{p\leq Y} \frac{\beta_{L_j}(p)}{p^{\sigma+it}}\right| ^{2k} dt \ll k! \left( \sum_{p\leq Y} \frac{|\beta_{L_j}(p)|^2}{p^{2\sigma}}\right)^{k} \leq k! (\xi_j\log \log Y+O(1))^k, 
\end{equation}
since $\sigma\geq 1/2$. Similarly, by Lemma \ref{lem:dir poly mmt 1} and a simple change of variable we have
\begin{equation}\label{MomSquaresP} \begin{split}
\int_T^{2T} \left| \sum_{p\leq \sqrt{Y}} \frac{\beta_{L_j} (p^2)}{p^{2\sigma+2it}}\right|^{2k} dt & = \frac{1}{2}\int_{2T}^{4T} \left| \sum_{p\leq \sqrt{Y}} \frac{\beta_{L_j} (p^2)}{p^{2\sigma+it}}\right|^{2k} dt  \\
&  \ll 
T k! \left(\sum_{p\leq Y} \frac{|\beta_{L_j} (p^2)|^2}{p^2}\right)^{k}.
\end{split} \end{equation}
Now, by \eqref{beta bound 0}, \eqref{beta bound 1}, assumption A3 and partial summation we obtain
$$ \sum_{p\leq Y} \frac{|\beta_{L_j} (p^2)|^2}{p^2} \ll  \sum_{p\leq Y} \sum_{i=1}^d\frac{|\alpha_{j, i}(p)|^2}{p^{2-2\theta}} \ll 1.
$$
Combining this estimate with \eqref{MinkowskiMom}, \eqref{MomPrimesA} and \eqref{MomSquaresP} completes the proof.
\end{proof}
As a consequence of this result and Lemma \ref{lem : Dir Poly approx}, we establish Lemma \ref{LargeDevLj}.
\begin{proof}[Proof of Lemma \ref{LargeDevLj}] 
Let $Y=e^{10G(T) \log\log T}$. By Lemma \ref{lem : Dir Poly approx} for all $t \in [T, 2T]$ except for a set of measure $\ll T \exp(-c_1 \log T/ G(T))$ (for some constant $c_1>0$), we have
$$ 
\log L_j(\sigma + it) = R_{L_j,Y} (\sigma+it) + O\left(\frac{1}{(\log T)^3}\right). 
$$
Therefore, it follows from Lemma \ref{lem:dir poly mmt 2} that 
\begin{equation}\label{BoundLargeDevL}
\begin{aligned}
& \frac{1}{T} \textup{meas} \left\{t\in [T, 2T] : |\log L_j(\sigma+it)|>\tau \right\}\\
 &\leq 
\frac{1}{T} \textup{meas} \left\{t\in [T, 2T] : |R_{L_j,Y} (\sigma+it)|>\tau- \frac{1}{\log T}\right\} + O\left(e^{-c_1 \frac{\log T}{G(T)}} \right)\\
& \leq \frac{1}{(\tau-1/\log T)^{2k}} \frac{1}{T} \int_{T}^{2T} |R_{L_j,Y} (\sigma+it)|^{2k} dt + O\left(  e^{-c_1 \frac{\log T}{G(T)}}  \right) \\
& \leq k! \left(\frac{\xi_j \log G(T)+ O(  \sqrt{\log G( T)}  + \log_3 T )}{(\tau-1/\log T)^2}\right)^k+ O\left(  e^{-c_1 \frac{\log T}{G(T)}}  \right).
\end{aligned}
\end{equation}
Using Stirling's formula, and choosing $$k=\lfloor (\tau-1/\log T)^2/(\xi_j \log G(T)+ C \max\{ \sqrt{\log G( T)}, \log_3 T\}  ) \rfloor $$ for some suitably large constant $C$ implies the result. 

We now establish the analogous bound for $\log L_j(\sigma, \X)$. Let $ {\varepsilon= 1/(\log T)^2}$ and $ {Y= e^{G(T) (\log\log T)^{5}}}.$ Then, it follows from Lemma \ref{ApproxDirichletRandom} that 
\begin{align*}
\pr\left(\left| \log L_j (\sigma, \X)- R_{L_j, Y}(\sigma, \X)\right|>\frac{1}{(\log T)^2}\right) & \ll \frac{ (\log T)^4}{ ( \log \log T)^5}  e^{-2(\log\log T)^5}
   \ll    e^{-(\log\log T)^5}.
\end{align*}
Now, using the same argument leading to \eqref{BoundLargeDevL} together with Lemma \ref{lem:dir poly mmt 2} we obtain 
\begin{multline*}
 \pr\left(|\log L_j(\sigma, \X)|>\tau \right)
 \leq 
\pr \left(|R_{L_j,Y} (\sigma, \X)|>\tau- \frac{1}{\log T}\right)  
+ O(e^{-(\log\log T)^5})  \\
 \leq k! \left(\frac{\xi_j \log G(T)+ O(\sqrt{ \log G(T)}+ \log_3 T  )}{(\tau-1/\log T)^2}\right)^k+ O( e^{-(\log\log T)^5}).  
\end{multline*}
Making the same choice of $k$ and using Stirling's formula completes the proof.
\end{proof}

We extend the $\X(p)$ multiplicatively by defining for $n=\prod_p p^\alpha$, $\X(n) =\prod_p \X(p)^\alpha$. We finish this section with the following standard lemma, which we shall use in Section \ref{sec:Discrepancy} to prove that the characteristic function of the joint distribution of $\log L_j(\sigma+it)$ is very close to that of the joint distribution of $\log L_j(\sigma, \X)$.  
\begin{lem}\label{lem:dir poly mmt 3}
Let $b_j(n)$ be complex numbers, such that $|b_j(n)|\leq C $ for all $1\leq j\leq J$ and $ n \geq 1 $ and for some constant $ C>0$. 
Let $k_j, k_j'$ be positive integers for $ j \leq J$  and write $k = \sum_{j\leq J} k_j$ and $k' = \sum_{j\leq J} k_j'$. 
Then uniformly for $Y, T\geq 2$ we have  
\be\notag 
\begin{split}
& \frac1T \int_T^{2T} \prod_{j=1}^J \left(\sum_{n \leq Y} \frac{ b_j (n)}{ n^{it}}\right)^{k_j}  \left(\overline{\sum_{ n \leq Y} \frac{b_j (n)}{n^{it}}}\right)^{k_j'}  dt \\
& = \E \bigg(\prod_{j=1}^J \left(\sum_{n \leq Y} b_j (n)\X(n)\right)^{k_j}  \left(\overline{\sum_{ n \leq Y} b_j (n)\X(n)}\right)^{k_j'} \bigg)+ O \bigg( \frac{ (CY^2)^{k+ k'}}{T}   \bigg).
\end{split}
\ee
\end{lem}
\begin{proof}
We have
\begin{align*}
  &\frac1T \int_T^{2T}  \prod_{j=1}^J \left(\sum_{n \leq Y} \frac{ b_j (n)}{ n^{it}}\right)^{k_j}  \left(\overline{\sum_{ n \leq Y} \frac{b_j (n)}{n^{it}}}\right)^{k_j'}  dt \\
\qquad \qquad& =  \frac1T \int_T^{2T} \bigg( \sum_{n_{i,j} \leq Y}    \frac{ b_1 (n_{1,1})   \cdots  
b_1 ( n_{k_1 , 1} ) }{(n_{1,1} \cdots n_{ k_1 , 1 } )^{it} } \   \cdots    \  \frac{ b_J (n_{1,J})     \cdots    b_J ( n_{k_J , J} ) }{(n_{1,J} \cdots n_{ k_J , J } )^{it} }\bigg) \\
& \qquad \qquad\qquad \cdot \bigg(\sum_{m_{i,j} \leq Y}    \frac{ \overline{ b_1 (m_{1,1})} \cdots \overline{b_1 ( m_{k_1' , 1} ) }}{(m_{1,1} \cdots m_{ k_1' , 1 } )^{-it} } \  \cdots    \  \frac{ \overline{b_J (m_{1,J})} \cdots  \overline{b_J ( m_{k_J' , J} )} }{(m_{1,J} \cdots m_{ k_J' , J } )^{-it} }\bigg)\;\; dt .
\end{align*}
The contribution of the diagonal terms  is
\begin{align*}
\Sigma_1=& \sum_{\substack{ n_{i,j}, m_{i,j} \leq Y \\   \prod n_{i,j} = \prod m_{i,j}   } }  \prod_{j=1}^J  \left(\prod_{i=1}^{k_j} b_j (n_{i,j})  \prod_{i=1}^{k_j'} \overline{b_j (m_{i,j})} \right)   \\
 & = \E \bigg(\prod_{j=1}^J \bigg(\sum_{n \leq Y} b_j (n)\X(n)\bigg)^{k_j}  \bigg(\overline{\sum_{ n \leq Y} b_j (n)\X(n)}\bigg)^{k_j'} \bigg).
\end{align*}
The off-diagonal contribution is
\begin{align*}
\Sigma_2   = &  \sum_{\substack{ n_{i,j}, m_{i,j} \leq Y \\   \prod n_{i,j} \neq \prod m_{i,j}   } }  \prod_{j=1}^J  \bigg(\prod_{i=1}^{k_j} b_j (n_{i,j})  \prod_{i=1}^{k_j'} \overline{b_j (m_{i,j}) } \bigg)\bigg(\frac{(m/n)^{2iT} - (m/n)^{iT} }{i T \log (m/n)}\bigg),
\end{align*}
where $n =\prod n_{i,j} $ and $ m =\prod m_{i,j} $. Since $ n, m \leq Y^{k+k'} $ and $n\neq m $,
$$ \frac{1}{|\log (m/n)| }  \ll Y^{k+ k'} . 
$$ 
Hence, we derive
\begin{align*}
\Sigma_2   \ll  \frac{ (CY)^{k+ k'}}{T}   \sum_{\substack{ n_{i,j}, m_{i,j} \leq Y \\   \prod n_{i,j} \neq \prod m_{i,j}   } }  1  \ll  \frac{ (CY^2)^{k+ k'}}{T}.
\end{align*}
This completes the proof.
\end{proof}


\section{Bounding the discrepancy : proof of Theorem \ref{Discrepancy}}\label{sec:Discrepancy}
Let  $\u = ( u_1 , \dots, u_J)$ and similarly $\v, \xb,$ and $\yb$ be vectors in $\mathbb R^{J}$, and define
$$ \Phhat ( \xb , \yb) := \int_{ \mathbb{R}^{2J}} e^{ 2 \pi i  (\xb \cdot \ub + \yb \cdot \vb) } d\Phi_T ( \u, \v),$$
and
$$ \Phrhat (\xb , \yb ) := \int_{ \mathbb{R}^{2J} }  e^{ 2 \pi i  (\xb \cdot \ub + \yb \cdot \vb) } d\Phrand ( \u, \v),
$$
where $\xb \cdot \ub  = \sum_{j=1}^J x_j u_j $ is the dot product. Then by the definitions  of $\Phi_T  $ and $\Phrand$ in \eqref{PhiT def} and \eqref{Phi rand def},  we may write  
\begin{equation}\notag
\begin{split}
\Phhat ( \xb, \yb) = \frac1T \int_T^{2T} 
\exp \Bigg[ 2\pi i & \sum_{j=1}^J \big(    x_j \log | L_j   ( \sigma  +it)  | +  y_j \arg  L_j  ( \sigma + it)   \big) \Bigg] dt
\end{split}\end{equation}
and
\begin{equation}\notag
\begin{split}
\Phrhat ( \xb, \yb) = \E \Bigg(   \exp \Bigg[ 2 \pi  i & \sum_{j=1}^J \big(  x_j \log |   L_j   ( \sigma, \X )|   +   y_j \arg L_j ( \sigma, \X )   \big) \Bigg] \Bigg).
\end{split}\end{equation}

\begin{pro}\label{characteristic}
Let $T$ be large and $\sigma=1/2+ 1/G(T)$ where $2<G(T)\leq \frac{\sqrt{\log T}}{\log\log T }$. Let $ || \xb ||_{\infty} := \sup_{1 \leq j \leq J} |x_j | $. Then, for any constant $A>0$ there exists a constant $c_1>0$ such that for all $ \xb$ and $ \yb $ with $ ||\xb ||_{\infty}, ||\yb||_{\infty} \leq c_1\frac{ \sqrt{\log T}}{ \sqrt{G(T)}\log\log T}$, we have
$$
 \Phhat ( \xb , \yb ) = \Phrhat (\xb ,\yb ) + O \left(\frac{1}{(\log T)^A}\right).
$$
\end{pro}

\begin{proof}
Let $Y=\exp(BG(T)\log\log T)$, where $B=2A+6$. Then for every $j\leq J$, it follows from Lemma \ref{lem : Dir Poly approx} that 
\be\label{TruncLj}
\log L_j(\sigma+it)=R_{L_j, Y}(\sigma+it) + O\left(\frac{1}{(\log T)^{A+1}}\right)
\ee
for all $t\in [T, 2T]$ except for a set of measure $\ll T \exp(-c_1 \log T/ G(T))$ for some constant $c_1>0$.  
Let $\mathcal{A}(T)$ be the set of points $t \in [T, 2T]$ for which \eqref{TruncLj} holds for all $j\le J$. Then 
$$\text{meas}(\mathcal{A}(T))\ll T \exp(-c_1 \log T/ G(T))\ll  T \exp(-\sqrt{\log T}) . $$
Hence by \eqref{TruncLj},  $ \Phhat(\xb, \yb)$ equals
\begin{align*}
& \frac1T \int_{\mathcal{A}(T)} \exp\left( 2 \pi i\left(\sum_{j=1}^J (x_j \re \log L_j(\sigma+it)+ y_j \im \log L_j(\sigma+it))\right)\right)dt +O\left(e^{-\sqrt{\log T}}  \right)\\
&  =\frac1T \int_T^{2T}\exp\left( 2 \pi   i \left(\sum_{j=1}^J (x_j \re  R_{L_j, Y}(\sigma+it)+ y_j \im R_{L_j, Y}(\sigma+it))\right)\right)dt +O\left(\frac{1}{(\log T)^A}\right).
\end{align*}
 Let $N=[\log T/(10BG(T) \log\log T)]$.  Then, it follows from the previous estimate that $\Phhat(\xb,  \yb )$ equals
\begin{equation}\label{TaylorChar}
\sum_{n=0}^{2N-1} \frac{(2 \pi i)^n}{n!}   \frac1T \int_T^{2T}\bigg(\sum_{j=1}^J (x_j \re R_{L_j, Y}(\sigma+it)+ y_j \im R_{L_j, Y}(\sigma+it))\bigg)^ndt + E_1,
\end{equation}
where 
\begin{equation}\label{ErrorTaylorChar}
\begin{aligned}
E_1 
& \ll  \frac{1}{(\log T)^A}+ \frac{ ( 2 \pi)^{2N} (||\xb||_{\infty}+||\yb||_{\infty} )^{2N}}{(2N)!}\frac1T \int_T^{2T}  \bigg(\sum_{j=1}^J|R_{L_j, Y}(\sigma+it)|\bigg)^{2N}dt\\
& \ll \frac{1}{(\log T)^A}+ \frac{N!}{(2N)!} (C_2 \log \log Y)^N  \frac{ (c_1^2  \log T)^N}{ G(T)^N ( \log \log T)^{2N}}  \ll \frac{1}{(\log T)^A},
\end{aligned}
\end{equation}
by Lemma \ref{lem:dir poly mmt 2}, Minkowski's inequality and Stirling's formula, where $c_1 $  and $ C_2$ are  positive constants.
 
Next, we handle the main term of \eqref{TaylorChar}. To this end, we use Lemma \ref{lem:dir poly mmt 3}, which implies that for all non-negative integers $k_1, k_2, \dots, k_{2J}$ such that $k_1+\cdots  + k_{2J} \leq 2N$ we have
\begin{align*}
& \frac1T \int_T^{2T} \prod_{j=1}^J(R_{L_j, Y}(\sigma+it))^{k_j}\prod_{\ell=1}^J  (\overline{R_{L_j, Y}(\sigma+it)})^{k_{J+\ell}}dt
\\
& = \ex\left(\prod_{j=1}^J(R_{L_j, Y}(\sigma, \X))^{k_j}\prod_{\ell=1}^J  (\overline{R_{L_j, Y}(\sigma, \X)})^{k_{J+\ell}}\right) + O(T^{-1/2}).
\end{align*}
Let $z_j = x_j  + i y_j $ and $\bar{z}_j$ be its complex conjugate.  Then it follows from this estimate that for all $0\leq n\leq 2N$ we have
\begin{align*}
&\frac1T \int_T^{2T}\bigg(\sum_{j=1}^J (x_j \re R_{L_j, Y}(\sigma+it)+y_j \im R_{L_j, Y}(\sigma+it))\bigg)^ndt\\
&=\frac{1}{2^n}  \frac1T \int_T^{2T}\bigg(\sum_{j=1}^J (\bar{z}_j  R_{L_j, Y}(\sigma+it)+ z_j  \overline{  R_{L_j, Y}(\sigma+it)) }  \bigg)^ndt\\
&= \frac{1}{2^n} \sum_{\substack{k_1,\dots, k_{2J}\geq 0\\ k_1+\cdots+k_{2J}=n}}{n\choose k_1,  \dots, k_{2J}}
\frac1T \int_T^{2T} \prod_{j=1}^J(\bar{z}_j R_{L_j, Y}(\sigma+it))^{k_j}\prod_{\ell=1}^J  (z_{\ell}\overline{R_{L_\ell , Y}(\sigma+it)})^{k_{J+\ell}}dt\\
&= \frac{1}{2^n} \sum_{\substack{k_1,\dots, k_{2J}\geq 0\\ k_1+\cdots+k_{2J}=n}}{n\choose k_1, k_2, \dots, k_{2J}}
 \ex\left(\prod_{j=1}^J( \bar{z}_j R_{L_j, Y}(\sigma, \X))^{k_j}\prod_{\ell=1}^J  (z_{\ell}\overline{R_{L_j, Y}(\sigma, \X)})^{k_{J+\ell}}\right) \\
& \quad \quad +O\bigg(T^{-1/2} \bigg( \sum_{j=1}^J |z_j| \bigg)^n\bigg)\\
&= \ex\bigg[ \bigg(\sum_{j=1}^J (x_j \re R_{L_j, Y}(\sigma, \X)+ y_j \im R_{L_j, Y}(\sigma, \X))\bigg)^n \bigg] +O\bigg(T^{-1/2}   J^n  (||\xb||_{\infty}+||\yb||_{\infty})^n\bigg).
\end{align*}
Inserting this estimate in \eqref{TaylorChar}, we derive that $\Phhat(\xb, \yb)$ equals 
\begin{equation}\label{EstCharRand1}
\begin{aligned}
& \sum_{n=0}^{2N-1} \frac{(2 \pi i)^n}{n!}\ex\bigg(\bigg(\sum_{j=1}^J (x_j \re R_{L_j, Y}(\sigma, \X)+ y_j \im R_{L_j, Y}(\sigma, \X))\bigg)^n\bigg) + O((\log T)^{-A}) \\
& =\ex\bigg(\exp\bigg(2 \pi  i \sum_{j=1}^J (x_j \re  R_{L_j, Y}(\sigma, \X)+ y_j \im R_{L_j, Y}(\sigma, \X)) \bigg) \bigg) +  O((\log T)^{-A})  , 
\end{aligned}
\end{equation}
where the last estimate follows by Lemma \ref{lem:dir poly mmt 2} and the same argument as in \eqref{ErrorTaylorChar}. 

Let $\varepsilon>0$ be a parameter to be chosen, and define $\B_{\varepsilon}$ to be the event 
$$ \left| \log L_j (\sigma, \X)- R_{L_j, Y}(\sigma, \X)\right| < \varepsilon,$$ for all $j\leq J$. Let $\B_{\varepsilon}^{c}$ be the complement of $\B_{\varepsilon}$.  Then it follows from  Lemma \ref{ApproxDirichletRandom} that 
$$  \pr(\B_{\varepsilon}^{c}) \ll \frac{1}{\varepsilon^2 (\log T)^{2B}  }.$$
Let $\mathbf{1}_{\B_{\varepsilon}}$ be the indicator function of the event $\B_{\varepsilon}$. Then it follows from this estimate that $ \Phrhat (\xb ,\yb ) $ equals
\begin{align*}
  & \ex\bigg( \mathbf{1}_{\B_{\varepsilon}}\cdot \exp\bigg( 2 \pi   i  \sum_{j=1}^J (x_j \re  \log L_j(\sigma, \X)+ y_j \im \log L_j(\sigma, \X)) \bigg) \bigg) + O\bigg(\frac{1}{\varepsilon^2 (\log T)^{2B} } \bigg)\\
 &=  \ex\bigg(\mathbf{1}_{\B_{\varepsilon}}\cdot \exp\bigg( 2 \pi   i  \sum_{j=1}^J (x_j \re  R_{L_j, Y}(\sigma, \X)+ y_j  \im R_{L_j, Y}(\sigma, \X)) \bigg) \bigg) + O\bigg(\varepsilon +  \frac{1}{ \varepsilon^2 (\log T)^{2B}  } \bigg)\\
 & = \ex\bigg(\exp\bigg(2 \pi i \sum_{j=1}^J (x_j \re  R_{L_j, Y}(\sigma, \X)+ y_j \im R_{L_j, Y}(\sigma, \X)) \bigg) \bigg) +  O\bigg(\varepsilon +  \frac{1}{ \varepsilon^2 (\log T)^{2B}  } \bigg).
 \end{align*}
Choosing $\varepsilon= (\log T)^{-2B/3}$ and inserting this estimate in \eqref{EstCharRand1} completes the proof. 
\end{proof}


The deduction of Theorem \ref{Discrepancy} from Proposition \ref{characteristic} uses Beurling-Selberg functions.
For $z\in \mathbb C$ let
\[
H(z) =\bigg( \frac{\sin \pi z}{\pi} \bigg)^2 \bigg( \sum_{n=-\infty}^{\infty} \frac{\textup{sgn}(n)}{(z-n)^2}+\frac{2}{z}\bigg)
\qquad\mbox{and} \qquad K(z)=\Big(\frac{\sin \pi z}{\pi z}\Big)^2.
\]
Beurling proved that the function $B^+(x)=H(x)+K(x)$
majorizes $\textup{sgn}(x)$ and its Fourier transform
has restricted support in $(-1,1)$. Similarly, the function $B^-(x)=H(x)-K(x)$ minorizes $\textup{sgn}(x)$ and its Fourier
transform has the same property (see  \cite[Lemma 5]{Vaaler}).

Let $\Delta>0$ and $a,b$ be real numbers with $a<b$. Take $\mathcal I=[a,b]$
and
define
\[
F_{\mathcal{I}, \Delta} (z)=\frac12 \Big(B^-(\Delta(z-a))+B^-(\Delta(b-z))\Big).
\]
Then we have the following lemma, which is proved in \cite{LLR2} (see  Lemma 7.1  therein and the discussion above it).

\begin{lem} \label{lem:functionbd}
The function $F_{\mathcal{I}, \Delta}$ satisfies the following properties
\begin{itemize}
\item[1.]
For all $x\in \mathbb{R}$ we have $
|F_{\mathcal{I}, \Delta}(x)| \le 1
$ and  
\begin{equation} \label{l1 bd}
0 \le \mathbf 1_{\mathcal I}(x)- F_{\mathcal{I}, \Delta}(x)\le K(\Delta(x-a))+K(\Delta(b-x)).
\end{equation}
\item[2.]
 The Fourier transform of $F_{\mathcal{I}, \Delta}$ is
\begin{equation} \label{Fourier}
\widehat F_{\mathcal{I}, \Delta}(y)=
\begin{cases}\widehat{ \mathbf 1}_{\mathcal I}(y)+O\Big(\frac{1}{\Delta} \Big) \mbox{ if } |y| < \Delta, \\
0 \mbox{ if } |y|\ge \Delta.
\end{cases}
\end{equation}
\end{itemize}
\end{lem}

\begin{proof}[Proof of Theorem \ref{Discrepancy}]
First, it follows from Lemma \ref{LargeDevLj} that with $\tau=(\log\log T)^2$ we have 
\begin{align*}
& \frac1T\text{meas}\big\{ t \in [T, 2T] : \Lv( \sigma + it) \notin  [ -(\log\log T)^2, (\log\log T)^2]^{2J}  \big\} \ll \frac{1}{(\log T)^{10}},
\end{align*}
and  
$$ \P \big(\Lv ( \sigma , \X) \notin [ -(\log\log T)^2, (\log\log T)^2]^{2J}   \big) \ll \frac{1}{(\log T)^{10}}.$$
Therefore, it suffices to consider rectangular regions $R \subset [ -(\log\log T)^2, (\log\log T)^2]^{2J}$.

Let $A=J+3$ and $c_1$ be the corresponding constant in Proposition \ref{characteristic}. Let 
$$\Delta:=c_1\frac{\sqrt{\log T}}{\sqrt{G(T)}\log\log T},
$$ and $ R=\prod_{j=1}^{J}[a_j,b_j] \times \prod_{j=1}^{J}[c_j,d_j]$ for $j=1,\ldots, J$, with
$0< b_j-a_j , d_j - c_j \le 2(\log\log T)^2$. 
We also write $\mathcal I_j=[a_j,b_j]$ and $\mathcal J_j=[c_j,d_j]$. 
By Fourier inversion, \eqref{Fourier}, and Proposition \ref{characteristic} we have  
\begin{equation} \label{long est}
\begin{split}
&\frac1T \int_T^{2T} \prod_{j=1}^J F_{\mathcal I_j, \Delta} \Big(  \log |L_j(\sigma+it)|\Big) 
F_{\mathcal J_j, \Delta}\Big(\arg L_j(\sigma+it)\Big) \, dt\\
&
=\int_{\mathbb R^{2J}} \bigg(\prod_{j=1}^J \widehat{F}_{\mathcal I_j, \Delta} (x_j) 
\widehat{F}_{\mathcal J_j, \Delta}( y_j)\bigg)  \Phhat(-\xb , - \yb) \, d\xb \, d\yb  \\
&
= \int\limits_{\substack{|x_j|,|y_j| \le \Delta \\ j=1,2, \ldots, J}} \bigg(\prod_{j=1}^J \widehat{F}_{\mathcal I_j, \Delta} (x_j) 
\widehat{F}_{\mathcal J_j, \Delta}( y_j)\bigg)  \Phrhat ( - \xb , -  \yb ) \, d\xb  \, d\yb    + E_2 \\
& 
=\mathbb E \bigg( \prod_{j=1}^J F_{\mathcal I_j, \Delta} \Big(   \log |L_j(\sigma,\X)|\Big) 
F_{\mathcal J_j, \Delta}\Big( \arg L_j(\sigma,\X)\Big) \bigg)+O\left(\frac{1}{(\log T)^2}\right),
\end{split}
\end{equation}
where $E_2 = O\left(\Delta^{2J}\frac{(  \log\log T)^{4J}}{(\log T)^A}\right)$.

Next note that $\widehat K(\xi)=\max(0,1-|\xi|)$. Applying Fourier inversion, Proposition \ref{characteristic} with $J=1$,
and Lemma \ref{lemma abs cont}  we obtain
\begin{equation} \notag
\begin{split}
& \frac1T \int_T^{2T} K\Big( \Delta  ( \log | L_1(\sigma+it) | -\alpha )\Big) \, dt
=\frac{1}{\Delta}\int_{-\Delta}^{\Delta}\Big(1-\frac{|\xi|}{\Delta}\Big) e^{-2\pi i \alpha \xi} \Phhat (\xi,0, \dots, 0) \, d\xi \ll  \frac{1}{\Delta},
\end{split}
\end{equation}
where $\alpha$ is an arbitrary real number. By this and \eqref{l1 bd} we have that
\begin{equation} \label{K bd}
\frac1T \int_T^{2T}  F_{\mathcal I_1, \Delta}\Big(\textup{Re} \log L_1(\sigma+it)\Big) \, dt
=\frac1T \int_{T}^{2T} \mathbf 1_{\mathcal I_1}\Big(\textup{Re} \log L_1(\sigma+it)\Big) dt+O(1/\Delta).
\end{equation}
Lemma \ref{lem:functionbd} implies that $|F_{\mathcal I_j, \Delta}(x)|, |F_{\mathcal J_j, \Delta}(x)| \le 1$ for  $j=1,\ldots, J$. Hence, by this and \eqref{K bd} we have
\begin{align*}
&\frac1T \int_T^{2T} \prod_{j=1}^J F_{\mathcal I_j, \Delta} \Big( \textup{Re}\log  L_j(\sigma+it)\Big) 
F_{\mathcal J_j, \Delta}\Big( \arg L_j(\sigma+it)\Big) \, dt \\
&=\frac1T \int_T^{2T} \mathbf 1_{\mathcal I_1} \Big( \textup{Re} \log L_1(\sigma+it)\Big) 
F_{\mathcal J_1, \Delta}\Big( \arg L_1(\sigma+it)\Big) \\
&\qquad \qquad   \times \prod_{j=2}^J F_{\mathcal I_j, \Delta} \Big( \textup{Re} \log L_j(\sigma+it)\Big) 
F_{\mathcal J_j, \Delta}\Big( \arg L_j(\sigma+it)\Big) \, dt+O(1/\Delta).
\end{align*}
By using the same argument, one can prove analogs of \eqref{K bd} for  
$\textup{Re} \log L_j(\sigma+it)$ with $2\leq j\leq J$ and $\arg L_j(\sigma+it)$ with $1\leq j\leq J$. We then derive
\begin{equation} \label{one}
\begin{split}
&\frac1T \int_T^{2T} \prod_{j=1}^J F_{\mathcal I_j, \Delta} \Big( \textup{Re} \log L_j(\sigma+it)\Big) 
F_{\mathcal J_j, \Delta}\Big( \arg L_j(\sigma+it)\Big) \, dt \\
= &\frac1T \int_T^{2T} \prod_{j=1}^J \mathbf 1_{\mathcal I_j, \Delta} \Big( \textup{Re} \log L_j(\sigma+it)\Big) 
\mathbf 1_{\mathcal J_j, \Delta}\Big( \arg L_j(\sigma+it)\Big) \, dt +O_J\left(\frac{1}{\Delta}\right) \\
=&   \Phi_T(R)+O_J\left(\frac{1}{\Delta}\right).
\end{split}
\end{equation}
A similar argument shows that 
\begin{equation} \label{two}
\begin{aligned}
\mathbb E \bigg( \prod_{j=1}^J F_{\mathcal I_j, \Delta} \Big( \textup{Re} \log L_j(\sigma,\X)\Big) 
F_{\mathcal J_j, \Delta}\Big( \arg L_j (\sigma,\X)\Big) \bigg) = \Phrand(R) +O_J\left(\frac{1}{\Delta}\right).
\end{aligned}
\end{equation}
Inserting the estimates \eqref{one} and \eqref{two} in \eqref{long est} completes the proof.

\end{proof}


\section{$L^{2k}$ norm of $\log |\sum_{j=1}^J b_j L_j(\sigma+it)|$ : Proof of Proposition \ref{MomentsLogF} }\label{sec:MomentsLogF}

To prove Proposition \ref{MomentsLogF} we follow the same strategy as in the proof of Proposition 2.5 of \cite{LLR} for the Riemann zeta function, but we encounter new technical difficulties which we shall describe later.

We first start with the following classical lemma, which is a generalization of a lemma
of Landau (see Lemma $\alpha$ in \cite[Section 3.9]{Ti}). 
\begin{lem} [Lemma 5.1 of \cite{LLR}]\label{Landau}
Let $0< r \ll 1$. Also, let $s_0
=\sigma_0+it$ and suppose $f(z)$ is 
analytic in $|z-s_0| \leq r$. Define
\[
M_{r}(s_0)=\max_{|z-s_0| \leq r} 
\bigg|\frac{f(z)}{f(s_0)} \bigg| +3 \qquad
\mbox{ and }  \qquad N_{ r}(s_0)=\sum_{|\varrho-s_0|
\leq r} 1,
\]
where the last sum runs over the zeros $\varrho$ of $f(z)$
in the closed disk of radius $r$ centered at $s_0$.
Then for $0<\delta<r/2$ and $|z-s_0| \leq r-2\delta$
we have
\[
\frac{f'}{f}(z)=\sum_{|\rho-s_0| \leq r-\delta}
\frac{1}{s-\rho} +O\bigg( \frac{1}{\delta^2}
\Big(\log M_{r}(s_0) +
N_{r-\delta}(s_0) (\log 1/\delta+1) \Big)\bigg).
\]
\end{lem}
 
Recall that $L_j(s)$ has a Dirichlet series representation 
$$ L_j(s) =\sum_{n=1}^{\infty} \frac{\alpha_{L_j}(n)}{n^s}, $$
for $\re(s)>1$. 
We shall apply Lemma \ref{Landau} to the following function
\begin{equation}\label{ModifiedLC}
 f(z) =  \frac{n_0^z }{\sum_{j \leq J} b_j \alpha_{L_j}(n_0)} \sum_{j=1}^J  b_j L_j ( z ),
 \end{equation}
where $n_0$ is the smallest positive integer $n$ such that $ \sum_{j=1}^J b_j \alpha_{L_j}(n)  \neq 0$.
We let $\rho$ run over the zeros of $f$. We recall that $\sigma=1/2+ 1 / G(T)$, and  
choose
\begin{equation*} \label{def parameters}
\delta:=\frac1{5G(T)},  \qquad
  r:=\sigma_0-\frac12 - \frac{1}{2G(T) }  , \qquad \text{ and } R:=r+\delta.
\end{equation*}
where $ \sigma_0$ is taken to be large (but fixed) so that $|f(\sigma_0 +it) | \geq 1/10$ and $\min_{\rho} |s_0 - \rho| \geq 1/10$ uniformly in $t$. A straightforward generalization of Lemmas 5.2 and 5.3 of \cite{LLR} leads to the following result.  To be precise, we only include the major steps of the proof. 
\begin{lem}\label{LemLLR}
  Let $\sigma , \delta, r , R,$ and $s_0= \sigma_0 + it $ be as above.  Then there exists an absolute constant $c>0$ such that for every positive integer $k$ we have
\begin{align*}
  &   \int_T^{2T} \bigg|\;  \log   \bigg| \sum_{j=1}^J  b_j L_j ( \sigma +it) \bigg| \; \bigg|^{2k}  dt  \\
  & \ll c^k (G(T)^{3/2} k+ G(T)^3\log G(T) )^{2k} \sum_{n=\lfloor T \rfloor}^{\lfloor 2T \rfloor}\bigg(\log \bigg(\bigg|\sum_{j=1}^J  b_j L_j(s_n)\bigg| +3\bigg)\bigg)^{2k},
\end{align*}
where $s_n=\sigma_n+ it_n$ for $ n >  0 $ is a point at which $|\sum_{j=1}^J  b_j L_j(s)|$ achieves its maximum value on the set $\cup_{n\leq t\leq n+1}D_R(\sigma_0 + it )$, and $D_R(z)$ is the disc of radius $R$ centered at $z$. 
\end{lem} 
\begin{proof}
First, applying Lemma \ref{Landau} to $f(z)$ defined in \eqref{ModifiedLC} we obtain for $ |z- s_0 | \leq r-2 \delta $ that
$$ \frac{ f'(z)}{f(z )} = \sum_{ | \rho - s_0 | \leq r- \delta} \frac{1}{ z-\rho} + O \bigg(\frac{1}{ \delta^2} ( \log M_r ( s_0 ) + N_{r-\delta} ( s_0 ) ( \log 1/\delta + 1 ) ) \bigg)  $$
 where $$ M_{r}(s_0)=\max_{|z-s_0| \leq r} 
\bigg|\frac{f(z)}{f(s_0)} \bigg| +3 \qquad
\mbox{ and }  \qquad N_{ r}(s_0)=\sum_{|\varrho-s_0|
\leq r} 1.$$
Now, a standard application of Jensen's formula shows that (see (5.4) of \cite{LLR})
$$ N_{r-\delta }(s_0) \leq \frac{r}{\delta} (\log M_r  (s_0) + \log 10 ). 
$$ 
Hence we derive
 $$ \frac{ f'(z ) }{f(z ) } = \sum_{ | \rho - s_0 | \leq r- \delta} \frac{1}{ z-\rho} + O \big(G(T)^3 \log G(T) \log M_r  ( s_0 )   \big)  $$
 for $ |z- s_0 | \leq r-2 \delta$. We integrate  both sides  from $s_0= \sigma_0 + it $ to $s=\sigma+it$ and take the real parts, to obtain
 \begin{align*}
 \log | f(s)| -\log |f(s_0)| & =   \sum_{ | \rho - s_0 | \leq r- \delta}\log | s -\rho| +  O \big( N_{r-\delta }( s_0 ) + G(T)^3 \log G(T) \log M_r  ( s_0 )   \big) \\ 
  & =   \sum_{ | \rho - s_0 | \leq r- \delta}\log | s -\rho|   + O \big(G(T)^3 \log G(T) \log M_r  (s_0 )\big), 
  \end{align*}
  since $\log |s_0-\rho|=O(1)$ for all zeros $\rho$ with $|\rho - s_0 | \leq r- \delta$ by our choice of $\sigma_0$. Furthermore, since $\log |f(s_0)|=O(1)$ and $\log|f(s)|= \log|\sum_{j=1}^J b_j L_j(s)|+O(1)$, we deduce that 
 $$ \log \bigg| \sum_{j=1}^J b_j L_j(\sigma+it) \bigg| \leq \sum_{ | \rho - s_0 | \leq r- \delta}\log | \sigma+it -\rho|   + c_1 G(T)^3 \log G(T) \log M_r (s_0 ), $$
 for some positive constant $c_1$. We now use the simple inequality 
 $$(x+ y)^{2k} \leq 2^{2k} \max(|x|, |y|)^{2k} \leq 2^{2k} (|x|^{2k}+ |y|^{2k})$$
  for all real numbers $x, y$, to deduce that
\begin{multline}\label{ReductionLLR}
   \int_T^{2T} \bigg|\;  \log   \bigg| \sum_{j=1}^J b_j L_j ( \sigma +it) \bigg| \; \bigg|^{2k}  dt  
 \leq 4^{k} 
  \int_T^{2T} \bigg(\sum_{|\rho-s_0| \leq r-\delta} 
|\log |\sigma+it-\rho|| \bigg)^{2k} dt\\
 + (2c_1 G(T)^3 \log G(T))^{2k}
  \int_T^{2T} (\log M_r (s_0 ))^{2k} dt.
\end{multline}
For $ n >0 $, let $s_n=\sigma_n+ it_n$ be a point at which $|\sum_{j=1}^J b_j L_j(s)|$ achieves its maximum value on the set $\cup_{n\leq t\leq n+1}D_R(\sigma_0 +it )$. Then, we note that 
\begin{equation}\label{Discrete1}\begin{split}
\int_T^{2T} (\log M_r (s_0 ))^{2k} dt & \leq \sum_{n=\lfloor T \rfloor}^{\lfloor 2T \rfloor } \int_n^{n+1} (\log M_r (s_0 ))^{2k} dt  \\
& \ll c_2^k \sum_{n=\lfloor T \rfloor}^{\lfloor 2T \rfloor }\left(\log \left(\left|\sum_{j=1}^J b_j L_j(s_n)\right| +3\right)\right)^{2k}
\end{split}
\end{equation}
for some absolute constant $c_2>0$.
Furthermore, a straightforward generalization of the proof of Lemma 5.3 of \cite{LLR} implies that 
\begin{equation}\label{Discrete2}
\begin{aligned}
&    \int_T^{2T} \bigg(\sum_{|\rho-s_0| \leq r-\delta} 
|\log |\sigma+it-\rho|| \bigg)^{2k} dt\\
&  \ll (c_3k)^{2k}  \sum_{n=\lfloor T \rfloor}^{\lfloor 2T \rfloor}\bigg(\frac{1}{\delta} \sum_{\ell \leq 1/\sqrt{\delta}} \log M_{R}(\sigma_0+i(n+\ell \sqrt{\delta}))\bigg)^{2k} \\
&\ll \frac{(c_3 k)^{2k}}{\delta^{3k}} \sum_{n=\lfloor T \rfloor}^{\lfloor 2T \rfloor}\bigg(\log  \bigg(\bigg|\sum_{j=1}^J b_j L_j(s_n)\bigg| +3\bigg) \bigg)^{2k},
\end{aligned}
\end{equation}
where $c_3>0$ is an absolute constant. Inserting the estimates \eqref{Discrete1} and \eqref{Discrete2} in \eqref{ReductionLLR} completes the proof.
\end{proof}
In the case of the Riemann zeta function, in order to bound $\sum_{n=\lfloor T \rfloor}^{\lfloor 2T \rfloor}\left(\log \left|\zeta(s_n)\right| +3\right)^{2k}$, the authors of \cite{LLR} use Jensen's inequality together with standard estimates for the second moment of $\zeta(s)$. However, estimates for the second moment  are not known in general for the $L$-functions in our class. Using a different approach, we were able to overcome this problem and establish the following result. 

\begin{lem} \label{MeanMax} 
Let $\delta, r, R,$ and $s_0 = \sigma_0 + it $ be as above.  Let $D_a(z)$ be the disc of radius $a$ centered at $z$. For $n > 0$, let $s_n=\sigma_n+ it_n$ be a point at which $|\sum_{j=1}^J b_j L_j(s)|$ achieves its maximum value on the set $\cup_{n\leq t\leq n+1}D_R(\sigma_0 + i t )$. Then there exist   positive constants $c_1 $ and $c_2$  such that for all positive integers $k \leq \log T/(c_1  G(T)\log\log T)$ we have 
$$ \sum_{n=\lfloor T \rfloor}^{\lfloor 2T \rfloor}\bigg(\log \bigg(\bigg|\sum_{j=1}^J b_j L_j(s_n)\bigg| +3\bigg)\bigg)^{2k} \ll T G(T)^2   (c_2 k \log\log T)^k. $$
\end{lem}

\begin{proof}
We first observe that 
\begin{align*}
  \bigg( \log \bigg(\bigg|\sum_{j=1}^J b_j L_j(s_n)\bigg| +3\bigg) \bigg)^{2k} & \leq   { C_1^k\max_{j \leq J}  \{ \left( \log ( |L_j ( s_n ) | +3)\right)^{2k} \}} \\
 &  \leq  C_1^k \sum_{j=1}^J \left( \log ( |L_j ( s_n ) | +3)      \right)^{2k} \end{align*}
for some constant $C_1 >0$ that depends on $J$ and the $b_j$. Thus, we have
$$ \sum_{n=\lfloor T \rfloor}^{\lfloor 2T \rfloor}\bigg(\log \bigg(\bigg|\sum_{j=1}^J b_j L_j(s_n)\bigg| +3\bigg)\bigg)^{2k} \leq    C_1^k  \sum_{j=1}^J  \sum_{n=\lfloor T \rfloor}^{\lfloor 2T \rfloor}  \left( \log ( |L_j ( s_n ) |  +3)        \right)^{2k} .$$
To prove the lemma, it is enough to show that
$$  \sum_{n=\lfloor T \rfloor}^{\lfloor 2T \rfloor}  \left( \log ( |L_j ( s_n ) | +3)        \right)^{2k} \ll  T G(T)^2  (C_2 k \log\log T)^k $$
for every $j \leq J $ and for some constant $C_2 > 0 $. 

Without loss of generality, we only consider the case $j = 1 $. Let 
\begin{align*}
 \mathcal{A}_1 (T) & := \{  \lfloor T \rfloor  \leq n \leq \lfloor 2T \rfloor  :   | L_1 ( s_n ) | \leq 5 \},\\
 \mathcal{A}_2 (T)& := \{   \lfloor T \rfloor  \leq n \leq \lfloor 2T \rfloor : | L_1 ( s_n ) | > 5,  L_1 ( s ) = 0 \mathrm{~for~some~} |s-s_n | \leq \delta \}, \\
 \mathcal{A}_3 (T)& := \{   \lfloor T \rfloor  \leq n \leq \lfloor 2T \rfloor : | L_1 ( s_n ) | > 5,  L_1 ( s ) \neq 0 \mathrm{~for~all~} |s-s_n | \leq \delta \}.
 \end{align*}
  Then we see that
  \begin{equation}\label{A1T bound}
    \sum_{n \in \mathcal{A}_1 (T) }  \left( \log ( |L_1 ( s_n ) | +3)        \right)^{2k}  \leq T ( \log 8)^{2k} . 
    \end{equation}
To bound the sum over $\mathcal{A}_2 (T)$, we use the classical Phragmen-Lindel\"of principle which implies that there exists $\kappa>0$ (which might depend on $x$ and $d$ in assumption A1) such that 
\begin{equation}\label{PL bound}
|L_1(x+iy )| \ll (1+|y|)^{\kappa}.
\end{equation}
If $ | \rho - s_n| \leq \delta $ and $L_1 ( \rho)=0$, then $ \re (\rho)  \geq 1/2 + 1/ (10G(T))$. By assumption A4 we have
\begin{equation}\label{A2T bound}\begin{split}
\sum_{n \in \mathcal{A}_2 (T) }  \left( \log ( |L_1 ( s_n ) | +3)        \right)^{2k} & \ll C_3^k ( \log T)^{2k} N_{L_1} ( 1/2 + 1/ ( 10 G(T)), T) \\
&  \ll C_3^k T e^{ - C_4  \frac{ \log T}{  G(T)} + 2k \log \log T}\\
& \ll C_3^k T e^{ -  \frac{C_4 }{2}  \frac{ \log T}{  G(T)}}
\end{split} \end{equation}
for some constants $ C_3 , C_4 > 0$ and for $ k \leq \log T/ ( c_1 G(T) \log \log T)$ by choosing $c_1$ sufficiently large. 
Lastly, for each $ n \in \mathcal{A}_3 (T)$ we have
\begin{align*}
0 \leq \log ( | L_1 ( s_n ) |+3 ) \leq  2 \log | L_1 ( s_n ) |  \ll  \frac{1}{  \pi \delta^2}   \iint_{D_{\delta}( s_n)} \log | L_1 ( x+iy) | dx dy ,
\end{align*}
since $ \log | L_1 ( s)| $ is  subharmonic by Theorem 17.3 of \cite{R}. By Jensen's inequality applied to the convex function $ \varphi( x) = x^{2k}$, we have
\begin{align*}
\left( \log ( | L_1 ( s_n ) |+3 ) \right)^{2k} \ll &  C_5^k   \left(   \frac{1}{ \pi \delta^2}   \iint_{D_{\delta}(s_n)} \log | L_1 ( x+iy) | dx dy \right)^{2k}   \\
\ll &  C_5^k   \frac{1}{  \pi   \delta^2}   \iint_{D_{\delta}(s_n)}(  \log | L_1 ( x+iy) | )^{2k} dx dy    \\
\ll &  C_5^k   \frac{1}{ \delta^2}   \iint_{D_{R'}(\sigma_0+it_n)}(  \log | L_1 ( x+iy) | )^{2k} dx dy   
\end{align*}
for some $  C_5 >0$ and $R'=R+\delta$. Thus,
\begin{equation}\label{discrete sum bound 1}
 \sum_{n \in \mathcal{A}_3 (T) }  \left( \log ( |L_1 ( s_n ) | +3)        \right)^{2k} \ll C_5^k  \sum_{ n \in \mathcal{A}_3 (T)  } \frac{1}{ \delta^2}   \iint_{D_{R'}(\sigma_0+it_n)}(  \log | L_1 ( x+iy) | )^{2k} dx dy .
\end{equation} 

Let $\mathcal S_\ell  =\{n \in \mathcal{A}_3 (T)   :   n \equiv \ell  \pmod{(4 \lceil R' \rceil+2)} \}$.
If $ m,n  \in \mathcal S_\ell$ and $m \neq n$
then $|m-n| \geq 4 \lceil R' \rceil +2$; so that $|t_m-t_n|\ge2R'+1$. This implies that
$D_{R'}(\sigma_0+it_n) \cap D_{R'}(\sigma_0+it_m) =\emptyset$. Thus, since the disks are disjoint
we see that  
\begin{equation}\label{DiscreteToCont}
\begin{aligned}
& \sum_{n \in \mathcal{S}_\ell  } \frac{1}{ \delta^2}   \iint_{D_{R'}(\sigma_0+it_n)} (  \log | L_1 ( x+iy) | )^{2k} dx dy      \\
& \ll  G(T)^2  \int_{\sigma_0-R'}^{\sigma_0+R'} \int_{T-2R'-1}^{2T+2R'+1}   (  \log | L_1 ( x+iy) | )^{2k} dy dx  .
\end{aligned}
\end{equation}
By adding \eqref{DiscreteToCont} for all $\ell   \pmod{(4 \lceil R' \rceil+2)}$ and using  \eqref{discrete sum bound 1}, we see that
\begin{equation}\label{discrete sum bound 2}
 \sum_{n  \in \mathcal{A}_3 (T)  }  \left( \log ( |L_1 ( s_n ) | +3)        \right)^{2k} \ll  C_5^k G(T)^2 \int_{\sigma_0-R'}^{\sigma_0+R'}  \int_{T-2R'-1}^{2T+2R'+1}   (  \log | L_1 ( x+iy) | )^{2k}  dy dx.
\end{equation}
Note that $\sigma_0-R'= \sigma_0-r-2\delta= 1/2+ 1/(10G(T)).$ Let $Y=\exp(100G(T)\log\log T)$. Then it follows from Lemma \ref{lem : Dir Poly approx} that uniformly for  $x \geq \sigma_0-R'$ we have
\be\label{TruncLj2}
\log L_1 (x+iy )=R_{L_1, Y}(x+iy ) + O\left(\frac{1}{(\log T)^{3}}\right),
\ee
for all $y \in [T, 2T]$ except for a set of measure $\ll T \exp(- C_6  \log T/ G(T))$, for some constant $ C_6 >0$.  
Let $\mathcal{A}(T)$ be the set of points $y\in [T, 2T]$ for which \eqref{TruncLj2} holds and let $\mathcal{A}^c(T)$ be its complement in $[T-2R'-1, 2T+ 2R'+1]$. Then we have
$$ 
\textup{meas} ( \mathcal{A}^c(T)) \ll T \exp\left(- C_6 \frac{\log T}{G(T)}\right).
$$

We now split the inner integral on the right-hand side of \eqref{discrete sum bound 2} in two parts, the first over $\mathcal{A}(T)$ and the second over $\mathcal{A}^c(T)$. By \eqref{PL bound} we obtain 
\begin{equation}\label{Part2Ac}
\begin{aligned}
\int_{\sigma_0-R'}^{\sigma_0+R'} 
\int_{\mathcal{A}^c(T)} \left(\log| L_1( x+iy )| \right)^{2k}  dy dx & \ll \textup{meas}\big( \mathcal{A}^c(T)\big) ( C_7  \log T)^{2k}\\
& \ll T \exp\left(- \frac{ C_6 \log T}{2G(T)}\right)
\end{aligned}
\end{equation}
for some positive constant $ C_7 $, where the last estimate follows from our assumption on $k$. 

Furthermore, if $y \in \mathcal{A}(T)$, then for $x \ge \sigma_0-R'$ we have 
\begin{align*}
 \left(\log| L_1( x+iy )| \right)^{2k} 
& \ll  4^ k (   |R_{L_1 , Y}  (x+iy ) |^{2k} +1  ) 
\end{align*} 
by \eqref{TruncLj2}. 
 Thus, by Lemma \ref{lem:dir poly mmt 2} and Stirling's formula we obtain
\begin{equation}\label{Part1A}
\begin{aligned}
& \int_{\sigma_0-R'}^{\sigma_0+R'} 
\int_{\mathcal{A}(T)}  \left(\log| L_1( x+iy )| \right)^{2k}  dy dx \\
&  \ll 4^k \left(  \int_{\sigma_0-R'}^{\sigma_0+R'} 
\int_{T}^{2T} |R_{L_1, Y} ( x+iy )|^{2k} dy dx  + T\right)  \ll T ( C_8   k  \log\log T)^k
\end{aligned}
\end{equation}
for some positive constant $C_8 $. Inserting the estimates \eqref{Part2Ac} and \eqref{Part1A} in \eqref{discrete sum bound 2} gives 
$$  \sum_{n  \in \mathcal{A}_3 (T) }  \left( \log ( |L_1 ( s_n ) | +3)        \right)^{2k}   \ll  TG(T)^2   ( C_9  k \log\log T)^k$$
 for some constant $  C_9   > 0$. This with \eqref{A1T bound} and \eqref{A2T bound} proves the lemma.
\end{proof}

 Proposition \ref{MomentsLogF} follows from Lemmas \ref{LemLLR} and \ref{MeanMax}.


\section{Analysis of the random model :  Proofs of Theorem \ref{MainRandom}, Proposition \ref{MomentsLogR} and Lemma \ref{ConcentrationRandom}}\label{sec:rand}
Recall that
$$ \Lv(\sigma, \X)= \Big(\log |L_1(\sigma, \X)|, \dots, \log |L_J(\sigma, \X)|, \arg L_1(\sigma, \X), \dots, \arg L_J(\sigma, \X) \Big).$$
We define its partial sum
 $$ \Lv_q (\sigma, \X) = \Big(\log |L_{1,q} (\sigma, \X)|, \dots, \log |L_{J,q} (\sigma, \X)|, \arg L_{1,q} (\sigma, \X), \dots, \arg L_{J,q} (\sigma, \X) \Big) $$
for a positive integer $q$, where 
$$ \log L_{j, q} (\sigma, \X) :=  \sum_{ p \leq q } \sum_{k=1}^\infty \frac{ \beta_{L_j}(p^k) \X(p)^k }{p^{k\sigma }} .$$  
We also define 
\be \label{def Lvq tail}
  \Lv_{>q} (\sigma, \X) := \Lv ( \sigma, \X) - \Lv_q ( \sigma, \X) .
  \ee
For  a Borel set $  \B$ in $\mathbb R^{2J}$ and for~$\sigma = 1/2+ 1/G(T) $, we define
 \be \begin{split} \label{phi rand q def}
\Phrq  (\B ) & :=  \mathbb{P}(\Lv_q (\sigma , \X) \in \B) ,\\
\Phrqtail  (\B ) & :=  \mathbb{P}(\Lv_{>q} (\sigma , \X) \in \B) 
\end{split}\ee
and their Fourier transforms
\be \begin{split} \label{phi rand q trans def}
   \Phrqhat ( \xb, \yb )& := \int_{ \mathbb{R}^{2J} }   e^{ 2 \pi i (\xb \cdot \ub + \yb \cdot \vb) } d\Phrq ( \ub, \vb ) , \\
   \Phrqtailhat (\xb, \yb) & := \int_{ \mathbb{R}^{2J} }   e^{ 2 \pi i (\xb \cdot \ub + \yb \cdot \vb) } d\Phrqtail ( \ub, \vb ) 
 \end{split} \ee
 for $\xb = ( x_1 , \ldots, x_J ) \in \mathbb{R}^J$ and $ \yb = ( y_1, \ldots, y_J )  \in \mathbb{R}^J$. 
 
 \subsection{Upper bounds for the density functions and the Fourier transforms of $\Lv(\sigma, \X), \Lv_q(\sigma, \X)$, and $\Lv_{>q}(\sigma, \X)$} 
 In this subsection, we prove that the distribution functions of  $\Lv(\sigma, \X), \Lv_q(\sigma, \X)$, and $\Lv_{>q}(\sigma, \X)$ are absolutely continuous, and provide bounds for their density functions and Fourier transforms. These will be used to prove Proposition \ref{MomentsLogR} and Lemma \ref{ConcentrationRandom}. We start with the following lemma.
 
  \begin{lem}\label{lemma abs cont} 
  Let $A>0$ be a given real number.  Then, there exists a positive integer $q(A)$ such that   
$$      \Phrqhat (\xb , \yb )  \ll_{q,A}     (  1+  || \xb ||_2 + || \yb ||_2 )^{ - A}$$
for every $ q \geq q(A)$, where 
$$ || \xb ||_2   := \sqrt{ \sum_{j \leq J}| x_j|^2   } . $$
 Furthermore, for any positive integer $q$ we have
 $$    \Phrqtailhat (\xb , \yb ) \ll_{q, A}     (  1+  || \xb ||_2 + || \yb ||_2 )^{ -  A }.$$
Thus,  $\Phrq$ is absolutely continuous for sufficiently large $ q>0 $ and $ \Phrqtail $ is absolutely continuous for any $ q > 0 $.   
  \end{lem} 
\begin{proof}
The absolute continuity of $\Phrq$ and $ \Phrqtail $ follows from the inequalities in the lemma (see Section 3 in \cite{JW}). Thus, it is enough to prove these inequalities. 

We first define for any prime $p$
\be\label{DefPSI}
\varphi_{p, \sigma} (\xb, \yb):= \mathbb{E} \left[ \exp\left(2 \pi i  \sum_{j=1}^J \big( x_j \re\left( g_j \big(\X(p)p^{- \sigma}\big)\right)    +     y_j \im \left( g_j \big( \X(p)     p^{- \sigma}\big)\right) \big) \right)   \right],
\ee
where
$$g_j (  u ) =  \sum_{k=1}^\infty \beta_{L_j } (p^k)     u^k .$$ 
Then we find that 
\be \label{CharShortNot}\begin{split}
  \Phrqhat (\xb , \yb ) &= \mathbb{E} \bigg[ \exp\left(2 \pi i  ( x_1 , \ldots, x_J , y_1 , \ldots, y_J ) \cdot \Lv_q ( \sigma , \X) \right)                 \bigg] = \prod_{ p \leq q } \varphi_{p, \sigma} (\xb, \yb). 
\end{split}\ee

By Lemma 2.5 in \cite{Le2} there is a constant $C>0$ such that
\be \label{Lem2.5}| \varphi_{p, \sigma} (\xb, \yb)| \leq \frac{C p^{\sigma/2} }{( \sum_{j=1}^J ( x_j^2 + y_j^2 )  )^{1/4} }
\ee
if $  \big| \sum_{j=1}^J   \beta_{L_j} (p)  ( x_j - i y_j )  \big|^2   \geq \delta    \sum_{j=1}^J  ( x_j ^2 + y_j^2 )$ for some constant $   \delta >0$. Note that Lemma 2.5 in \cite{Le2} holds even for complex coefficients $a_j$ with minor modification. In that case the condition in the last line of the lemma should be $ | \sum_{j=1}^J a_j ( y_j -i  y_j') | \geq \delta || \yb ||_2 $. 

Let $q_1>0$ be a large positive integer to be chosen later and define a sequence $ q_n$ of integers inductively by $ q_{n+1} = 2^{q_n }$. We shall prove that given $ \ub$ and $ \vb$, there exists a prime $p$ in the interval $( q_{n-1} , q_n ] $ such that 
$$  \bigg| \sum_{j=1}^J \beta_{L_j} (p)   ( x_j - i y_j )  \bigg|^2   \geq  \frac12  ( \min_{j\leq J } \xi_j )    \sum_{j=1}^J  ( x_j ^2 + y_j^2 ) $$ 
holds. Suppose not. Then multiplying both sides by $1/p$ and summing over all primes  $p$ in $( q_{n-1}, q_n ] $ we have
$$ \sum_{ q_{n-1} < p < q_n }  \frac{ \big| \sum_{j=1}^J  \beta_{L_j} (p)   ( x_j - i y_j )  \big|^2 }{p}   \leq  \frac12  ( \min_{j\leq J } \xi_j )    \sum_{j=1}^J ( x_j ^2 + y_j^2 )  ( \log   q_{n-1}  +O( \log \log q_{n-1} )  ).$$
On the other hand by \eqref{SOC} we see that 
\begin{align*}
 \sum_{ q_{n-1} < p < q_n }  \frac{ \big| \sum_{j=1}^J   \beta_{L_j} (p)   ( x_j - i y_j )  \big|^2 }{p}&  = \sum_{j=1}^J ( x_j^2 + y_j^2 ) \bigg(  \sum_{ q_{n-1} < p \leq q_n }  \frac{ | \beta_{L_j} (p) |^2 }{p}  + O(1) \bigg) \\
& = \sum_{j=1}^J ( x_j^2 + y_j^2 )  (  \xi_j  \log  q_{n-1} + O(   \log \log q_{n-1}) ) \\
& \geq ( \min_{j \leq J} \xi_j )       \sum_{j=1}^J ( x_j^2 + y_j^2 ) ( \log q_{n-1} + O ( \log \log q_{n-1} )).
\end{align*}
This is a contradiction if $q_{n-1}$ is sufficiently large.  

Now, take $q= q_{m+1}$ where $m=\lfloor4A\rfloor+1$. Then using \eqref{Lem2.5} together with the trivial bound $|\varphi_{p, \sigma} (\xb, \yb)|\leq 1$ we obtain
$$
 |  \Phrqhat (\xb , \yb ) | \leq  \prod_{n=1}^m   \prod_{ q_n < p \leq q_{n +1}  } |\varphi_{p, \sigma} (\xb, \yb)|  \leq  \prod_{n=1}^m \frac{C q_{n+1}^{\sigma/2} }{  ( \sum_{j=1}^J  x_j^2 + y_j^2  )^{1/4} }  \leq     \frac{C_{q,m} }{  ( \sum_{j=1}^J x_j^2 + y_j^2  )^{A} }    
$$
for some  constant $C_{q,m} >0 $. This completes the proof.
 
 To prove the second inequality, we choose $\ell$ such that $ q_\ell>q $. Then for $m=\lfloor4A\rfloor+1$ we obtain similarly that
$$
 |  \Phrqtailhat (\xb , \yb ) | \leq  \prod_{n=\ell }^{\ell + m -1}    \prod_{ q_n  < p \leq q_{n +1}  } |\varphi_{p, \sigma} (\xb, \yb)|  
 \leq  \prod_{n=1}^m \frac{C q_{n+1}^{\sigma/2} }{  ( \sum_{j=1}^J x_j^2 + y_j^2  )^{1/4} }
 \leq    \frac{C'_{q,m,\ell}  }{  ( \sum_{j=1}^J x_j^2 + y_j^2  )^{A}}    
$$
 for some constant $ C'_{q,m,\ell} >0 $.

\end{proof}

 By Lemma \ref{lemma abs cont} and Section 3 in \cite{JW}, there is an integer $q>0$ such that both $\Phrq$ and $ \Phrqtail$ have continuous density functions, say $H_{q, T}  ( \ub, \vb) $ and $ H_{>q, T} (\ub, \vb) $, respectively. One can also see that $H_{>q, T} ( \ub, \vb)$ has partial derivatives of any order. Since $\Phrand =  \Phi_{>1,T}^{\mathrm{rand}}$, it follows that $\Phrand$ has a continuous density function which we shall denote throughout by $H_T ( \ub, \vb) $. These density functions are real valued and nonnegative.

\begin{lem} \label{lemma H_T bound 1} 
Let $0< \lambda <  ( 24J  \max_{j \leq J} \xi_j )^{-1} $ be a fixed real number. For all $\ub, \vb \in \mathbb{R}^J$ we have
$$ H_T ( \ub, \vb ) \ll_\lambda e^{  - \frac{\lambda}{\log G(T)}\sum_{j=1}^J ( u_j^2 + v_j^2 ) }.$$
\end{lem}
\begin{proof}
Let $q$ be a positive integer. By a standard convolution argument we have
\begin{align*}
 H_T ( \ub, \vb ) & = \int_{ \mathbb{R}^{2J}}   H_{q,T} ( \ub - \xb,  \vb - \yb ) d\Phrqtail (\xb,\yb) \\
 & = \int_{ \mathbb{R}^{2J}}   H_{q,T} ( \ub - \xb,  \vb - \yb ) H_{>q, T} ( \xb, \yb ) d\xb d\yb  
 \end{align*}
for $\ub, \vb \in \mathbb{R}^J$. Since 
$$
  | \Lv_q( \sigma , \X ) |^2   = \sum_{j=1}^J  | \log L_{j, q} ( \sigma , \X) |^2 
 \leq \sum_{j=1}^J \sum_{ p \leq q } \sum_{k=1}^\infty \frac{ | \beta_{L_j } ( p^k) | }{p^{k/2}} :=  R_q^2   ,
$$
  we have
$ H_{q,T} ( \xb, \yb ) = 0 $
for $ \sum_{j=1}^J ( x_j^2 + y_j ^2 ) >  R_q^2 $. Let 
$$ B_q (\ub, \vb)  := \{ ( \xb, \yb) \in \mathbb{R}^{2J} :         \sum_{j=1}^J (x_j - u_j)^2 + (y_j -v_j) ^2  \leq  R_q^2 \} $$
be the $2J$ dimensional ball of radius $R_q$ centered at $ (\ub, \vb ) $, then we see that
\begin{align*}
  H_T ( \ub, \vb )  &  = \int_{   B_q (\ub, \vb)  }   H_{q,T} ( \ub - \xb,  \vb - \yb )  H_{>q, T} ( \xb, \yb ) d\xb d\yb \\
  & \leq \Big(\sup_{(\xb, \yb)\in \mathbb{R}^{2J}}H_{q,T} ( \xb, \yb) \Big)  \Phrqtail (   B_q  (\ub, \vb) ).
 \end{align*}
Since the measure $ \tilde{\Phi}_{q,\sigma}^{\mathrm{rand}} ( \B) := \P [ \Lv_q ( \sigma, \X ) \in \B ] $ and its density function $ \tilde{H}_{q,\sigma}( \xb, \yb)$ depend continuously on $\sigma \geq 1/2 $, we see that
$$  \sup_{T \geq T_0} \sup_{(\xb, \yb)\in \mathbb{R}^{2J}}H_{q,T} ( \xb, \yb) \leq   M_q:=\sup_{ 1/2 \leq \sigma \leq 2/3} \sup_{ \xb, \yb \in \mathbb{R}^J } \tilde{H}_{q,\sigma} ( \xb, \yb) < \infty  $$
for a sufficiently large constant $T_0 > 0$.  
Hence, we deduce that
\be \label{rand inequality 1}
 H_T ( \ub, \vb )   \leq M_q  \Phrqtail (B_q  (\ub, \vb) ) . 
 \ee
 Thus, it remains to find an upper bound for $\Phrqtail ( B_q (\ub, \vb) ) $. 
 
 First, we remark that if $(\xb, \yb)\in  B_q (\ub, \vb)$ and $||(\ub, \vb)||_2 \geq 2R_q$ then 
\be\label{Lower_xbyb}
 ||(\xb, \yb)||_2 \geq \frac{1}{2}||(\ub, \vb)||_2 , 
 \ee
 since otherwise $ \frac12 || ( \ub, \vb)||_2  < || ( \ub, \vb  ) ||_2  - || ( \xb, \yb ) ||_2  \leq || ( \ub, \vb  ) -  ( \xb, \yb ) ||_2 \leq  R_q $ which contradicts our assumption. Let $\lambda< (24 J\max_{j\leq J} \xi_j)^{-1}$ be a positive real number. Then it follows from \eqref{Lower_xbyb} that for $(\ub, \vb)\in \mathbb{R}^{2J}$ such that $||(\ub, \vb)||_2 \geq 2R_q$ we have
 \be \label{rand inequality 3}\begin{split}
   \Phrqtail &(   B_q  (\ub, \vb) )  = \int_{ B_q  (\ub, \vb)  }   H_{>q,T} ( \xb, \yb ) d\xb d\yb \\
   & \leq  e^{  - \frac{\lambda}{\log G(T)}\sum_{j=1}^J ( u_j^2 + v_j^2 ) }  \int_{ B_q  (\ub, \vb)  }   e^{ \frac{4\lambda}{\log G(T)} \sum_{j=1}^J ( x_j^2 + y_j ^2 )}  H_{>q,T} ( \xb, \yb ) d\xb d\yb\\
   & \leq  e^{  - \frac{\lambda}{\log G(T)}\sum_{j=1}^J ( u_j^2 + v_j^2 ) }  \int_{  \mathbb{R}^{2J}   }   e^{ \frac{4\lambda}{\log G(T)} \sum_{j=1}^J ( x_j^2 + y_j ^2 )}  H_{>q,T} ( \xb, \yb ) d\xb d\yb.
   \end{split}  
   \ee
To complete the proof we shall establish that for any real number $   0<\lambda' <     ( 6J  \max_{j \leq J} \xi_j  )^{-1} $ we have  
\be \label{rand inequality 2}
\int_{ \mathbb{R}^{2J} } e^{  \frac{ \lambda' }{ \log G(T) } \sum_{j=1}^J (x_j^2 + y_j^2 ) }   H_{>q,T} ( \xb, \yb) d\xb d\yb    = O_{q, \lambda'}  (1)
\ee
 as $T \to \infty$. 
 Indeed, assuming \eqref{rand inequality 2} we obtain by \eqref{rand inequality 1} and \eqref{rand inequality 3} that 
 $$  H_T ( \ub, \vb )  \ll_q e^{  - \frac{\lambda}{\log G(T)}\sum_{j=1}^J ( u_j^2 + v_j^2 ) },$$
 for $||(\ub, \vb)||_2 \geq 2R_q$. Therefore, choosing $q$ to be large but fixed we deduce that for all $(\ub, \vb)\in \mathbb{R}^{2J}$ we have 
 $$H_T ( \ub, \vb )  \ll e^{  - \frac{\lambda}{\log G(T)}\sum_{j=1}^J ( u_j^2 + v_j^2 )}$$
where the implicit constant is absolute. 
 
We now proceed to establish \eqref{rand inequality 2}. Our proof is basically the same as the second part of the proof of Proposition 2.2 in  \cite{Le4}.
First, note that 
\begin{multline*}
 \int_{ \mathbb{R}^{2J} } e^{  \frac{ \lambda' }{ \log G(T) } \sum_{j=1}^J (x_j^2 + y_j^2 ) }   H_{>q,T} ( \xb, \yb) d\xb d\yb \\
  = \E \bigg[ \exp \bigg(    \frac{ \lambda' }{ \log G(T) } \sum_{j=1}^J \bigg|  \sum_{ p > q } \sum_{k=1}^\infty    \frac{   \beta_{L_j}(p^k) \X(p)^k }{p^{k\sigma}}   \bigg|^2 \bigg) \bigg] . 
 \end{multline*}
Since  by \eqref{prime sum bounds} we have
\begin{align*}
 \sum_{ p > q } \sum_{k=1}^\infty    \frac{   \beta_{L_j}(p^k) \X(p)^k }{p^{k\sigma}} & =  \sum_{ p  } \sum_{k=1}^\infty    \frac{   \beta_{L_j}(p^k) \X(p)^k }{p^{k\sigma}}  +O_q (1) \\
 & =      \sum_{ p    }    \frac{   \beta_{L_j}(p ) \X(p)  }{p^{ \sigma}} +    \sum_{ p    }   \frac{   \beta_{L_j}(p^2) \X(p)^2 }{p^{2\sigma}} +     O_q (1) ,   
 \end{align*}
 we see that 
\begin{align*}
& \E \bigg[ \exp \bigg(    \frac{ \lambda'}{ \log G(T) } \sum_{j=1}^J  \bigg|  \sum_{ p > q } \sum_{k=1}^\infty    \frac{   \beta_{L_j}(p^k) \X(p)^k }{p^{k\sigma}}   \bigg|^2 \bigg) \bigg]  \\
& \ll_q
\E \bigg[ \exp \bigg(    \frac{ 3 \lambda'}{ \log G(T) } \sum_{j=1}^J \bigg|    \sum_{ p    }    \frac{   \beta_{L_j}(p ) \X(p)  }{p^{ \sigma}}  \bigg|^2 + \frac{ 3\lambda'}{ \log G(T) } \sum_{j=1}^J  \bigg|    \sum_{ p    }   \frac{   \beta_{L_j}(p^2) \X(p)^2 }{p^{2\sigma}}   \bigg|^2   \bigg) \bigg] .
\end{align*}
By H\"{o}lder's inequality, the above is
\begin{align*}
\leq \prod_{j=1}^J \E \bigg[ \exp \bigg(    \frac{ 6J  \lambda'}{ \log G(T) }   \bigg|    \sum_{ p    }    \frac{   \beta_{L_j}(p ) \X(p)  }{p^{ \sigma}}  \bigg|^2    \bigg) \bigg]^{\frac1{2J}} \E \bigg[ \exp \bigg(      \frac{ 6J \lambda'}{ \log G(T) }  \bigg|    \sum_{ p    }   \frac{   \beta_{L_j}(p^2) \X(p)^2 }{p^{2\sigma}}   \bigg|^2   \bigg) \bigg]^{ \frac1{2J}} .
\end{align*}
By inequality (18.8) of \cite{J} (which is an easy application of Parseval's identity), the above is
$$ \leq  \prod_{j=1}^J  \bigg(  1-   \frac{ 6J  \lambda' }{ \log G(T) }  \sum_p  \frac{  | \beta_{L_j}(p ) |^2  }{p^{ 2  \sigma}}          \bigg)^{ - \frac1{2J}} \bigg(       1-   \frac{ 6J \lambda'}{ \log G(T) }  \sum_p  \frac{  | \beta_{L_j}(p^2  ) |^2  }{p^{ 4  \sigma}}         \bigg)^{ -\frac1{2J}} . $$
By \eqref{prime sum bounds}, we see that
$$ \sum_p  \frac{  | \beta_{L_j}(p^2  ) |^2  }{p^{ 4  \sigma}}        \leq \sum_p  \frac{  | \beta_{L_j}(p^2  ) |^2  }{p^{ 2}}        < \infty.$$
Furthermore, it follows from \eqref{SOC2} that 
$$ \sum_p  \frac{  | \beta_{L_j}(p ) |^2  }{p^{ 2  \sigma}} = \xi_j \log G(T) +O(1) .$$
Hence, we obtain
$$
\int_{ \mathbb{R}^{2J} } e^{  \frac{ \lambda' }{ \log G(T) } \sum_{j \leq J} (x_j^2 + y_j^2 ) }   H_{>q,T} ( \xb, \yb) d\xb d\yb   \ll_q \prod_{j=1}^J  \bigg( 1 -    \frac{ 6J\lambda'  ( \xi_j \log G(T) + O(1) )}{\log G(T)} \bigg)^{- \frac1{2J}} 
   \ll_{q,  \lambda'  } 1
$$
  since $ 0< \lambda'  <   ( 6J  \max_{j \leq J} \xi_j  )^{-1} $. This completes the proof of \eqref{rand inequality 2} and hence the result.

\end{proof}
From the above lemma, we deduce the following proposition and Lemma \ref{ConcentrationRandom}.
\begin{proof}[Proof of Proposition \ref{MomentsLogR}]
By Lemma \ref{lemma H_T bound 1} we see that
$$
 \E  \bigg( \bigg| \log \bigg|\sum_{j=1}^J b_j L_j(\sigma, \X)\bigg|\bigg|^{2k} \bigg)  = \int_{ \mathbb{R}^{2J}}  \bigg| \log \bigg| \sum_{j =1}^J b_j e^{u_j + i v_j }  \bigg| \bigg|^{2k} H_T ( \u, \v)  d\u d \v . 
$$
Furthermore, it follows from Lemma 2.3 in \cite{Le4} that there exists a constant $C>0$ such that for any $M>0$ we have
\be \label{eqn Lemma 2.3 Le4}
   \int_{ \mathbb{R}^{2J}}  \bigg|  \log \bigg| \sum_{j =1}^J b_j e^{u_j + i v_j }  \bigg| \bigg|^{2k}  e^{  - \frac{1}{M}\sum_{j=1}^J( u_j ^2 + v_j ^2 ) }  d\u d \v \ll M^J  (Ck)^k   (M+ k )^k.
   \ee
Applying this result with $M=\log G(T)/\lambda $ completes the proof. 

\end{proof}

\begin{proof}[Proof of Lemma \ref{ConcentrationRandom}]

First, using that $H_T(\ub, \vb)$ is uniformly bounded in $\ub, \vb$ we obtain
\begin{equation}\label{Concentration1}
 \begin{aligned}
 &\pr\left(\Lv(\sigma, \X) \in [-M, M]^{2J} \ \text{ and } \ R< \bigg|\sum_{j=1}^J b_j L_j(\sigma, \X) \bigg|<R+\varepsilon\right) \\
 &\ll \int_{\substack{\ub\in [-M, M]^{J}, \vb \in [-M, M]^{J}\\ R<|\sum_{j=1}^J b_j e^{u_j+iv_j}|<R+\varepsilon}} d\ub d\vb   \ll M^{J} \int_{\substack{\ub\in [-M, M]^{J}, \vb \in [0, 2\pi]^{J}\\ R<|\sum_{j=1}^J b_j e^{u_j+iv_j}|<R+\varepsilon}} d\ub d\vb,\\
 \end{aligned}
 \end{equation}
where the last estimate is obtained by splitting the range of each $v_j$ into intervals of the form $[2k\pi, (2k+1)\pi]$ and using that $e^{iv_j}$ is periodic of period $2\pi$. 
By the change of variables $r_1=e^{u_1}$, the last integral in \eqref{Concentration1} equals 
\begin{equation}\label{Concentration2}
 \int_{[0 ,2\pi]^{J-1}}\int_{ [-M,M ]^{J-1}}  \left(\int_{\mathcal{R}_0 } \frac{dr_1}{r_1} dv_1\right) d u_2 \dots d u_Jd v_2 \dots d v_J, 
\end{equation}
where $\mathcal{R}_0 := \{ ( r_1 , v_1) \in [e^{-M}, e^M] \times[0, 2\pi] :  R<|b_1 r_1 e^{iv_1}+ \sum_{j=2}^J b_j e^{u_j+iv_j}|<R+\varepsilon \} $.
We shall now bound the inner integral by changing the polar coordinates $(r_1, v_1)$ to cartesian coordinates $x, y$, defined by $x=r_1 \cos(v_1)$ and $y=r_1 \sin(v_1)$. Let $Z= \sum_{j=2}^J \frac{b_j}{b_1} e^{u_j+iv_j}$.  
The set $\{(x, y)\in \mathbb{R}^2 : R_1 <|x+iy+Z|< R_2\}$ corresponds to the annulus of radii $R_1, R_2$ centered at $-Z$ with volume $\pi (R_2^2-R_1^2)$. 
Thus, we have 
\begin{multline*}
\int_{\mathcal{R}_0 } \frac{dr_1}{r_1} dv_1
 = \int_{\substack{e^{-M}\leq \sqrt{x^2+y^2}\leq e^M \\ R/|b_1|<|x+iy +Z|<(R+\varepsilon)/|b_1|}} \frac{dxdy}{x^2+y^2}\\
 \leq e^{2M} \int_{R/|b_1|<|x+iy +Z|<(R+\varepsilon)/|b_1|} dxdy  = \pi e^{2M}\frac{(R+\varepsilon)^2-R^2}{|b_1|^2} \ll e^{2M} (R\varepsilon+ \varepsilon^2) .
\end{multline*}
  Inserting this estimate in \eqref{Concentration2} and combining it with  \eqref{Concentration1} we deduce
$$ 
\pr\left(\Lv(\sigma, \X) \in [-M, M]^{2J} \ \text{ and } \ R<\bigg|\sum_{j=1}^J b_j L_j(\sigma, \X)\bigg|<R+\varepsilon\right) \ll M^{2J-1} e^{2M} (R\varepsilon+ \varepsilon^2).
$$

\end{proof}

    
  \subsection{Asymptotic formulas for  $\Phrhat ( \xb, \yb )$ and $H_T(\ub, \vb)$}

In order to prove Theorem \ref{MainRandom} we need an asymptotic formula for the density function $H_T(\ub, \vb)$ that is valid for a certain set of $(\ub, \vb)$. To this end we prove the following result.
\begin{lem}\label{lem asymp Phrhat}
Let $ \xi_{\mathrm{min}} =\min_{j \leq J} \xi_j $. Then we have
\be \label{lem asymp Phrhat eqn 1}  
 |  \Phrhat (\xb , \yb )|    \leq e^{  -     \pi^2 \xi_{\mathrm{min}}  ( ||\xb||_2^2 + ||\yb||_2^2 ) (  \frac12  \log  G(T) + O(1)) } 
  \ee
for $  ||\xb ||_2^2 + || \yb ||_2^2  \leq   e^{2 \sqrt{G(T)}}$. Moreover, 
there exists a constant $c_4>0$ such that
\be \label{lem asymp Phrhat eqn 2}
 \Phrhat ( \xb, \yb ) =  e^{ \mathcal{B}_{2, \sigma} (\zb)  } \bigg( 1 +    \sum_{m=3}^5  \mathcal{B}_{m, \sigma} (\zb )  + O  ( || \zb ||_2^6 )\bigg) 
\ee
holds for 
$$ \zb := \xb + i \yb = ( x_1 + i y_1 , \ldots, x_J + i y_J ) \in \mathbb{C}^J $$
 and $ || \zb ||_2 \leq c_4 $, where each $\mathcal{B}_{m, \sigma}(\zb)$ is a homogeneous polynomial in $\zb$ and $\overline{\zb}$ of degree $m$,   
\be \label{lem asymp Phrhat eqn 3} \begin{split}
 \mathcal{B}_{2,\sigma}(\zb) =&  -  \pi^2   \log G(T)   \sum_{j=1}^J \xi_j  (x_j^2 + y_j^2 )  \\
 & + \sum_{j_1, j_2 \leq J } \bigg( C_{j_1, j_2  }  + O \bigg( \frac{\log G(T)}{ G(T)} \bigg) \bigg)   (x_{j_1 } - i y_{j_1} )( x_{j_2 } + i y_{j_2 })    , 
\end{split} \ee
for some constants $ C_{j_1, j_2 }$ and 
\be \label{lem asymp Phrhat eqn 4}
 \mathcal{B}_{m, \sigma} (\zb)  =  \mathcal{B}_{m, 1/2}(\zb) + O \bigg(   \frac{ || \zb||_2^m}{G(T)} \bigg) 
 \ee
for $ m = 3, 4, 5$.  
\end{lem}

\begin{proof}
Recall from \eqref{CharShortNot} that 

\be\label{ProductChar}
\Phrhat (\xb , \yb )  = \prod_{p} \varphi_{p, \sigma} (\xb, \yb), 
\ee
where $\varphi_{p, \sigma} (\xb, \yb)$ is defined in \eqref{DefPSI}. Now, using \eqref{DefPSI} and expanding the exponential we obtain
\begin{align*}
\varphi_{p, \sigma} (\xb, \yb)&=\mathbb{E} \bigg[ \exp\bigg(2 \pi i  \sum_{j=1}^J  ( x_j \re ( g_j  (\X(p)p^{- \sigma} ) )    +     y_j \im  ( g_j  ( \X(p)     p^{- \sigma} ) )  ) \bigg)   \bigg]\\
&=  \sum_{ \kb , \lb \in (\mathbb{Z}_{ \geq 0 })^J }  \frac{  (\pi i )^{\mathcal{K}( \kb + \lb ) } \overline{\zb}^{ \kb }  \zb^{  \lb } }{\kb! \lb!} \E \bigg[  \prod_{j=1}^J  g_j \bigg( \frac{\X(p) }{p^{\sigma}} \bigg)^{k_j }  \overline{ g_j \bigg( \frac{\X(p) }{p^{\sigma}}} \bigg)^{\ell_j } \bigg],
\end{align*}
where for $ \kb = ( k_1 , \ldots, k_J ) \in (\mathbb{Z}_{ \geq 0 })^J$, we define $\mathcal{K} ( \kb) := k_1 + \cdots + k_J $, $ \kb! :=k_1 ! \cdots k_J ! $, 
$\zb = \xb + i \yb   $,
 $\overline{\zb} = \xb - i \yb $ and $ \zb^\kb :=   z_1^{k_1 }  \cdots  z_J ^{ k_J} $. Let
$$ A_{p, \sigma } (\kb,\lb) :=  \E \bigg[  \prod_{j=1}^J  g_j \bigg( \frac{\X(p) }{p^{\sigma}} \bigg)^{k_j }  \overline{ g_j \bigg( \frac{\X(p) }{p^{\sigma}}} \bigg)^{\ell_j } \bigg].$$
Since $ A_{p, \sigma}(0,0)= 1$, and $ A_{p, \sigma}(0, \kb)= A_{p,\sigma}(\kb, 0)= 0 $ for $ \kb \neq 0 $, we deduce that 
\be\label{PSIA}
 \varphi_{p, \sigma} (\xb, \yb)= 1+ R_{p, \sigma}(\zb),
 \ee
where 
$$ R_{p, \sigma}(\zb) := \sum_{\kb \neq 0}  \sum_{\lb \neq 0 }    \frac{  (\pi i )^{\mathcal{K}( \kb + \lb ) } \overline{\zb}^{ \kb }  \zb^{  \lb } }{\kb! \lb!} A_{p,\sigma } (\kb,\lb).$$  
We now proceed to bound the sum $ R_{p, \sigma}(\zb)$ in a certain range of $\zb$ and $p$. By \eqref{beta bound 1} and \eqref{beta bound 2}, we see that
$$   g_j \bigg(  \frac{\X(p)}{p^{\sigma}} \bigg)   = \frac{ \beta_{L_j}(p)\X(p) }{p^{\sigma}} + O \bigg(  \frac{   \sum_{i=1}^d | \alpha_{j,i} (p) |^2 }{p^{2 \sigma}} \bigg)   = O \bigg( \frac{ 1 }{p^{  \sigma}} \sqrt{  \sum_{i=1}^d | \alpha_{j,i} (p) |^2} \bigg) = O \bigg(   \frac{1}{p^{1/2-\theta}} \bigg)   . 
$$
Hence, there exists a constant $ c_0 >0$ such that both 
\be \label{gj estimate} \begin{split}
\bigg|  g_j \bigg(  \frac{\X(p)}{p^{\sigma}} \bigg)   \bigg|& \leq c_0  \frac{ 1 }{p^{  \sigma}} \sqrt{  \sum_{i=1}^d | \alpha_{j,i} (p) |^2} ,\\
\bigg|  g_j \bigg(  \frac{\X(p)}{p^{\sigma}} \bigg)   \bigg| & \leq c_0   \frac{1}{p^{1/2-\theta}} 
\end{split} \ee
hold for every prime $p$ and every $ j \leq J$. Thus, we obtain
\be \label{Apsigmaxy 1} \begin{split}
 |R_{p, \sigma}(\zb) | &\ll   \sum_{\kb \neq 0}  \sum_{\lb \neq 0 }   \frac{1}{\kb! \lb!}     \prod_{j=1}^J  \bigg(\frac{ c_0 \pi ||\zb||_2  }{p^{1/2 }} \sqrt{  \sum_{i=1}^d | \alpha_{j,i} (p) |^2}      \bigg)^{k_j + \ell_j  }   \\
 & \ll     \frac{  || \zb||_2^2 }{ p } \sum_{j=1}^J   \sum_{i=1}^d | \alpha_{j,i} (p) |^2 \ll     \frac{  || \zb||_2^2 }{ p^{ 1- 2 \theta}}
\end{split}\ee 
provided that $ || \zb ||_2 \ll p^{1/2-\theta}$.
This implies that in the range $||\zb ||_2\leq Y $ with $ Y :=\exp (   \sqrt{ G(T)} ) $, there is a constant $c_1 > 0 $ such that
$
|R_{p, \sigma}(\zb) | \leq \frac12
$
  holds for all primes  $ p \geq  c_1 Y^{1/(1/2-\theta)} $.
Therefore, using that $|\varphi_{p, \sigma} (\xb, \yb)|\leq 1$ for all primes $p$, together with \eqref{ProductChar} and \eqref{PSIA} we have
\be \label{Phrhat bound 1} \begin{split}
 |  \Phrhat (\xb , \yb )|  & \leq \prod_{ p \geq c_1 Y^{c(\epsilon)}   }  |1+ R_{p, \sigma}(\zb) |  \\
 &   \leq \bigg| \exp \bigg( \sum_{ p \geq c_1  Y^{c(\epsilon)} } R_{p, \sigma}(\zb) + O \big(  \sum_{ p \geq c_1    Y^{c(\epsilon)}   }| R_{p, \sigma}(\zb)|^2  \big) \bigg) \bigg|,  
\end{split} \ee
 for $  ||\zb ||_2   \leq Y$ and any $\epsilon>0 $ fixed, where 
 $$ c(\epsilon):= (1+\epsilon)/(1/2-\theta) . $$

The second $p$-sum in \eqref{Phrhat bound 1} is
$$
\sum_{ p \geq c_1  Y^{c(\epsilon)}   }| R_{p, \sigma}(\zb)|^2   
  \ll ||\zb ||_2^4 \sum_{ p \geq c_1   Y^{c(\epsilon)}  }\frac{1}{p^{2-2\theta}}  \sum_{j=1}^J \sum_{i =1}^d  | \alpha_{j,i}(p)|^2  \ll  ||\zb ||_2^4  Y^{-2-\epsilon}   \ll    ||\zb ||_2^2 Y^{ -\epsilon} 
$$
for $   ||\zb ||_2   \leq Y $ and any $\epsilon>0$ fixed by   \eqref{Apsigmaxy 1}, assumption A3 and partial summation.
The first $p$-sum in \eqref{Phrhat bound 1} is
\begin{align*}
 \sum_{ p \geq c_1   Y^{c(\epsilon)}  } R_{p, \sigma}(\zb) = &  \sum_{ p \geq c_1   Y^{c(\epsilon)}  }  \sum_{\kb \neq 0}  \sum_{\lb \neq 0 }   \frac{(\pi i )^{\mathcal{K}( \kb + \lb ) } \overline{\zb}^{ \kb }  \zb^{  \lb } }{\kb! \lb!} A_{p, \sigma}(\kb, \lb) \\
 =& - \pi^2 \sum_{j_1, j_2 \leq J} \overline{ z_{j_1}}  z_{j_2} \sum_{ p \geq c_1   Y^{c(\epsilon)}  } \E \bigg[    g_{j_1} \bigg( \frac{\X(p) }{p^{\sigma}} \bigg)  \overline{  g_{j_2} \bigg( \frac{\X(p) }{p^{\sigma} } \bigg) }  \bigg] \\
 & + O \bigg(  \sum_{ p \geq c_1   Y^{c(\epsilon)}   } \sumstar_{  \kb, \lb      }   \frac{   (\pi ||z||_2)^{\mathcal{K}( \kb + \lb ) } }{\kb! \lb!} |A_{p, \sigma}(\kb, \lb)   | \bigg) ,
 \end{align*}
 where the $*$-sum is over $\kb, \lb \in ( \mathbb{Z}_{\geq 0 } )^J$ with $ \kb \neq 0 $, $\lb \neq 0 $ and $ \mathcal{K}(\kb + \lb ) \geq 3 $.
By \eqref{gj estimate}, assumption A3 and partial summation, the above $O$-term is 
\begin{align*}
 &  \ll   \sumstar_{  \kb, \lb   }   \frac{   (c_0 \pi ||\zb||_2 )^{\mathcal{K}( \kb + \lb ) } }{\kb! \lb!}\sum_{ p \geq c_1  Y^{c(\epsilon)}  } \frac{1}{ p^{ 1 + ( \mathcal{K}(\kb+\lb)-2 )(1/2-\theta)}  }  \sum_{j=1}^J  \sum_{i =1}^d |\alpha_{j,i}(p)|^2  \\
 & \ll   \sumstar_{  \kb, \lb    }   \frac{  (c_0 \pi ||\zb||_2 )^{\mathcal{K}( \kb + \lb )}   }{\kb! \lb!} Y^{ 2 - \mathcal{K}(\kb + \lb) - \epsilon /2   } 
 \ll Y^{-1 - \epsilon/2} ||\zb||_2^3  \leq    Y^{-\epsilon/2 } ||\zb||_2^2  
 \end{align*}
for $ ||\zb ||_2 \leq Y $. Thus we derive
\begin{multline*}
 |  \Phrhat (\xb , \yb )|  \\
    \leq \bigg| \exp \bigg(   - \pi^2 \sum_{j_1, j_2 \leq J} \overline{ z_{j_1}}  z_{j_2}  \sum_{ p \geq c_1  Y^{c(\epsilon)}   }   \E \bigg[     g_{j_1} \bigg( \frac{\X(p) }{p^{\sigma}} \bigg)  \overline{  g_{j_2} \bigg( \frac{\X(p) }{p^{\sigma} } \bigg) } \bigg]   + O \bigg(   Y^{-\epsilon/2 } ||\zb||_2^2  \bigg)  \bigg) \bigg|   
\end{multline*}
 for $   ||\zb||_2   \leq Y $ and for any $\epsilon>0$ fixed. 
 Moreover, by \eqref{prime sum bounds} and \eqref{SOCTail}, the sum above equals
\begin{align*}
   & \sum_{j_1, j_2 \leq J}  \overline{ z_{j_1}}  z_{j_2}  \sum_{ p \geq c_1   Y^{c(\epsilon)}  }    \sum_{k=1}^\infty   \frac{  \beta_{L_{j_1}}(p^k)    \overline{  \beta_{L_{j_2}}(p^k)} }{p^{2k\sigma}}  \\
  = & \sum_{j_1, j_2 \leq J}  \overline{ z_{j_1}}  z_{j_2} \sum_{ p \geq c_1   Y^{c(\epsilon)}  }       \frac{  \beta_{L_{j_1}}(p)  \overline{ \beta_{L_{j_2}}(p)}  }{p^{2 \sigma}}  +O \big( ||\zb||_2^2  \big)   
  =  \frac12  \log G(T)\sum_{j=1}^J \xi_j |z_j |^2  + O \big( ||\zb||_2^2  \big) .
\end{align*}
  Therefore, we deduce that
\be \label{Phrhat bound 3} \begin{split}
 |  \Phrhat (\xb , \yb )|  &   \leq   e^{   -  \left( \frac{\pi^2 }{2}   \log G(T) +O(1)\right) \sum_{j=1}^J  \xi_j |z_j|^2     } \leq e^{  -     \pi^2 \xi_{\mathrm{min}}  ||\zb||_2^2 ( \frac12 \log G(T) + O(1)) } 
\end{split} \ee
for $  ||\zb ||_2 \leq Y  $ where $ Y = e^{\sqrt{G(T)}}$. This proves \eqref{lem asymp Phrhat eqn 1}.

Next we find an asymptotic formula for $ \Phrhat $. By \eqref{Apsigmaxy 1}, there is a constant $c_4>0$ such that 
$ |R_{p, \sigma}(\zb) | \leq \frac12 $
for $ ||\zb ||_2 \leq c_4 $ and for every prime $p$.
Hence, it follows from  \eqref{ProductChar} and \eqref{PSIA} that
$$
\Phrhat ( \xb, \yb ) =   \exp \bigg(  \sum_p R_{p, \sigma}(\zb)  - \frac12 \sum_p R_{p, \sigma}(\zb)^2 + O \bigg( \sum_p |R_{p, \sigma}(\zb)|^3 \bigg)  \bigg).
$$
The $O$-term above is 
$$
  \ll ||\zb||_2^6 \sum_p   \frac{1}{p^{1+ 2(1 - 2 \theta) }}  \sum_{j=1}^J \sum_{i =1}^d  | \alpha_{j,i}(p)|^2 \ll  ||\zb||_2^6   
$$
for $||\zb||_2     \leq c_4$  by \eqref{Apsigmaxy 1}, assumption A3  and  partial summation. 
We observe that  the sum $\sum_p R_{p, \sigma}(\zb)  - \frac12 \sum_p R_{p, \sigma}(\zb)^2 $ has a power series representation in $z_1 , \ldots, z_J , \bar{z}_1 , \ldots , \bar{z}_J $ without a constant term. Let $ B_{\sigma}(\kb, \lb)$ be the coefficient of $\overline{\zb}^{\kb} \zb^{\lb}$ in this sum. Then we have
\begin{multline*}
  \sum_{ \kb \neq 0 }  \sum_{ \lb  \neq 0 } B_{\sigma}(\kb, \lb)\overline{\zb}^{\kb} \zb^{\lb} 
   =  \sum_{\kb \neq 0}  \sum_{\lb \neq 0 }   \frac{  (\pi i )^{\mathcal{K}( \kb + \lb ) } \overline{\zb}^{ \kb }  \zb^{  \lb } }{\kb! \lb!}  \sum_p A_{p, \sigma}(\kb, \lb)\\
    - \frac12  \sum_{\kb_1, \kb_2 \neq 0}  \sum_{\lb_1, \lb_2 \neq 0 }   \frac{  (\pi i )^{\mathcal{K}( \kb_1 + \lb_1  + \kb_2+ \lb_2  ) } \overline{\zb}^{ \kb_1+ \kb_2 }  \zb^{  \lb_1 + \lb_2 } }{\kb_1! \kb_2!\lb_1!\lb_2!} \sum_p  A_{p, \sigma} (\kb_1, \lb_1) A_{p, \sigma} (\kb_2, \lb_2).
  \end{multline*}
Therefore, if $\mathcal{K} ( \kb + \lb ) = 2 $ or $ 3$, then 
 $$ B_\sigma (\kb, \lb )    =   \frac{  (\pi i )^{\mathcal{K}( \kb + \lb ) } }{ \kb! \lb!}  \sum_p A_{p, \sigma}(\kb, \lb),$$
while in the case $\mathcal{K} ( \kb+\lb )\geq 4$, we have
 $$ B_{\sigma}(\kb, \lb)   =  \frac{(\pi i )^{\mathcal{K}( \kb + \lb ) } }{ \kb! \lb!} \sum_p A_{p, \sigma}(\kb, \lb) -  \frac{1}{2}   \sum_{ \substack{ \kb_1 , \kb_2 \neq 0 \\ \kb_1+\kb_2= \kb } } \sum_{ \substack{ \lb_1, \lb_2\neq 0  \\ \lb_1 + \lb_2 = \lb}  }  \frac{  (\pi i )^{\mathcal{K}( \kb + \lb    ) }}{\kb_1! \kb_2!\lb_1!\lb_2!} \sum_p  A_{p, \sigma} (\kb_1, \lb_1)A_{p, \sigma} (\kb_2, \lb_2)   $$
For $\mathcal{K} ( \kb+ \lb ) \geq 4$, we have
\begin{align*}
   B_{\sigma}& (\kb, \lb)   \ll     \frac{\pi^{\mathcal{K}( \kb + \lb ) } }{ \kb! \lb!} \sum_p | A_{p, \sigma}(\kb, \lb) | +   \sum_{ \substack{ \kb_1 , \kb_2 \neq 0 \\ \kb_1+\kb_2= \kb } } \sum_{ \substack{ \lb_1 , \lb_2 \neq 0  \\ \lb_1 + \lb_2 = \lb}  }  \frac{   \pi^{\mathcal{K}( \kb + \lb    ) }   }{\kb_1! \kb_2!\lb_1!\lb_2!} \sum_p  | A_{p, \sigma} (\kb_1, \lb_2) |  | A_{p, \sigma} (\kb_2, \lb_2) |\\
  & \ll  \frac{ (c_0  \pi)^{\mathcal{K}( \kb + \lb ) } }{ \kb! \lb!} \sum_p \frac{\max_{j} \sum_i |\alpha_{j, i }(p)|^2}{p^{1+ 2 ( 1/2-\theta) }} +    \sum_{ \substack{ \kb_1 , \kb_2 \neq 0 \\ \kb_1+\kb_2= \kb } } \sum_{ \substack{ \lb_1 , \lb_2 \neq 0  \\ \lb_1+ \lb_2 = \lb}  }  \frac{  (c_0 \pi)^{\mathcal{K}( \kb + \lb    ) }   }{\kb_1! \kb_2!\lb_1!\lb_2!} \sum_p \frac{\max_{j} \sum_i |\alpha_{j, i }(p)|^2}{p^{1+ 2 ( 1/2-\theta) }} \\
 & \ll   \frac{  c_5  ^{\mathcal{K}( \kb + \lb ) } }{ \kb! \lb!},
   \end{align*}
for some constant $c_5>0$, where the implicit constant is independent of $\kb$ and $\lb$. 
Hence, we deduce that
\begin{align*}
 \sum_{ \substack{ \kb , \lb  \neq 0 \\  \mathcal{K}( \kb+ \lb ) \geq 6   }  } B_{\sigma}(\kb, \lb) \overline{\zb}^{\kb} \zb^{\lb} & \ll  \sum_{   \mathcal{K}( \kb+ \lb ) \geq 6    } \frac{ (c_5 || \zb ||_2 )^{\mathcal{K}( \kb + \lb ) } }{ \kb! \lb!} \ll ||\zb||_2^6
\end{align*}
for $|| \zb ||_2 \leq c_4 $. Therefore, we obtain
\be  \label{Phrhat asymp 1} 
\Phrhat ( \xb, \yb ) =   \exp \bigg( \sum_{m=2}^5  \mathcal{B}_{m, \sigma} (\zb )  + O  ( || \zb ||_2^6 ) \bigg)
 \ee
for $|| \zb ||_2 \leq c_4 $, where
$$ \B_{m, \sigma} (\zb )  :=  \sum_{ \substack{ \kb , \lb  \neq 0 \\  \mathcal{K}( \kb+ \lb ) = m    }  } B_{\sigma}(\kb, \lb)  \overline{\zb}^{\kb} \zb^{\lb}  $$ 
is a homogeneous polynomial in $z_1 , \ldots, z_J, \overline{z_1} , \ldots, \overline{z_J} $ of degree $m$.

For each $ \kb, \lb \neq 0 $ satisfying $ \mathcal{K} ( \kb + \lb ) = 3, 4, 5 $,  $ B_{\sigma}(\kb, \lb) $ is a Dirichlet series absolutely convergent for $\re( s) > 1/2-  \epsilon$ for some $\epsilon>0$. Thus, each coefficient $ B_{\sigma}(\kb, \lb) $ in such case satisfies 
$$ B_{\sigma}(\kb, \lb) = B_{1/2}(\kb, \lb) + O(1/G(T)) , $$
which proves \eqref{lem asymp Phrhat eqn 4}.
This also implies that
\begin{equation} \label{B m sigma zb bound}
 \mathcal{B}_{m, \sigma} (\zb )  = O(  || \zb ||_2^m )
 \end{equation}
for $ m = 3, 4, 5$ and for $ || \zb ||_2 \leq c_4 $. 
On the other hand, when $ m=2 $, we see that
\be \label{B2 sum}
  \mathcal{B}_{2, \sigma} (\zb)  = \sum_{j_1, j_2 \leq J}  g_{j_1, j_2 } ( \sigma )   \overline{ z_{j_1}}  z_{j_2}, 
  \ee
 where
\begin{align*}
g_{j_1, j_2 } ( \sigma ) : =  & - \pi^2  \sum_p     \E \bigg[     g_{j_1} \bigg( \frac{\X(p) }{p^{\sigma}} \bigg) \overline{  g_{j_2} \bigg( \frac{\X(p) }{p^{\sigma}} \bigg) } \bigg] \\
  = &- \pi^2   \sum_p    \sum_{k=1}^\infty   \frac{   \beta_{L_{j_1}}(p^k)  \overline{\beta_{L_{j_2}}(p^k)}     }{p^{2k\sigma}}  \\
  = &  - \pi^2   \sum_p       \frac{ \beta_{L_{j_1}}(p)  \overline{   \beta_{L_{j_2}}(p) }  }{p^{2 \sigma}}  - \pi^2   \sum_p    \sum_{k=2}^\infty   \frac{   \beta_{L_{j_1}}(p^k)  \overline{\beta_{L_{j_2}}(p^k)}     }{p^{ k }}  + O \bigg(  \frac{1}{G(T)} \bigg).
\end{align*}
The second $p$-sum on the last line is convergent by \eqref{prime sum bounds}. Therefore, it follows from \eqref{SOC2} that
\be \label{B2 coeff 3}
g_{j_1, j_2 } ( \sigma) = - \pi^2 \delta_{j_1, j_2 } \xi_{j_1} \log G(T) + C_{j_1, j_2} + O\bigg(  \frac{ \log G(T)}{G(T)} \bigg) 
\ee
for some constant $ C_{j_1, j_2 } $.
This proves \eqref{lem asymp Phrhat eqn 3}. 
By  \eqref{Phrhat asymp 1} and \eqref{B m sigma zb bound} we see that
$$ \Phrhat ( \xb, \yb ) =  e^{ \mathcal{B}_{2, \sigma} (\zb)  } \bigg( 1 +    \sum_{m=3}^5  \mathcal{B}_{m, \sigma} (\zb )  + O  ( || \zb ||_2^6 )\bigg) $$
holds for $|| \zb ||_2 \leq c_4 $. This proves \eqref{lem asymp Phrhat eqn 2} and hence completes the proof.

\end{proof}

Using Lemma \ref{lem asymp Phrhat} we establish the following result, which is the key ingredient in the proof of Theorem \ref{MainRandom}. 

\begin{lem}\label{lem asymp H_T}
For all $\ub, \vb \in \mathbb{R}^J$ we have
 \begin{align*}
H_T ( \ub, \vb )   
= & \frac{1}{  ( \log G(T))^J} e^{-\frac{1}{ \log G(T)} \sum_{j=1}^J \xi_j^{-1} ( u_j^2 + v_j^2 ) } 
P_T (\ub, \vb)   +   O \bigg(   \frac{1}{  ( \log G(T))^{J+3} }    \bigg),
\end{align*} 
where  
$$ P_T ( \ub, \vb ) = \frac{1   }{  \pi^{J} \prod_{j=1}^J \xi_j }   +  \sum_{m=2}^5   \sum_{r=0}^m   \frac{  Q_{r ,m  } ( \ub, \vb)  }{  ( \log G(T))^{(m+r)/2} } $$
is a polynomial in $\ub , \vb $ of degree $\leq 5 $ and  $ Q_{r, m  } ( \ub, \vb) $ is a homogeneous polynomial in $\ub, \vb$ of degree $r$ for $ r \leq m \leq 5 $.
  \end{lem}

\begin{proof}
Recall that the density function $H_T ( \ub, \vb) $ is the inverse Fourier transform of $ \Phrhat ( \xb, \yb ) $, so that
$$ H_T ( \ub, \vb ) = \int_{ \mathbb{R}^{2J}} e^{  - 2 \pi i ( \xb \cdot \ub + \yb \cdot \vb ) } \Phrhat(\xb, \yb) d\xb d\yb . $$
By Lemma \ref{lemma abs cont}, \eqref{lem asymp Phrhat eqn 1} and by changing the variables to polar coordinates, with $ \zb = \xb + i \yb $ as in the previous lemma, we find that
\begin{align*}
H_T ( \ub, \vb )    = &  \int_{ ||\zb ||_2 \leq c_4  } e^{  - 2 \pi i ( \xb \cdot \ub + \yb \cdot \vb ) } \Phrhat(\xb, \yb) d\xb d\yb  \\
&+   O \bigg(  \int_{  c_4 < || \zb ||_2 \leq e^{\sqrt{G(T)}} } e^{-   \xi_{\mathrm{min}} \log G(T)  || \zb||_2^2 } d\xb d \yb + \int_{ || \zb ||_2 > e^{\sqrt{G(T)}}}      \frac{d\xb d\yb }{  ( 1+ || \zb||_2)^{2J+1} }     \bigg)   \\
= &  \int_{ ||\zb ||_2 \leq c_4  } e^{  - 2 \pi i ( \xb \cdot \ub + \yb \cdot \vb ) } \Phrhat(\xb, \yb) d\xb d\yb  \\
&+   O \bigg(  \int_{  c_4 }^{ e^{\sqrt{G(T)}} } e^{-  \xi_{\mathrm{min}} \log G(T)  w^2 }  w^{2J-1}dw  + \int_{   e^{\sqrt{G(T)}}}^\infty       \frac{1}{  ( 1+ w)^{2J+1} }  w^{2J-1}    dw    \bigg)   \\
= &  \int_{ ||\zb ||_2 \leq c_4  } e^{  - 2 \pi i ( \xb \cdot \ub + \yb \cdot \vb ) } \Phrhat(\xb, \yb) d\xb d\yb  +   O \bigg(   \frac{1}{ G(T)^{ c_4^2 \xi_{\mathrm{min}} } \log G(T) }    \bigg).   
\end{align*} 
 Therefore, it follows from  \eqref{lem asymp Phrhat eqn 2} and  \eqref{lem asymp Phrhat eqn 3} that 
\begin{align*}
H_T ( \ub, \vb )   = &  \int_{ ||\zb ||_2 \leq c_4  } e^{ \mathcal{B}_{2, \sigma} (\zb)  - 2 \pi i ( \xb \cdot \ub + \yb \cdot \vb ) }      \bigg( 1 +    \sum_{m=3}^5  \mathcal{B}_{m, \sigma} (\zb )  + O  ( || \zb ||_2^6 )\bigg)     d\xb d\yb \\
&  +   O \bigg(   \frac{1}{ G(T)^{ c_4^2 \xi_{\mathrm{min}} } \log G(T) }    \bigg)    \\
= & \int_{ ||\zb ||_2 \leq c_4  } e^{\mathcal{B}_{2, \sigma} (\zb)  - 2 \pi i ( \xb \cdot \ub + \yb \cdot \vb ) }    \bigg( 1 +    \sum_{m=3}^5  \mathcal{B}_{m, \sigma} (\zb )   \bigg)    d\xb d\yb  +   O \bigg(   \frac{1}{  ( \log G(T))^{J+3} }    \bigg) .
\end{align*} 
Let 
$$ \B'_{2,\sigma}(\zb) :=  \B_{2,\sigma}(\zb) + \pi^2 \log G(T) \sum_{j=1}^J \xi_j |z_j|^2 ,$$
then by \eqref{B2 sum}  and \eqref{B2 coeff 3}, we have 
$$\B'_{2,\sigma}(\zb) = \sum_{j_1 , j_2 \leq J} \bigg( C_{j_1, j_2 } + O\bigg(  \frac{ \log G(T)}{G(T)} \bigg)\bigg)      \overline{ z_{j_1}}  z_{j_2}  .$$  
Thus, by \eqref{lem asymp Phrhat eqn 3}  we have
\begin{align*}
H_T ( \ub, \vb )   
= & \int_{ ||\zb ||_2 \leq c_4  } e^{ -  \pi^2 \log G(T) \sum_{j=1}^J  \xi_j ( x_j^2 + y_j^2 )   - 2 \pi i ( \xb \cdot \ub + \yb \cdot \vb ) }    \bigg( 1 +    \sum_{m=2}^5  \tilde{\mathcal{B}}_{m,\sigma} (\zb )   \bigg)    d\xb d\yb   \\
&+   O \bigg(   \frac{1}{  ( \log G(T))^{J+3} }    \bigg),
\end{align*} 
where each $ \tilde{\mathcal{B}}_{m,\sigma} (\zb ) $ for $ m= 2, 3, 4, 5$ is a homogeneous polynomial of degree $m$  defined to satisfy the identity 
$$  \bigg( 1 + \B'_{2, \sigma} (\zb) + \frac12 \B'_{2, \sigma} (\zb)^2   \bigg)\bigg( 1 +    \sum_{m=3}^5  \mathcal{B}_{m, \sigma} (\zb ) \bigg) = 1 +    \sum_{m=2}^5  \tilde{\mathcal{B}}_{m,\sigma} (\zb )+ O ( || \zb ||_2^6 ).$$ 
Furthermore, one has
$$ 
\tilde{\mathcal{B}}_{m,\sigma} (\zb )= \sum_{k=0}^m \left(D_{k, m} + O \left(\frac{ \log G(T)}{G(T)} \right)\right) \overline{z}^k z^{m-k},
$$
for some constants $D_{k, m}.$ Thus, the polynomials  
$$ \tilde{\mathcal{B}}_{m}(\zb):= \sum_{k=0}^m D_{k, m} \overline{z}^k z^{m-k}$$
are independent of $\sigma$ and we see that
\begin{align*}
H_T ( \ub, \vb )   
= & \int_{ ||\zb ||_2 \leq c_4  } e^{ -  \pi^2 \log G(T) \sum_{j=1}^J  \xi_j ( x_j^2 + y_j^2 )   - 2 \pi i ( \xb \cdot \ub + \yb \cdot \vb ) }    \bigg( 1 +    \sum_{m=2}^5  \tilde{\mathcal{B}}_{m}(\zb)  \bigg)    d\xb d\yb   \\
&+   O \bigg(   \frac{1}{  ( \log G(T))^{J+3} }    \bigg).
\end{align*} 
Extending the range of integration to all of $\mathbb{R}^{2J}$ and changing the $x_j$ and $y_j$ to $x_j/\sqrt{\pi^2 \log G(T)}$ and $y_j/\sqrt{\pi^2 \log G(T)}$ respectively, we obtain
\begin{align*}
H_T& ( \ub, \vb )   
=  \int_{  \mathbb{R}^{2J}  } e^{ -  \pi^2 \log G(T) \sum_{j=1}^J \xi_j ( x_j^2 + y_j^2 )   - 2 \pi i ( \xb \cdot \ub + \yb \cdot \vb ) }    \bigg( 1 +    \sum_{m=2}^5  \tilde{\mathcal{B}}_{m}(\zb)  \bigg)    d\xb d\yb   \\
& \qquad \qquad +   O( ( \log G(T))^{-J-3} ) \\
= &  \int_{  \mathbb{R}^{2J}  } e^{ -   \sum_{j=1}^J \xi_j ( x_j^2 + y_j^2 )   -  2  \frac{    i}{ \sqrt{ \log G(T)}} ( \xb \cdot \ub + \yb \cdot \vb ) }    \bigg( 1 +    \sum_{m=2}^5  \frac{ \tilde{\mathcal{B}}_{m}(\zb) }{ \pi^m  ( \log G(T))^{m/2} } \bigg)  \frac{   d\xb d\yb }{ \pi^{2J} ( \log G(T))^J}  \\
&\quad +   O( ( \log G(T))^{-J-3} )  .
\end{align*} 
The above exponent is 
  $$ -\frac{1}{ \log G(T)} \sum_{j=1}^J  \xi_j^{-1} ( u_j^2 + v_j^2 ) -  \sum_{j=1}^J \xi_j \bigg(     \bigg(   x_j  +   \frac{i u_j }{ \xi_j  \sqrt{ \log G(T)}}  \bigg)^2 +  \bigg(   y_j  +   \frac{i v_j }{ \xi_j  \sqrt{ \log G(T)}}   \bigg)^2  \bigg) .$$
Hence, making the change of variables 
$$
\tilde{\zb} = ( \tilde{x}_1 + i \tilde{y}_1, \ldots , \tilde{x}_J + \tilde{y}_J ), $$
where
$$ \tilde{x}_j =    x_j  -   \frac{i u_j }{ \xi_j  \sqrt{ \log G(T)}},   \quad  \tilde{y}_j =   y_j  -   \frac{i v_j }{ \xi_j  \sqrt{ \log G(T)}}   $$   
for every $ j =1  , \ldots, J$,  we see that
\begin{align*}
H_T ( \ub, \vb )   
= & \frac{1}{  ( \log G(T))^J} e^{-\frac{1}{ \log G(T)} \sum_{j=1}^J \xi_j^{-1} ( u_j^2 + v_j^2 ) } 
P_T (\ub, \vb)   +   O \bigg(   \frac{1}{  ( \log G(T))^{J+3} }    \bigg),
\end{align*} 
where $P_T ( \ub, \vb )$ is a polynomial in $\ub$ and $\vb$ defined by
$$ P_T ( \ub, \vb ) =  \frac{1}{ \pi^{2J}} \int_{  \mathbb{R}^{2J}  } e^{ -  \sum_{j=1}^J \xi_j ( x_j^2 + y_j^2 ) }  \bigg( 1 +    \sum_{m=2}^5    \frac{\tilde{\mathcal{B}}_m   (   \tilde{\zb}  )  }{ \pi^m  ( \log G(T))^{m/2} }   \bigg)    d\xb d\yb . $$
By expanding the polynomial $\pi^{-m} \tilde{\mathcal{B}}_m (  \tilde{\zb}) $, we have
$$ \pi^{-m-2J}\tilde{\mathcal{B}}_m (   \tilde{\zb}) = \sum_{k=0}^m   \frac{  Q_{ k,m }  ( \xb, \yb , \ub, \vb ) }{   ( \log G(T))^{k/2}} ,$$
where $Q_{k, m  } ( \xb, \yb, \ub, \vb ) $ is a homogeneous polynomial in $\xb, \yb,\ub , \vb$ of degree $m$ and each term has $k$ factors in $\ub, \vb$. Therefore, we have
$$ P_T ( \ub, \vb ) = \frac{1   }{  \pi^{J} \prod_{j=1}^J \xi_j }   +  \sum_{m=2}^5   \sum_{k=0}^m   \frac{  Q_{k ,m  } ( \ub, \vb)  }{  ( \log G(T))^{(m+k)/2} } , $$
where 
$$ Q_{k, m } ( \ub, \vb):= \int_{  \mathbb{R}^{2J}  } e^{ -  \sum_{j=1}^J  \xi_j ( x_j^2 + y_j^2 ) }        Q_{k,m   }  ( \xb, \yb, \ub, \vb )        d\xb d\yb $$
is a homogeneous polynomial in $\ub, \vb$ of degree $k$ for $k \leq m \leq 5 $.

\end{proof}


\subsection{Proof of Theorem \ref{MainRandom}}
 
Let $\sigma=1/2+ G(T)$ and $\sigma_i=  1/2+ G_i(T)$ for $i=1, 2$. Then, recall that 
$$ \M(\sigma)=\E  \left(\log \bigg|\sum_{j=1}^J b_j L_j(\sigma, \X)\bigg|\right)= \int_{\mathbb{R}^{2J} }     \log \bigg| \sum_{j =1}^J b_j e^{u_j + i v_j }  \bigg|  H_T ( \u, \v)  d\u d\v.$$
Let $M_1 :=  \sqrt{ \eta  \log G(T) \log \log G(T)} $, where $\eta>0$ is a suitably large constant that depends on $J$ and the $\xi_j$.    
By the Cauchy-Schwarz inequality, Lemma \ref{lemma H_T bound 1} with $\lambda = (30 J\max_{j\leq J} \xi_j)^{-1}$and equation \eqref{eqn Lemma 2.3 Le4} we see that 
\begin{align*} 
\bigg| & \int_{   \mathbb{R}^{2J} \setminus  [-M_1, M_1]^{2J}  }      \log \bigg| \sum_{j =1}^J b_j e^{u_j + i v_j }  \bigg|  H_T ( \u, \v)  d\u d\v  \bigg| \\
& \leq 
\bigg( \int_{    \mathbb{R}^{2J} \setminus  [-M_1, M_1]^{2J}  }   \bigg| \log \bigg| \sum_{j =1}^J b_j e^{u_j + i v_j }  \bigg| \bigg|^{2 }  H_T ( \u, \v)  d\u d\v \bigg)^{ \frac12} \\
& \qquad \cdot
\bigg( \int_{     \mathbb{R}^{2J} \setminus  [-M_1, M_1]^{2J}    }    H_T ( \u, \v)  d\u d\v \bigg)^{  \frac12 }  \\
& \ll   ( \log G(T))^{   \frac{J+1}{2}   }        \bigg(  \frac{ ( \log G(T))^{J + \frac12 } }{ M_1 } e^{-   \frac{ \lambda M_1^2}{\log G(T)} } \bigg)^{\frac12 }    \ll \frac{ 1}{ ( \log G(T) )^{3} }. 
\end{align*}
This implies
$$
    \M(\sigma)     =    \int_{ [-M_1 , M_1  ]^{2J}}   \log \bigg| \sum_{j =1}^J b_j e^{u_j + i v_j }  \bigg|  H_T ( \u, \v)  d\u d \v  
     + O \bigg( \frac{ 1}{ ( \log G( T))^3}   \bigg).
$$
We now use the asymptotic formula for $H_T(\u, \v)$ in Lemma \ref{lem asymp H_T} to obtain
\begin{align*}
      \M(\sigma)  =  &    \frac{1}{  ( \log G(T))^J}  \int_{ [-M_1 , M_1  ]^{2J} }   \log \bigg| \sum_{j =1}^J b_j e^{u_j + i v_j }  \bigg|   e^{-\frac{1}{ \log G(T)} \sum_{j=1}^J \xi_j^{-1} ( u_j^2 + v_j^2 ) } 
P_T (\ub, \vb) d\ub d\vb  \\
&  +   O \bigg(   \frac{1}{  ( \log G(T))^{J+3} }  \int_{ [-M_1 , M_1  ]^{2J}}  \bigg|  \log \bigg| \sum_{j =1}^J b_j e^{u_j + i v_j }  \bigg|   \bigg|   d\ub d \vb + \frac{ 1}{ ( \log G( T))^{3}}   \bigg) .
\end{align*}
By the Cauchy-Schwarz inequality and \eqref{eqn Lemma 2.3 Le4} we have
\begin{align*}
  \bigg( \int_{ [-M_1 , M_1  ]^{2J}} & \bigg|  \log \bigg| \sum_{j =1}^J b_j e^{u_j + i v_j }  \bigg|   \bigg|   d\ub d \vb \bigg)^2    
  \leq (2 M_1)^{2J} \int_{ [-M_1 , M_1]^{2J}}   \bigg|  \log \bigg| \sum_{j =1}^J b_j e^{u_j + i v_j }  \bigg|   \bigg|^2    d\ub d \vb  \\
  &\ll  M_1^{2J} \int_{ \mathbb{R}^{2J}}   \bigg|  \log \bigg| \sum_{j =1}^J b_j e^{u_j + i v_j }  \bigg|   \bigg|^2  e^{ -   \frac{1}{M_1^2 } \sum_{j=1}^J ( u_j ^2 + v_j^2 ) }  d\ub d \vb   \ll M_1^{2(2J+1)} .
  \end{align*}
Thus we  have 
\begin{align*}
\M(\sigma)   = &    \frac{1}{  ( \log G(T))^J}  \int_{ [-M_1 , M_1  ]^{2J} }   \log \bigg| \sum_{j =1}^J b_j e^{u_j + i v_j }  \bigg|   e^{-\frac{1}{ \log G(T)} \sum_{j=1}^J \xi_j^{-1} ( u_j^2 + v_j^2 ) } 
P_T (\ub, \vb) d\ub d\vb  \\
&  +   O \bigg(   \frac{( \log \log G(T))^{J+(1/2) }}{  ( \log G(T))^{5/2} } \bigg). 
\end{align*}
We now use Lemma \ref{lem asymp H_T}, which gives
\begin{align*}
    \M(\sigma) &=  \frac{1   }{    \prod_{j \leq J } \xi_j   ( \pi  \log G(T))^J}  \int_{ [-M_1 , M_1  ]^{2J} }   \log \Big| \sum_{j =1}^J b_j e^{u_j + i v_j }  \Big|   e^{-\frac{1}{ \log G(T)} \sum_{j=1}^J \xi_j^{-1} ( u_j^2 + v_j^2 ) } 
  d\ub d\vb  \\
&+        \sum_{m=2}^5   \sum_{r=0}^m    \int_{ [-M_1 , M_1  ]^{2J} }   \log \Big| \sum_{j =1}^J b_j e^{u_j + i v_j }  \Big|   e^{-\frac{1}{ \log G(T)} \sum_{j=1}^J \xi_j^{-1} ( u_j^2 + v_j^2 ) } 
 \frac{   Q_{r ,m  } ( \ub, \vb)   d\ub d\vb  }{  ( \log G(T))^{J+ (m+r)/2} }  \\
&  +   O \bigg(   \frac{( \log \log G(T))^{J+(1/2) }}{  ( \log G(T))^{5/2} } \bigg) . 
    \end{align*} 
 By the Cauchy-Schwarz inequality and \eqref{eqn Lemma 2.3 Le4}, we can replace $[-M_1, M_1]^{2J}$ in both integrals by $ \mathbb{R}^{2J}$, at the cost of an error term of size $\ll 1/(\log G(T))^{3}$ if $\eta$ is suitably large. Let 
$$ Q_{r, m} ( \ub, \vb) = \sum_{ \substack{ \kb, \lb \\  \mathcal{K}(\kb+\lb) = r}  } q_{r, m, \kb, \lb } \ub^{\kb} \vb^{\lb}.$$  
Since $ Q_{r ,m  } ( \ub, \vb)$ is a homogeneous polynomial of degree $r$, by changing the variables $ \ub, \vb $ to  $ \sqrt{ \log G(T)}\ub, \sqrt{\log G(T)}\vb $, we obtain
\begin{align*}
    \M(\sigma) 
        & =  \frac{1   }{  \pi^{J} \prod_{j=1}^J \xi_j    }  \int_{ \mathbb{R}^{2J} }   \log \Big| \sum_{j =1}^J b_j e^{(u_j + i v_j )\sqrt{ \log G(T)} }  \Big|   e^{-   \sum_{j=1}^J \xi_j^{-1} ( u_j^2 + v_j^2 ) } 
  d\ub d\vb  \\
& +       \sum_{m=2}^5   \sum_{r=0}^m   \int_{ \mathbb{R}^{2J} }   \log \Big| \sum_{j =1}^J b_j e^{(u_j + i v_j) \sqrt{ \log G(T)} }  \Big|   e^{-  \sum_{j=1}^J \xi_j^{-1} ( u_j^2 + v_j^2 ) } 
    \frac{  Q_{r ,m  } ( \ub, \vb)  d\ub d\vb    }{  ( \log G(T))^{   m/2  } }     \\
&  +   O \bigg(   \frac{( \log \log G(T))^{J+(1/2) }}{  ( \log G(T))^{5/2} } \bigg)   \\
        =  &  \frac{1   }{  \pi^{J} \prod_{j=1}^J  \xi_j    }  I   (0,0, \sigma )  +       \sum_{m=2}^5   \sum_{r=0}^m   \frac{  1  }{  ( \log G(T))^{   m/2  } }   \sum_{ \substack{ \kb, \lb \\  \mathcal{K}(\kb+\lb) = r}  } q_{r, m, \kb, \lb }   I   (\kb, \lb, \sigma) \\
&  +   O \bigg(   \frac{( \log \log G(T))^{J+(1/2) }}{  ( \log G(T))^{5/2} } \bigg) ,
    \end{align*} 
where
$$ I  (\kb, \lb, \sigma ) :=  \int_{ \mathbb{R}^{2J} }   \log \Big| \sum_{j =1}^J b_j e^{(u_j + i v_j) \sqrt{ \log G(T)} }  \Big|   e^{-  \sum_{j=1}^J \xi_j^{-1} ( u_j^2 + v_j^2 ) } 
  \ub^\kb \vb^\lb   d\ub d\vb.
  $$
The above estimation  also holds for $\M(\sigma_1)$ and $\M(\sigma_2)$. Therefore, we deduce that 
\be \label{Exp estimation eqn 6}\begin{split}
\M(\sigma)-\M(\sigma_i)
& = \frac{1   }{  \pi^{J} \prod_{j=1}^J \xi_j    }  \big(I (0,0, \sigma )- I (0,0, \sigma_i )\big) \\
& +  \sum_{m=2}^5   \sum_{r=0}^m   \frac{  1  }{  ( \log G(T))^{   m/2  } }   \sum_{ \substack{ \kb, \lb \\  \mathcal{K}(\kb+\lb) = r}  } q_{r, m, \kb, \lb }  \big( I   (\kb, \lb, \sigma)-I   (\kb, \lb, \sigma_i)\big) \\
&  +   O \bigg(   \frac{( \log \log G(T))^{J+(1/2) }}{  ( \log G(T))^{5/2} } \bigg) .
\end{split}
\ee

This integral $I (\kb, \lb, \sigma )$  was estimated  in \cite{Le4} when $\kb, \lb$ are fixed and $G(T)$ is a power of $\log T$. Let $\mathcal{R}_n := \{ \ub \in \mathbb{R}^J : u_n = \max \{ u_1 , \ldots, u_J \} \} $, then $ I  (\kb, \lb, \sigma ) $ equals 
\begin{align*}
   &  \sum_{n =1}^J   \int_{ \mathbb{R}^J}  \int_{ \mathcal{R}_n }   \log \Big| \sum_{j =1}^J b_j e^{(u_j + i v_j) \sqrt{ \log G(T)} }  \Big|   e^{-  \sum_{j=1}^J \xi_j^{-1} ( u_j^2 + v_j^2 ) } 
  \ub^\kb \vb^\lb   d\ub d\vb   \\
 =  &  \sum_{n =1}^J   \int_{ \mathbb{R}^J}  \int_{ \mathcal{R}_n }   \log \Big|   b_n e^{(u_n+ i v_n) \sqrt{ \log G(T)} }  \Big|   e^{-  \sum_{j=1}^J \xi_j^{-1} ( u_j^2 + v_j^2 ) } 
  \ub^\kb \vb^\lb   d\ub d\vb + \sum_{n=1}^J \mathcal{E}_n  ( \kb, \lb , \sigma)    ,
  \end{align*}
  where $\mathcal{E}_n  ( \kb, \lb , \sigma)$ is defined by 
 \be\label{def EnkblbGT}  
     \int_{ \mathbb{R}^{ J} }    \int_{\mathcal{R}_n }   \log \Big| 1 + \sum_{j\neq n   }  \frac{b_j}{b_n}  e^{((u_j -u_n ) + i (v_j  -v_n) ) \sqrt{ \log G(T)} }  \Big|   e^{-  \sum_{j=1}^J \xi_j^{-1} ( u_j^2 + v_j^2 ) } 
  \ub^\kb \vb^\lb   d\ub d\vb   .
  \ee
  Moreover, define
    \begin{align*}
  d_{\lb} & := \int_{\mathbb{R}^J} e^{  - \sum_{j=1}^J v_j^2/\xi_j } \vb^{\lb} d\vb , \\
   D_1(\kb, \lb)&:= d_{\lb}\sum_{n=1}^J  \int_{\mathcal{R}_n }     e^{ - \sum_{j=1}^J u_j^2 / \xi_j }   u_n \ub^\kb d\ub, \\
  D_2 (\kb, \lb)&:=d_{\lb}\sum_{n=1}^J  \log |b_n|  \int_{\mathcal{R}_n } e^{ - \sum_{j=1}^J u_j^2 / \xi_j } \ub^\kb d\ub,
 \end{align*} 
  then we find that
   \begin{equation}  \label{ITkl asymp 1}
 I  (\kb, \lb, \sigma ) =   \sqrt{\log G(T)}  \cdot D_1(\kb, \lb) +   D_2 (\kb, \lb)  + \sum_{n=1}^J \mathcal{E}_n  ( \kb, \lb , \sigma) .
\end{equation}
Note that 
$$d_{\lb}  =  \begin{cases} 0, & \text{ if } \ell_j \text{ is odd for some }  j,\\ \prod_{j=1}^J  \big( \xi_j^{( \ell_j+1)/2 } \Gamma ( ( \ell_j + 1)/2 )  \big), & \text{ if } \ell_j \text{ is even for all } j .
\end{cases}
$$
By changing $(\log T)^\theta$ to $G(T)$ in the proof of \cite[Proposition 2.4]{Le4}, it follows that 
 $$
  \mathcal{E}_n  ( \kb, \lb, \sigma)=  O \bigg(  \frac{1}{ ( \log G(T))^{1/4}} \bigg)$$
if $ G(T) \gg 1$, but this bound is not sufficient for our purpose. 
  Instead, by \eqref{ITkl asymp 1}, we estimate the difference 
   \begin{align*}
   I   (\kb, \lb, \sigma) & - I( \kb, \lb, \sigma_i)  \\
   = &  ( \sqrt{ \log G(T)} - \sqrt{\log G_i (T)} )D_1(\kb, \lb)   +  \sum_{n=1}^J ( \mathcal{E}_n  ( \kb, \lb, \sigma) - \mathcal{E}_n ( \kb, \lb, \sigma_i) ) \\
   = &    \frac{(-1)^i D_1(\kb, \lb)}{ 2 ( \log G(T))^{3/2}}   +  \sum_{n=1}^J ( \mathcal{E}_n  ( \kb, \lb, \sigma) - \mathcal{E}_n ( \kb, \lb, \sigma_i) )  +  O \bigg( \frac{1}{ ( \log G(T))^{5/2}} \bigg)  .
   \end{align*}
      By a symmetry, we only estimate   $ \mathcal{E}_1  ( \kb, \lb, \sigma) - \mathcal{E}_1 ( \kb, \lb, \sigma_i) $. By \eqref{def EnkblbGT} and a simple substitution, we see that
 \begin{align*}
  \mathcal{E}_1 & ( \kb, \lb ,\sigma ) - \mathcal{E}_1 ( \kb, \lb , \sigma_i )  \\
 =  &   \int_{ \mathbb{R}^{ J} }    \int_{\mathcal{R}_1 }   \log \Big| 1 + \sum_{j\neq 1   }  \frac{b_j}{b_1}  e^{((u_j -u_1 ) + i (v_j  -v_1) ) \sqrt{ \log G(T)} }  \Big|   e^{-  \sum_{j=1}^J \xi_j^{-1} ( u_j^2 + v_j^2 ) } 
  \ub^\kb \vb^\lb   d\ub d\vb   \\
     & -  \int_{ \mathbb{R}^{ J} }    \int_{\mathcal{R}_1 }   \log \Big| 1 + \sum_{j\neq 1   }  \frac{b_j}{b_1}  e^{((u_j -u_1 ) + i (v_j  -v_1) ) \sqrt{ \log G_i (T)} }  \Big|   e^{-  \sum_{j=1}^J \xi_j^{-1} ( u_j^2 + v_j^2 ) } 
  \ub^\kb \vb^\lb   d\ub d\vb  \\
   =  &   \int_{ \mathbb{R}^{ J} }    \int_{\mathcal{R}_1 }   \log \Big| 1 + \sum_{j\neq 1   }  \frac{b_j}{b_1}  e^{((u_j -u_1 ) + i (v_j  -v_1) ) \sqrt{ \log G(T)} }  \Big|   \\
   & \cdot \bigg(  e^{-  \sum_{j=1}^J \xi_j^{-1} ( u_j^2 + v_j^2 ) } -   e^{- \frac{ \log G(T)}{\log G_{i}(T)}   \sum_{j=1}^J \xi_j^{-1} ( u_j^2 + v_j^2 ) }   \bigg( \frac{ \log G(T)}{\log G_{i}(T)} \bigg)^{J+ \mathcal{K}(\kb+\lb)/2} \bigg)   \ub^\kb \vb^\lb   d\ub d\vb. 
    \end{align*}
  Since   $  \frac{ \log G(T)}{ \log G_{i}(T)} = 1 + O \big( \frac{1}{ ( \log G(T))^2}\big)$,  by adapting the proof of Lemma 2.5 in \cite{Le4}, we find that the above is
    \begin{align*}
    \ll &  \frac{1}{ ( \log G(T))^2 }    \int_{ \mathbb{R}^{ J} }    \int_{\mathcal{R}_1 }   \bigg| \log \bigg| 1 + \sum_{j\neq 1   }  \frac{b_j}{b_1}  e^{((u_j -u_1 ) + i (v_j  -v_1) ) \sqrt{ \log G(T)} }  \bigg|  \bigg|  \\
    & \cdot e^{-    \big( 1 + O \big( \frac{1}{( \log G(T))^{2}} \big)\big) \sum_{j=1}^J  \xi_j^{-1} ( u_j^2 + v_j^2 ) }   \bigg(  \sum_{j=1}^J  (u_j^2 + v_j^2 )  +1  \bigg)  \ub^\kb \vb^\lb   d\ub d\vb \\
  \ll & \frac{1}{ ( \log G(T))^{9/4}}.   
    \end{align*}
Thus, we deduce that 
$$
   I   (\kb, \lb, \sigma) - I( \kb, \lb, \sigma_i) \\
   =  \frac{(-1)^i D_1 (\kb, \lb)}{ 2 ( \log G(T))^{3/2}}      +    O \bigg( \frac{1}{ ( \log G(T))^{9/4}} \bigg) .
$$
Inserting this estimate in \eqref{Exp estimation eqn 6} gives 
$$ 
\M(\sigma)-\M(\sigma_i)=  \frac{(-1)^i D_1 (0,0)}{ 2  \pi^{J} (\prod_{j=1}^J \xi_j)  ( \log G(T))^{3/2}} +    O \bigg( \frac{1}{ ( \log G(T))^{9/4}} \bigg) ,
$$
where 
$$ D_1(0, 0)=  \pi^{J/2}\prod_{j=1}^J \sqrt{\xi_j}  \sum_{n=1}^J  \int_{\mathcal{R}_n }     e^{ - \sum_{j=1}^J u_j^2 / \xi_j }   u_n d\ub 
$$
by $\Gamma(1/2)=\sqrt{\pi}.$ This completes the proof.

\section{acknowledgment}

We thank the anonymous referees for carefully reading the paper, and for their numerous comments and suggestions.

\end{document}